\documentclass[11pt,a4paper]{amsart}
\usepackage{amsmath}
\usepackage{amsfonts}
\usepackage{amssymb}
\usepackage{graphicx}
\usepackage[utf8]{inputenc}
\usepackage{pdflscape}
\usepackage{float}
\usepackage{mathtools}
\usepackage{appendix}
\usepackage{comment}
\usepackage[pagebackref=true, bookmarksopen=true, colorlinks=true, linkcolor=red, citecolor=blue]{hyperref}
\usepackage{epsfig}
\usepackage{tikz}
\usepackage{mathrsfs}

\numberwithin{equation}{section}
\makeatletter
\@namedef{subjclassname@2010}{\textup{2020} Mathematics Subject Classification}
\makeatother

\newtheorem{theorem}{Theorem}[section]
\newtheorem{definition}{Definition}[section]
\newtheorem{lemma}{Lemma}[section]
\newtheorem{proposition}{Proposition}[section]
\newtheorem{corollary}{Corollary}[section]
\newtheorem{remark}{Remark}[section]

\providecommand{\norm}[1]{\lvert\lvert#1\rvert\rvert}

\begin{document}
\title[Hierarchical Controllability ]{HIERARCHICAL LOCAL NULL CONTROLLABILITY
FOR A NON-DEGENERATE QUASI-LINEAR PARABOLIC SYSTEM 
}

\author[Cristian Loli]{Cristian Loli}  
	\address{CRISTIAN LOLI \hfill\break
	Universidade Federal Fluminense, Instituto de Matemática e Estatistica, Brazil.}
	\email{\url{cristianloli@id.uff.br}}
\author[George Bautista]{George J. Bautista} 
	\address{GEORGE J. BAUTISTA  \hfill\break
	Universidade Federal Fluminense, Instituto de Matemática e Estatistica, Brazil.}
	\email{\url{geojbs25@gmail.com}}   
\author[Juan Limaco]{Juan Limaco}  
	\address{JUAN LIMACO  \hfill\break
	Universidade Federal Fluminense, Instituto de Matemática e Estatistica, Brazil.}
	\email{\url{jlimaco@id.uff.br}}
\author[Rafael Lobosco]{Rafael Lobosco} 
	\address{RAFAEL LOBOSCO \hfill\break
	Instituto Federal do Rio de Janeiro, Brazil.}
	\email{\url{rafael.lobosco@ifrj.edu.br}}    
\author[Dany Nina]{Dany Nina-Huaman} 
    \address{DANY NINA-HUAMAN \hfill\break Universidad Nacional Agraria La  Molina, Peru.}
    \email{\url{danynina3003@gmail.com}}
\author[Luis Yapu]{Luis P. Yapu} 
	\address{LUIS P. YAPU \hfill\break
	Universidade Federal Fluminense, Instituto de Matemática e Estatistica, Brazil\\ \hfill\break
    Chair of Dynamics, Control, Machine Learning and Numerics, Friedrich-Alexander-Universit\"at Erlangen-N\"urnberg, Germany.}
	\email{\url{luis.yapu@gmail.com}}



\numberwithin{equation}{section}

\begin{abstract}
    In this manuscript, we are concerned with  the hierarchical control of a  non-degenerate quasi-linear parabolic system, posed on a bounded interval. We use the Stackelberg–Nash strategy with a single control called  \textit{leader} and two controls called \textit{followers}.  First, we study the controllability for the linear system, where we will demonstrate a Carleman inequality associated with the adjoint problem. Then, using the results obtained for the linear case and applying Liusternik's Inverse Function Theorem, we obtain local null controllability for the quasi-linear problem.
 
\end{abstract}

\date{}

\maketitle

\section{ \bf Introduction } 
Let $I=(0,L)$  be an open and bounded interval of $\mathbb{R}$ and $\omega$,  $\omega_i$ $i=1, 2$, be a non empty open subsets of
$I$. For $T>0$, we denote by $Q=I\times(0,T)$. We consider  the following 
quasi-linear parabolic problem 
\begin{align}
	\begin{cases}	\label{Eq1}	
		y_{1,t}-
        (a(y_1)y_{1,x})_x + F_1(y_1)   = h\chi_\omega+v^1\chi_{\omega_1}+v^2\chi_{\omega_2}, &\text{ in } Q,\\  
		y_{2,t}-
        (a(y_2)y_{2,x})_x + F_2(y_2) + c y_1 = 0, &\text{ in } Q,\\  
        y_1(0,t) = y_1(L,t) =0,\,  y_2(0,t) = y_2(L,t) = 0, &\text{ in } (0,T), \\
        y_1(x,0) = f(x), \,  y_2(x,0) = g(x),  &\text{ in } I.
	\end{cases}
\end{align}Here, $y_1$ and $y_2$ are the states, the open subset $\omega \subset I$ is \emph{the main control domain}, the open subsets $\omega_1, \omega_2 \subset I$ are \emph{the secondary control domain}; $h$ and  $v^i$, $i=1,2$, are different
control functions whose actions are carried out in the system through the  subsets $\omega$  and $\omega_i$, respectively, where $h$ is the leader and $v^i$ are the followers,  while $\chi_{\omega}$ and $\chi_{\omega_i}$ denote the characteristic functions of $\omega$ and $\omega_i$, respectively, and  $(f, g)$ is the  initial data.

The controllability of partial differential equations has become one of the most active research areas in control theory due to its broad range of applications in engineering, physics, biology, economics, and environmental sciences. In many practical situations, the dynamics of the system are influenced by several decision-makers whose objectives may differ or even conflict. Such scenarios naturally motivate the development of hierarchical control strategies capable of coordinating multiple agents while preserving individual optimization goals.

Among the various frameworks proposed for this purpose, the Stackelberg–Nash strategy occupies a prominent position. Inspired by the pioneering ideas of Stackelberg \cite{Stack}, Nash \cite{Nash}, and Pareto \cite{Pareto}, this approach combines hierarchical decision-making with noncooperative optimization. Within this setting, a leader determines a global strategy, whereas several followers react optimally according to their own performance criteria. The followers are expected to reach a Nash equilibrium, and the leader subsequently exploits this equilibrium structure to achieve a prescribed global objective.

The mathematical analysis of hierarchical control problems governed by partial differential equations has attracted considerable attention over the last decades. Foundational contributions were established by Lions \cite{Lions, Lions-Stack}, who introduced Stackelberg-type strategies in distributed parameter systems and investigated hierarchical control problems for wave and parabolic equations. These pioneering works demonstrated that game-theoretic concepts can be successfully integrated into the framework of PDE controllability.

Subsequently, Díaz \cite{Diaz2002} developed a rigorous theoretical framework for Stackelberg–Nash games, establishing existence and uniqueness results for hierarchical equilibria and providing a precise mathematical characterization of the interactions between leaders and followers. In parallel, in \cite{Glowinski}, the authors studied Nash equilibrium associated with parabolic equations and proposed computational methodologies that contributed significantly to the practical implementation of hierarchical control strategies.

The development of controllability techniques for PDEs has been strongly supported by the theory of Carleman inequalities. Important contributions in this direction can be found in \cite{Cannarsa-2016}, \cite{Li-Clark-EFCara}, \cite{Fa-Zuazua}, \cite{EFC-Li2003}, \cite{EFC-Teresa-10}, \cite{Cara-Limaco-21}, \cite{EFC-Nina-Nu-Vi}, \cite{LiuZhang} and the references therein. 
The monographs by Fursikov and Imanuvilov \cite{FurImanu} and by Imanuvilov and Yamamoto \cite{Imanu-Yama} constitute fundamental references regarding observability inequalities and controllability methods for parabolic equations. These techniques have become the cornerstone for proving controllability properties of both linear and nonlinear distributed systems.

More recently, hierarchical control theory has experienced substantial progress. Limaco et al., in \cite{Li-Clark}, investigated Stackelberg–Nash strategies for parabolic equations in non-cylindrical domains with moving boundaries, showing that hierarchical controllability can be achieved even in geometrically evolving environments. Later, in \cite{Gui-Mar-Ro}, the authors extended hierarchical methodologies to the Stokes system, demonstrating the applicability of Stackelberg–Nash concepts to fluid mechanics models.

A particularly important advance was achieved by Araruna et al. \cite{Ara-EFC-San2015, Ara-EFC-Gue}, who obtained some of the first hierarchical controllability results in the context of exact controllability. Their work represented a significant step forward in the transition from approximate hierarchical strategies to rigorous controllability frameworks.

The literature has also witnessed a growing interest in Nash equilibrium problems and hierarchical controllability for parabolic systems with multiple objectives. In this direction, in \cite{Hernandez-Teresa-16}, it is studied multi-objective optimization problems governed by distributed systems, while, in \cite{Kere-Mer-Mop},  the authors analyzed Nash equilibrium strategies for nonlinear parabolic equations. Related developments were presented by Landry et al.(see  \cite{Landry2023}), in the context of distributed control problems involving several competing agents.

Particularly relevant to the present work are the contributions developed by Limaco and his collaborators. For example, in \cite{EFC-Li2003, Cara-Limaco-21}, obtained several controllability results for coupled parabolic systems using Carleman estimates and inverse mapping arguments. More recently, in \cite{GYL-Carleman-2025}, the authors investigated Stackelberg–Nash hierarchical strategies for parabolic systems, establishing null controllability results under multi-objective settings. Likewise, in \cite{Danny2022, Danny2021}, studied hierarchical Nash equilibria for nonlinear parabolic equations and developed analytical techniques combining optimality systems, convexity arguments and Carleman estimates. Similar approaches were further explored in \cite{Nunez-Li2023}, for the context of coupled distributed parameter systems.

These works reveal a common methodological structure: the characterization of Nash equilibria through suitable optimality systems, the derivation of observability inequalities by means of Carleman estimates, and the extension of controllability results to nonlinear settings through inverse mapping theorems. However, despite the significant progress achieved in recent years, most available results concern linear or semilinear systems. The study of hierarchical controllability for quasi-linear parabolic systems remains comparatively unexplored, mainly because nonlinearities appear in the principal part of the differential operators, creating substantial analytical difficulties.

Motivated by this situation, the present paper investigates a hierarchical null controllability problem for  coupled quasi-linear parabolic system  \eqref{Eq1}. The model consists of two coupled state equations whose diffusion coefficients depend on the state variables. The control mechanism includes a leader control acting on a region $\omega$ and two follower controls acting on  subdomains $\omega_1$ and $\omega_2$. For each fixed leader strategy, the followers minimize independent cost functionals and are required to reach a Nash equilibrium.

The main objective is to determine a leader control that, together with the associated Nash equilibrium of the followers, drives both state variables to zero at a prescribed final time. To accomplish this task, we first characterize the Nash equilibrium and establish convexity properties of the corresponding cost functionals. We then analyze the linearized optimality system and derive a suitable Carleman inequality for the associated adjoint problem. These estimates lead to the controllability of the linear hierarchical system. Finally, Lyusternik’s Inverse Function Theorem is employed to transfer the controllability properties of the linearized problem to the original quasi-linear model.

The remainder of the paper is organized as follows. Section $2$ introduces some notation and states the main result. Section $3$ is devoted to the characterization of the Nash equilibrium and the convexity analysis of the cost functionals. Section $4$ presents the linearized problem, Carleman estimates and the controllability result for the linear system. Finally, Section $5$ proves the local hierarchical null controllability of the original quasi-linear system through an application of Lyusternik’s Inverse Function Theorem.

\section{ \bf Notation and statement of  main results}\label{sec1}

 We first introduce a few considerations. Let us  suppose that $\omega,\omega_1,\omega_2\subset I$ be nonempty open subsets such that
\[
\omega_i\cap\omega=\varnothing,
\qquad i=1,2.
\]
We assume that the initial data satisfy
\[
(f,g)\in [L^2(I)]^2.
\]
We impose the following assumptions on the coefficients and nonlinearities.
The diffusion coefficient satisfies
    \begin{equation} \label{condiciones_a4}
    \begin{cases}
        a\in C^3(I),\quad 0 < M_0 \leq a(r) \leq M_1,\\
        |a^{(j)}(r)|\leq M,\quad j=1,2,3,
    \end{cases}
    \end{equation}
    for suitable positive constants $M_0,M_1,M$. Moreover,
    \begin{equation}
        \begin{cases} \label{condiciones_a4f}
 		 F_i\in C^2(\mathbb{R}), \quad F_i(0)=0,\\
             |{F_{i}^{(j)}(r)}|\leq M, \text{ for all } r \in \mathbb{R}, \, i=1,2,\quad j=1,2.
 	      \end{cases}
    \end{equation}
   Finally, the coupling coefficient $c\in L^\infty(Q)$ is assumed to satisfy
   \begin{equation}\label{condiciones_a4c(x,t)}
    0<c_0<c(x,t), \text{ in } (\omega_d\cap\omega)\times(0,T),
    \end{equation}
    where $c_0>0$ is a suitable constant and  the set $\omega_d$ is defined in Section 4.

 Under these assumptions, we investigate the hierarchical null controllability
of system \eqref{Eq1} through a Stackelberg--Nash strategy. More precisely, the
control $h$, acting on $\omega$, is regarded as the \emph{leader}, whereas
$v^1$ and $v^2$, acting respectively on $\omega_1$ and $\omega_2$, are the
\emph{followers}.

The objective is to determine a leader control $h$ such that, when the
followers respond according to a Nash equilibrium associated with their
individual cost functionals, the resulting state satisfies
\[
y_1(\cdot,T)=y_2(\cdot,T)=0
\qquad
\text{in }I.
\]
      
\subsection{Nash equilibrium}
    Let $\omega_{1,d},\omega_{2,d}\subset I$ be nonempty open subsets, referred
to as the \emph{observation domains} of the followers. For $i=1,2$, let
\[
y_{i,d}^j\in L^2(\omega_{i,d}\times(0,T)),
\qquad
j=1,2,
\]
denote the desired trajectories associated with follower $i$.

For a fixed leader control
\[
h\in L^2(\omega\times(0,T)),
\]
we associate with each follower $i$ the cost functional
        \begin{equation} \label{energy}
            \begin{cases}
                J_1(h , v^1, v^2) := \displaystyle\frac{\alpha_1}{2}\int_{\omega_{1,d}\times (0,T)} (|y_1-y^1_{1,d}|^2+|y_2-y^1_{2,d}|^2)+\frac{\mu_1}{2}\int_{\omega_{1}\times (0,T)} \rho_*^2|v^1|^2, \\
                \\
                J_2(h , v^1, v^2) := \displaystyle\frac{\alpha_2}{2}\int_{\omega_{2,d}\times (0,T)} (|y_1-y^2_{1,d}|^2+|y_2-y^2_{2,d}|^2)+\frac{\mu_2}{2}\int_{\omega_{2}\times (0,T)} \rho_*^2|v^2|^2.  
            \end{cases}
        \end{equation}
        Here $\alpha_i,\mu_i>0$, $i=1,2$, are fixed constants, and $\rho_* \in C^{\infty}([0, T])$ is the weight function defined in \eqref{pesos} (it is a suitable positive function that blows up at $t=0$ and $t=T$). 

The first term in $J_i$ measures the discrepancy between the state and the
	trajectory desired by follower $i$, whereas the second term penalizes the
	size of the corresponding follower's control.    
\begin{definition} The control process can be described as follows.

\vglue 0.2 cm
 
 1. \textbf{Nash equilibrium.} The followers $v^1$ and $v^2$ assume that the leader $h$ has made a choice and intend to be in Nash equilibrium for the costs $J_i$, $i = 1, 2$. That is, once $h$ has been fixed, we look for a control $v^i\in L^2(\omega_i \times (0,T))$ that satisfies
 \begin{equation*}
     J_1(h, v^1, v^2) = \underset{\widehat{v}^1\in L^2(\omega_1 \times (0,T))}{min} \text{ }J_1(h,\widehat{v}^1, v^2), \qquad
    J_2(h, v^1, v^2) = \underset{\widehat{v}^2\in L^2(\omega_2 \times (0,T))}{min} \text{ }J_2(h, v^1, \widehat{v}^2).
  \end{equation*}

2. \textbf{Hierarchical null controllability with Nash equilibrium.} Given that the Nash equilibrium was reached for all fixed h, we looked for a specific control $h = \widehat{h} \in L^2(\omega\times (0,T))$ such that together with the Nash equilibrium $(v^1,v^2) = (v^1(h),v^2(h))$, the solution of system \eqref{Eq1} satisfies
\begin{equation}
    (y_1,y_2)(\cdot,T) = (0, 0), \ \text{ in } I.
\end{equation}
\end{definition}

\begin{definition}[\textbf{Nash quasi-equilibrium}]
Let the leader control $h$ be given. The pair $(v^1, v^2)$ is called a Nash quasi-equilibrium of the functionals $(J_1, J_2)$ if it satisfies
\begin{equation}
 \begin{cases} \label{equili}
     \dfrac{\partial J_1}{\partial v^1}(h, v^1, v^2)( \widehat{v}^1, 0) = 0, \text{ }\forall\widehat{v}^1 \in L^2(\omega_1\times (0,T)), \vspace{0.3cm} \\
     \dfrac{\partial J_2}{\partial v^2}(h, v^1, v^2)(0,\widehat{v}^2)= 0, \text{ }\forall\widehat{v}^2 \in L^2(\omega_2\times (0,T)).     
 \end{cases}
\end{equation}     
\end{definition}

\begin{remark}
If the functionals $J_1$ and $J_2$ defined by \eqref{energy} are convex
and if $(v^1, v^2)$ is a Nash quasi-equilibrium then it is a Nash equilibrium (therefore, in that case we have the equivalence of Nash quasi-equilibrium and Nash equilibrium). When a problem is nonlinear the functionals $J_1$ and $J_2$ are in general not convex.
\end{remark}

\subsection{Main result}
 The goal is to establish hierarchical null controllability for the system \eqref{Eq1}, and to do so some regularity conditions are needed. The following theorem guarantees us this controllability.
 
\begin{theorem}\label{teo:main_theorem_intro}
    Under the  assumptions \eqref{condiciones_a4}, \eqref{condiciones_a4f} and \eqref{condiciones_a4c(x,t)}, there exists $\varepsilon > 0$ such that, for any  $(f,g) \in [H_0^1(I)]^2$, satisfying 

    \begin{equation*}
        \Vert f \Vert_{H_0^1(I)}\leq \varepsilon,\quad \Vert g \Vert_{H_0^1(I)}\leq\varepsilon,\quad \Vert \tilde{\rho} y^{i}_{j,d} \Vert_{L^2(Q)}\leq\varepsilon,\quad i,j=1,2,
    \end{equation*}

    then, there exist a control $h \in L^2(\omega\times (0,T))$ and associated Nash equilibrium $(v^1,v^2)$ with $v^i \in L^2(0,T; L^2(\omega_i))$, $i=1, 2$, such that the corresponding solution to \eqref{Eq1} verifies
     \begin{equation*}
     	y_1(\cdot,T)=0,\quad y_2(\cdot,T)=0, \quad\text{in  }I.
     \end{equation*}
\end{theorem}\textbf{}
Here, $\tilde\rho$ denotes the weight introduced in \eqref{pesos}.
In our result, the initial data are assumed to belong to $H_0^1(I)$ and to be sufficiently small in this space. This is essential for the proof presented below.
\begin{itemize}
\item The main difficulties in the proof arise from the fact that the nonlinear terms occur in the principal part of the differential operators.
\item To overcome these difficulties, we employ an argument based on the so-called Lyusternik inverse mapping theorem in Hilbert spaces.
\end{itemize}

\section{Preliminary results}

\subsection{\bf Well-posedness }

 We will first examine the well
posedness of the  system \eqref{Eq1} in suitable spaces.

The following result is a direct consequence of the well-posedness analysis developed in \cite{rincon2006nonlinear}.

\begin{theorem}\label{wellpo_nondege}
Assume that hypotheses \eqref{condiciones_a4}-\eqref{condiciones_a4c(x,t)} are satisfied. Then, for any $(f, g)\in [H^1_0(I)]^2$,    $h  \in L^2(\omega\times (0,T))$ and $v^i\in L^2(\omega_i \times (0,T))$, $i=1, 2$, satisfying $$ \Vert f \Vert_{H_0^1(I)}\leq \varepsilon_0,\quad \Vert g \Vert_{H_0^1(I)}\leq\varepsilon_0,$$ for some positive constant $\varepsilon_0$,  the system \eqref{Eq1} admits a unique weak solution $(y_1, y_2)$, such that
\begin{equation}\label{lmiwellpos}
\begin{split}
  &\sup_{t\in\left[0,T\right]} \left( |(y_1(\cdot, t),y_2(\cdot, t))|_{[L^2(I)]^2}+ \Vert (y_1(\cdot,t),y_2(\cdot,t))\Vert_{[H^1_0(I)]^2} \right) \\
 &+\Vert (y_{1,xx},y_{2,xx}\Vert_{[L^2(Q)]^2}  +\Vert (y_{1,t},y_{2,t}\Vert_{[L^2(Q)]^2} \\
&\leq C\left( \|h\|_{L^2(\omega\times (0,T))}  + \|v^1\|_{L^2(\omega_1\times (0,T))} + \Vert v^2 \Vert_{L^2(\omega_2\times (0,T))} +\Vert (f, g)\Vert_{[H^1_0(I)]^2}\right).   
\end{split}
\end{equation}
\end{theorem}

\subsection{\bf Characterization of the Nash quasi-equilibrium and convexity}

\subsubsection{Characterization of Nash quasi-equilibrium}

Initially, we will derive the optimality system, which will be obtained from the characterization of the Nash quasi-equilibrium.

\begin{proposition}\label{propo_cara_equi_nash}
    Let $(\widehat{v}^1, \widehat{v}^2) \in  L^2(0,T;L^2(\omega_1))\times L^2(0,T;L^2(\omega_2))$ be the Nash quasi-equilibrium pair for $(J_1, J_2)$, where $h \in L^2(\omega\times (0,T))$ and $\mu_i$, $i = 1, 2$, are some positive constants.  Then, there exist $p_j^i \in L^2(0,T;L^2(\omega_j))$,  $i,j = 1, 2$, such that the Nash quasi-equilibrium pair $(\widehat{v}^1, \widehat{v}^2)$  is characterized by
    \begin{equation} \label{carac1}
        \widehat{v}^i = -\frac{1}{\mu_i}\rho_*^{-2}p_1^i, \, \,  \text{  in  } \, \omega_i\times( 0,T),\,  i = 1, 2,
    \end{equation}
    where $(y_1, y_2, p_1^1, p_2^1,p_1^2, p_2^2)$ is the solution of the following optimality system
    	\begin{align} \label{to1}
    		\begin{cases}
    			y_{1,t}-(a(y_1)y_{1,x})_x + F_1(y_1)   = h\chi_\omega-\frac{1}{\mu_1}\rho_*^{-2}p_1^1\chi_{\omega_1}-\frac{1}{\mu_2}\rho_*^{-2}p_1^2\chi_{\omega_2}, &\text{ in } Q,\\  
    			y_{2,t}-(a(y_2)y_{2,x})_x + F_2(y_2)  +cy_1 = 0, &\text{ in } Q,\\  
    			-p_{1,t}^1-(a(y_1)p_{1,x}^1)_x + a'(y_1)y_{1,x}p_{1,x}^1+F_1'(y_1)p_1^1+cp_2^1=\alpha_1(y_1-y_{1,d}^1)\chi_{\omega_{1,d}}, &\text{ in } Q,\\  
    			-p_{1,t}^2-(a(y_1)p_{1,x}^2)_x + a'(y_1)y_{1,x}p_{1,x}^2+F_1'(y_1)p_1^2+cp_2^2=\alpha_2(y_1-y_{1,d}^2)\chi_{\omega_{2,d}}, &\text{ in } Q,\\  
    			-p_{2,t}^1-(a(y_2)p_{2,x}^1)_x + a'(y_2)y_{2,x}p_{2,x}^1+F_2'(y_2)p_2^1=\alpha_1(y_2-y_{2,d}^1)\chi_{\omega_{1,d}}, &\text{ in } Q,\\  
    			-p_{2,t}^2-(a(y_2)p_{2,x}^2)_x + a'(y_2)y_{2,x}p_{2,x}^2+F_2'(y_2)p_2^2=\alpha_2(y_2-y_{2,d}^2)\chi_{\omega_{2,d}}, &\text{ in } Q,\\  
    			y_1(0,t) = y_1(L,t) =0, y_2(0,t) = y_2(L,t) = 0, &\text{ on } (0,T),\\  
    			p_1^i(0,t) = p_1^i(L,t) =0, \, p_2^i(0,t) = p_2^i(L,t) = 0, i=1,2&\text{ on } (0,T),\\
    			y_1(x,0) = f(x), \,  y_2(x,0) = g(x), &\text{ in } I,\\
    			p_1^i(x,T) =0, p_2^i(x,T) =0, i=1,2, &\text{ in } I.\\
    		\end{cases}
    	\end{align}
\end{proposition}

\begin{proof}

Since $(\widehat{v}^1, \widehat{v}^2)$  is a Nash quasi-equilibrium, from \eqref{equili}, we have the following equality:
\begin{equation}\label{der. func1}
\alpha_i\int_{\omega_{1,d}\times (0,T)} \left[ z^i_1(y_1-y^i_{1,d})+ z^i_2(y_2-y^i_{2,d}) \right]+\mu_i\int_{\omega_{1}\times (0,T)} \rho_*^2 v^i\widehat{v}^i=0,
\end{equation}for all ${v}^i \in  L^2(0,T;L^2(\omega_i))$, $i=1, 2$, where $(z_1^i,z_2^i)$ satisfies the system
\begin{align} \label{sist. der. 1}
		\begin{cases}		
			z^i_{1,t}-(a(y_1)z^i_{1,x})_x -(a'(y_1)y_{1,x}z^i_1)_x+F_1'(y_1)z^i_1   = \widehat{v}^i\chi_{\omega_1},&\text{ in } Q,\\  
			z^i_{2,t}-(a(y_2)z^i_{2,x})_x -(a'(y_2)y_{2,x}z^i_2)_x+F_2'(y_2)z^i_2  +cz^i_1 = 0, &\text{ in } Q,\\  
			z^i_1(0,t) = z^i_1(L,t) =0, z^1_2(0,t) = z^i_2(L,t) = 0, &\text{ on } (0,T),\\
			z^i_1(x,0) = 0, z^i_2(x,0) = 0, &\text{ in } I.
		\end{cases}
	\end{align}On the other hand, for $i=1, 2$, let $(p^i_1, p^i_2)$ be the  adjoint system  to \eqref{sist. der. 1}. Thus, $(p^i_1, p^i_2)$ is solution of the system
    \begin{align}\label{adj.der1}
		\begin{cases}		
			-p^i_{1,t}-(a(y_1)p^i_{1,x})_x +a'(y_1)y_{1,x}p^i_{1,x}+F_1'(y_1)p^i_1 +cp_i^1  =  \alpha_1(y_1-y^i_{1,d})\chi_{\omega_{i,d}},  &\text{ in } Q,\\  
			-p^i_{2,t}-(a(y_2)p^i_{2,x})_x -a'(y_2)y_{2,x}p^i_{2,x}+F_2'(y_2)p^i_2   = \alpha_1(y_2-y^i_{2,d})\chi_{\omega_{i,d}}, &\text{ in } Q,\\  
			p^i_1(0,t) = p^i_1(L,t) =0, p^i_2(0,t) = p^i_2(L,t) = 0, &\text{ on } (0,T),\\
			p^i_1(x,T) = 0, p^i_2(x,T) = 0, &\text{ in } I.
		\end{cases}
	\end{align}To derive the identity \eqref{carac1}, we multiply $\eqref{adj.der1}_1$ and $\eqref{adj.der1}_2$ by $z^i_1$ and $z^i_2$, respectively, and integrate by parts on $Q$, to obtain
\begin{equation}\label{mulzandadjunp}
\alpha_i\int_{\omega_{i,d}\times (0,T)} \left[z^i_1(y_1-y^i_{1,d})+z^i_2(y_2-y^i_{2,d})\right]=\int_{\omega_{i}\times (0,T)} v^i  p_1^i,
\end{equation}for all ${v}^i \in  L^2(0,T;L^2(\omega_i))$, $i=1, 2$. Finally, from \eqref{der. func1} and \eqref{mulzandadjunp}, we deduce \eqref{carac1} and  the proof is complete .
\end{proof}

\begin{remark}\label{import_remark_1} Using the estimate \eqref{lmiwellpos}    and performing the same process as in  \cite{Nunez-Li2023}, we have that the system  \eqref{to1} admits a unique weak solution $(y_1, y_2, p_1^1, p_2^1,p_1^2, p_2^2)$, such that 
\begin{equation}\label{lmiwellpos_optimalsiste}
\begin{split}
&\sup_{t\in\left[0,T\right]} \left( |(y_1(\cdot, t),y_2(\cdot, t))|_{[L^2(I)]^2}+ \Vert (y_1(\cdot,t),y_2(\cdot,t))\Vert_{[H^1_0(I)]^2} \right) \\
 & +\sum^2_{i=1}\left[\sup_{t\in\left[0,T\right]} \left( |(p^{i}_1(\cdot, t),p^{i}_2(\cdot, t))|_{[L^2(I)]^2}+ \Vert (p^{i}_1(\cdot,t),p^{i}_2(\cdot,t))\Vert_{[H^1_0(I)]^2} \right)\right] \\
&\leq C\Big( \|h\|_{L^2(\omega \times (0,T))} + \Vert (f, g)\Vert_{[H^1_0(I)]^2} \\
&+ \sum_{k=1}^2 \|y_{k,d}^1\|_{L^2(\omega_{1,d}\times (0,T))} + \sum_{k=1}^2 \|y_{k,d}^2\|_{L^2(\omega_{2,d}\times (0,T))} \Big).
\end{split}
\end{equation}
\end{remark}

\subsubsection{Convexity of the functionals: Nash equilibrium}

Now, we will study the convexity of functionals defined in \eqref{energy}. To achieve this, we prove the result for the system in a general form:
\begin{equation}\label{eqprin544}
	\begin{cases}
		y_{1,t}-b(t)\left(a_1(y_{1})y_{1,x}\right)_x-B_1(x,t) y_{1,x}+F_1(y_1,y_2) \\
        ={h}\chi_{\omega}+{v}^1\chi_{\omega_1}+{v}^2\chi_{\omega_2}, & \ \ \ \text{in} \ \ \ {Q}, \\
		y_{2,t}-b(t)\left(a_2(y_{2})y_{2,x}\right)_x-B_2(x,t) y_{2,x}+F_2(y_1,y_2)=0, & \ \ \ \text{in} \ \ \ {Q},\\
		y_1(0,t)=y_1(L,t)=y_2(0,t)=y_2(L,t)=0, & \ \ \ \text{on} \ \ \ (0,T),\\
		y_1(0)=y_1^0, \ y_2(0)=y_2^0, & \ \ \ \text{in} \ \ \ I,
	\end{cases}
\end{equation}	
where the coefficients $b(t)$, $B_i(x,t)$, $i=1,2$, are bounded, and $0 < b_0 < b(t)$, for some constant $b_0$. The functions $F_i$, $i=1,2$ are Lipschitz.

We have the following result.

\begin{theorem}\label{teoconvex}
	Assume that \eqref{condiciones_a4}, \eqref{condiciones_a4f} hold and that $y_1^0, y_2^0\in H^1_0(I)$ and $y^i_{1,d} \in L^2(\omega_{1,d} \times (0,T))$, $y^i_{2,d} \in L^2(\omega_{2,d} \times (0,T))$. Let us suppose that $h\in L^2(\omega \times (0,T))$ and $\mu_i, i=1,2$ are sufficiently large. Then, if $(v^1,v^2)$ is a Nash quasi-equilibrium for $J_i, i=1,2$, there exists a constant $C>0$ independent of $\mu_i$, $i=1,2$, such that
	\begin{equation}\label{ecsegundaderivada}
		D_i^2 J_i (h; v^1,v^2)\cdot (\widetilde{v}^i,\widetilde{v}^i)\geq C\|\widetilde{v}^i\|^2_{L^2(\omega_i\times (0,T))},\ \ \ \ \ \
		\forall  \ \widetilde{v}^i \in L^2(\omega_i\times (0,T)), \ \ \ i=1,2.
	\end{equation}
\end{theorem}

\begin{proof}
  The proof is based on the approach used in \cite[Proposition 1.4,]{Ara-EFC-San2015}. Let $(v^1,v^2)$ be a Nash quasi-equilibrium for $J_i$, $i=1,2$. 

   Given $\lambda\in \mathbb{R}$ and $\widetilde{v}^1, \overline{v}^1$ in $L^2(\omega_1\times (0,T))$, let us consider $(y_1^\lambda, y_2^\lambda)$  the solution of the following system:
\begin{equation}\label{ylambda}
        \begin{cases}
    		y^\lambda_{1,t}-b(t)\left(a_1(y^\lambda_{1})y^\lambda_{1,x}\right)_x-B_1(x,t) y^\lambda_{1,x}+F_1(y^\lambda_1,y^\lambda_2) \\
            ={h}\chi_{\omega}+(v^1+\lambda\widetilde{v}^1)\chi_{\omega_1}+{v}^2\chi_{\omega_2}, & \ \ \ \text{in} \ \ \ {Q},\\
    		y^\lambda_{2,t}-b(t)\left(a_2(y^\lambda_{2})y^\lambda_{2,x}\right)_x-B_2(x,t) y^\lambda_{2,x}+F_2(y^\lambda_1,y^\lambda_2)=0,& \ \ \ \text{in} \ \ \ {Q},\\
    	   y^\lambda_1(0,t)=y^\lambda_2(L,t)=y^\lambda_2(0,t)=y^\lambda_2(L,t)=0, & \ \ \ \text{on} \ \ \ (0,T),\\
    		y^\lambda_1(0)=y_0^1, \ y^\lambda_2(0)=y_2^0, & \ \ \ \text{in} \ \ \ I.
    	\end{cases}
    \end{equation} and we denote $y_i:=\left. y_i^\lambda\right|_{\lambda=0}$, $i=1, 2$.  Thus, we deduce that
    \begin{equation}\label{ecD2defi}
        \begin{split}
            & \langle D_1 J_1(h,v^1+\lambda\widetilde{v}^1,v^2 ),\overline{v}^1 \rangle-  \langle D_1 J_1(h,v^1,v^2 ),\overline{v}^1 \rangle\\
            &= \lambda \mu_1 \int_{0}^{T}\int_{\mathcal{O}_{1}} \rho_*^2 \overline{v}^1\widetilde{v}^1dxdt
    		+\alpha_1\int_{0}^{T}\int_{\mathcal{O}_{1,d}}\left( (y_1^\lambda-y^1_{1,d})q^{1,\lambda}_1+(y_2^\lambda-y^1_{2,d})q^{1,\lambda}_2 \right)\ dx dt \\
    		&-\alpha_1\int_{0}^{T}\int_{\mathcal{O}_{1,d}}\left( (y_1-y^1_{1,d})q^{1}_1+(y_2-y^1_{2,d})q^{1}_2 \right)\ dx dt,
        \end{split}
    \end{equation}where $q^{1,\lambda}_1, q^{1,\lambda}_2$ are the derivative of the state $y^\lambda_1$ and $y^\lambda_2$ with respect to $v^1+\lambda\widetilde{v}^1$ in the direction $\widetilde{v}^1$. That is, $(q^{1,\lambda}_1, q^{1,\lambda}_2)$ is solution of system   

\begin{equation}\label{qlambda}
    	\begin{cases}
    		q^{1,\lambda}_{1,t}-b(t)\left(a_1'(y^\lambda_{1}) q^{1,\lambda}_{1} y^\lambda_{1,x} + a_1(y^\lambda_{1,x})q^{1,\lambda}_{1,x}\right)_x-B_1(x,t) q^{1,\lambda}_{1,x}\\
             +D_1 F_1(y_1^\lambda,y_2^\lambda)q^{1,\lambda}_1+ D_2 F_1(y_1^\lambda,y_2^\lambda)q^{1,\lambda}_2=\widetilde{v}^1\chi_{\omega_1}, & \ \ \text{in} \ \ {Q}, \\
    		q^{1,\lambda}_{2,t}-b(t)\left( a_2'(y^\lambda_{2}) q^{1,\lambda}_{2} y^\lambda_{2,x} + a_2(y^\lambda_{2})q^{1,\lambda}_{2,x}\right)_x-B_2(x,t) q^{1,\lambda}_{2,x}\\
            +D_1 F_2(y_1^\lambda,y_2^\lambda) q^{1,\lambda}_1+ D_2 F_2(y_1^\lambda,y_2^\lambda)q^{1,\lambda}_2=0, &\ \ \text{in} \ \ {Q},\\
    		q^{1,\lambda}_1(0,t)=q^{1,\lambda}_1(L,t)=q^{1,\lambda}_2(0,t)=q^{1,\lambda}_2(L,t)=0, & \ \ \text{on} \ \ (0,T),\\
    		q^{1,\lambda}_1(0)=0, \ q^{1,\lambda}_2(0)=0, & \ \ \text{in} \ \ I,
    	\end{cases}
    \end{equation}with  notation  $q_i:=\left. q_i^\lambda\right|_{\lambda=0}$, $i=1, 2$.

    Moreover, let us consider the adjoint system of \eqref{qlambda}, given by 
    \begin{equation}\label{qlambdaadjunto}
    	\begin{cases}
    		-p^{1,\lambda}_{1t}-b(t)\left(a_1'(y^\lambda_{1}) p^{1,\lambda}_{1} y^\lambda_{1,x} + a_1(y^\lambda_{1})p^{1,\lambda}_{1,x}\right)_x+\left(B_1(x,t)p^{1,\lambda}_1\right)_x \\
            + D_1 F_1(y_1^\lambda,y_2^\lambda)p^{1,\lambda}_1 + D_1 F_2(y_1^\lambda,y_2^\lambda)p^{1,\lambda}_2 =\alpha_1\left( y^\lambda_1-y_{1,d}^1  \right) \chi_{\omega_{1,d}}, &\ \text{in} \ {Q}, \\
    		-p^{1,\lambda}_{2,t}-b(t)\left(a_2'(y^\lambda_{2})p^{1,\lambda}_{2}y^\lambda_{2,x} + a_2(y^\lambda_{2})p^{1,\lambda}_{2,x}\right)_x+\left(B_2(x,t)p^{1,\lambda}_2\right)_x \\
            + D_2 F_1(y_1^\lambda,y_2^\lambda)p^{1,\lambda}_1 + D_2 F_2(y_1^\lambda,y_2^\lambda)p^{1,\lambda}_2 =\alpha_1\left( y^\lambda_2-y_{2,d}^1  \right)\chi_{\omega_{2,d}}, &\ \text{in} \ {Q},\\
    		p^{1,\lambda}_1(0,t)=p^{1,\lambda}_1(L,t)=p^{1,\lambda}_2(0,t)=p^{1,\lambda}_2(L,t)=0, & \ \text{on} \  {(0,T)},\\
    		p^{1,\lambda}_1(T)=0, \quad p^{1,\lambda}_2(T)=0, & \ \text{in} \ I,
    	\end{cases}
    \end{equation}Multiplying the first and the second equation of \eqref{qlambda} by $p^{1,\lambda}_{1}$ and $p^{1,\lambda}_{2}$ , respectively, integrating by parts over $Q$ and using \eqref{qlambdaadjunto}, we have that
\begin{equation}\label{inter1}
		\alpha_1\int_{0}^{T}\int_{\omega_{1,d}}\left( (y_1^\lambda-y^1_{1,d})q^{1,\lambda}_1+(y_2^\lambda-y^1_{2,d})q^{1,\lambda}_2 \right)\ dx dt=\int_{Q} \widetilde{v}^1\chi_{\omega_1}p^{1,\lambda}_{1} \ dxdt.
	\end{equation}Similarly, for $\lambda=0$, we obtain
	\begin{equation}\label{inter2}
		\alpha_1\int_{0}^{T}\int_{\omega_{1,d}}\left( (y_1-y^1_{1,d})q^{1}_1+(y_2-y^1_{2,d})q^{1}_2 \right)\ dx dt=\int_{Q} \widetilde{v}^1\chi_{\omega_1}p^{1}_{1} \ dxdt.	\end{equation}Therefore, from \eqref{ecD2defi}, \eqref{inter1} and \eqref{inter2}, we obtain the following identity:
        \begin{equation}\label{ecD2defisimpli}
        \begin{split}
          &\langle D_1 J_1(h,v^1+\lambda\widetilde{v}^1,v^2 ),\overline{v}^1 \rangle-  \langle D_1 J_1(h,v^1,v^2 ),\overline{v}^1 \rangle \\
          &=\lambda \mu_1 \int_{0}^{T}\int_{\omega_1} \rho_*^2 \overline{v}^1\widetilde{v}^1 \ dxdt
 	      +\int_{0}^{T}\int_{\omega_{1}}\left( p^{1,\lambda}_1-p^{1}_1 \right)\widetilde{v}^1 \ dx dt.   
        \end{split}
    \end{equation}On the other hand, from \eqref{ylambda} and \eqref{qlambdaadjunto}, we deduce that
    \begin{equation}\label{auxaux_1}
    \begin{cases}
		(y_1^\lambda-y_1)_{t}-b(t)\left((a_1(y^\lambda_{1})-a_1(y_{1}))y^\lambda_{1,x} + a_1(y_{1})(y_1^\lambda-y_1)_{x}\right)_x \\
        \qquad -B_1(x,t) (y_1^\lambda-y_1)_{x}+F_1(y^\lambda_1,y^\lambda_2)-F_1(y_1,y_2)=\lambda\widetilde{v}^1\chi_{\omega_1}, &\\
		(y_2^\lambda-y_2)_{t}-b(t)\left((a_2(y^\lambda_{2})-a_1(y_{2}))y^\lambda_{2,x} + a_2(y_{2})(y_2^\lambda-y_2)_{x}\right)_x \\
        \qquad -B_2(x,t) (y_2^\lambda-y_2)_{x}
        +F_2(y^\lambda_1,y^\lambda_2)-F_2(y_1,y_2)=0,&
	\end{cases}
    \end{equation}and
    \begin{equation}\label{auxaux_2}
      \begin{cases}
			-(p^{1,\lambda}_1-p^{1}_1)_{t}  -b(t)\left[ (a_1(y_1^\lambda)-a_1(y_1))p^{1,\lambda}_{1,x} + a_1(y_1) (p^{1,\lambda}_{1,x}-p^1_{1,x}) \right]_x \\
            -b(t)\left[(a_1'(y_1^\lambda)-a_1'(y_1))p^{1,\lambda}_1 y^\lambda_{1,x} + a_1'(y_1)(p^{1,\lambda}_1-p^1_1)y^\lambda_{1,x} + a_1'(y_1)p^1_1 (y^\lambda_{1,x}-y_{1,x})\right]_x\\
			+[D_1 F_1(y_1^\lambda,y_2^\lambda)-D_1 F_1(y_1,y_2)]p^{1,\lambda}_1 + D_1 F_1(y_1,y_2)(p^{1,\lambda}_1-p^{1}_1)\\
            +[D_1 F_2(y_1^\lambda,y_2^\lambda)-D_1 F_2(y_1,y_2)]p^{1,\lambda}_2 + D_1 F_2(y_1,y_2)(p^{1,\lambda}_2-p^{1}_2) \\
            +\left(B_1(x,t)(p^{1,\lambda}_1-p^{1}_1)\right)_x=\alpha_1\left( y^\lambda_1-y_1  \right)\chi_{\omega_{1,d}},  \\
			-(p^{1,\lambda}_2-p^{1}_2)_{t}   -b(t)\left[ (a_2(y_2^\lambda)-a_2(y_2))p^{1,\lambda}_{2,x} + a_2(y_2) (p^{1,\lambda}_{2,x}-p^1_{2,x}) \right]_x \\
            -b(t)\left[ (a_2'(y_2^\lambda)-a_2'(y_2))p^{1,\lambda}_2 y^\lambda_{2,x} + a_2'(y_2)(p^{1,\lambda}_2-p^1_2)y^\lambda_{2,x} + a_2'(y_2)p^1_2 (y^\lambda_{2,x}-y_{2,x}) \right]_x\\
            +[D_2 F_1(y_1^\lambda,y_2^\lambda)-D_2 F_1(y_1,y_2)]p^{1,\lambda}_1 + D_2 F_1(y_1,y_2)(p^{1,\lambda}_1-p^{1}_1)\\
            +[D_2 F_2(y_1^\lambda,y_2^\lambda)-D_2 F_2(y_1,y_2)]p^{1,\lambda}_2 + D_2 F_2(y_1,y_2)(p^{1,\lambda}_2-p^{1}_2)\\
			+\left(B_2(x,t) (p^{1,\lambda}_2-p^{1}_2)\right)_x=\alpha_1\left( y^\lambda_2-y_2  \right)\chi_{\omega_{2d}},
		\end{cases}  
\end{equation}respectively. 

\noindent Thus, if we set 
\begin{equation*}
	 \theta_i:=\lim_{\lambda\to 0} \frac{1}{\lambda}(y^\lambda_i - y_i), \, \, \text{ and \, \, }	\eta^1_i:=\lim_{\lambda\to 0} \frac{1}{\lambda}(p_i^{1,\lambda} - p_i^1),  \ \ i=1,2, 
	\end{equation*}from \eqref{auxaux_1}, \eqref{auxaux_2} and using \eqref{condiciones_a4}, it follows that 
\begin{equation}\label{sistemaconvexo1}
	\begin{cases}
		\theta_{1,t}-b(t)\left[a_1'(y_{1})\theta_1 y_{1,x}+a_1(y_{1})\theta_{1,x}\right]_x-B_1(x,t)\theta_{1,x}\\
        +D_1 F_1(y_1,y_2)\theta_1 + D_2 F_1(y_1,y_2)\theta_2 = \widetilde{v}^1\chi_{\omega_1}, & \ \ \ \text{in} \ \ \ {Q},\\
		\theta_{2,t}-b(t)\left[a_2'(y_{2})\theta_2 y_{2,x}+a_2(y_{2})\theta_{2,x}\right]_x-B_2(x,t)\theta_{2,x}\\
        +D_1 F_2(y_1,y_2)\theta_1 + 
        D_2 F_2(y_1,y_2)\theta_2=0,& \ \ \ \text{in} \ \ \ {Q},\\
		\theta_1(0,t)=\theta_1(L,t)=\theta_2(0,t)=\theta_2(L,t)=0, & \ \ \ \text{on} \ \ \ (0,T),\\
		\theta_1(0)=0, \ \theta_2(0)=0, & \ \ \ \text{in} \ \ \ I,
	\end{cases}
\end{equation} and
    \begin{equation}\label{sistemaconvexo2}
		\begin{cases}
			-\eta^1_{1,t} 
            -b(t) \left[ a_1'(y_1)\eta^1_1 y_{1,x} + a_1(y_1) \eta^1_{1,x}\right]_x \\
             + \left(B_1(x,t) \eta^1_1\right)_x
            + D_1 F_1(y_1,y_2)\eta_1^1 + D_1 F_2(y_1,y_2)\eta_2^1 +\mathcal{N}_1^1(\eta,\theta) \\
             =\alpha_1\theta_1\chi_{\omega_{1,d}}, &\ \text{in} \ {Q},\\
			-\eta^1_{2,t} 
            - b(t) \left[ a_2'(y_2)\eta^1_2 y_{2,x} + a_2(y_2) \eta^1_{2,x}\right]_x \\
            + \left(B_2(x,t) \eta^1_2\right)_x
            + D_2 F_1(y_1,y_2)\eta_1^1 + D_2 F_2(y_1,y_2)\eta_2^1 + \mathcal{N}_2^1(\eta,\theta) \\
            =\alpha_1\theta_2\chi_{\omega_{1,d}}, &\ \text{in} \ {Q},\\
			\eta^{1}_1(0,t)=\eta^{1}_1(L,t)=\eta^{1}_2(0,t)=\eta^{1}_2(L,t)=0, & \ \text{on} \  {(0,T)},\\
			\eta^{1}_1(T)=0, \ \eta^{1}_2(T)=0, & \ \text{in} \ I,
		\end{cases}
	\end{equation}respectively,  where
    \begin{equation}\label{N_1N_2masnafren}
        \begin{split}
			\mathcal{N}_1^1(\eta,\theta) &=[D_{11}^2 F_1(y_1,y_2)\theta_1+D_{12}^2 F_1(y_1,y_2)\theta_2]p_1^1
			+[D_{11}^2 F_2(y_1,y_2)\theta_1+D_{12}^2 F_2(y_1,y_2)\theta_2]p_2^1 \\
            &\qquad -b(t) \left[ a_1''(y_{1})\theta_1 p^1_1 y_{1,x} + a_1'(y_{1})p^1_1 \theta_{1,x} + a_1'(y_1) \theta_1 p^1_{1,x} \right]_x, \\
            \mathcal{N}_2^1(\eta,\theta) &= [D_{21}^2 F_1(y_1,y_2)\theta_1+D_{22}^2 F_1(y_1,y_2)\theta_2]p_1^1
            +[D_{21}^2 F_2(y_1,y_2)\theta_1+D_{22}^2 F_2(y_1,y_2)\theta_2]p_2^1 \\
            &\qquad - b(t) \left[ a_2''(y_{2})\theta_2 p^1_2 y_{2,x} + a_2'(y_{2})p^1_2 \theta_{2,x} + a_2'(y_2) \theta_2 p^1_{2,x} \right]_x.
        \end{split}
	\end{equation}Then, from \eqref{sistemaconvexo1}, \eqref{sistemaconvexo2} and considering $\Bar{v}^{1}=\Tilde{v}^{1}$in \eqref{ecD2defisimpli}, we have that
    \begin{equation}
		\label{D2Jigual}
		\langle D_1^2 J_1(h,v^1,v^2), (\bar{v}^{1}, \bar{v}^{1}) \rangle = \int_{0}^{T}\int_{\omega_1} \eta_1^1 \bar{v}^{1} \ dxdt + \mu_1 \int_{0}^{T}\int_{\omega_1} \rho_*^2 |\bar{v}^{1}|^{2} \ dxdt. 
	\end{equation}

\textbf{Claim:}  There exists a constant $C> 0$ such that
\begin{equation}
\label{eq:integral_eta_w1_bounded3}
		\left|\int_{\omega_1 \times (0,T)} \eta_1^1 \bar{v}^{1} \ dxdt\right|\leq C  \|\bar{v}^{1}\|^{2}_{L^2(\omega_1\times (0,T))}
	\end{equation}Indeed, in order to obtain estimate \eqref{eq:integral_eta_w1_bounded3}, we multiply the first equation in  $\eqref{sistemaconvexo1}$ by $\eta_1^1$, the second one by $\eta_2^1$, integrate by parts on $Q$ and using the equation in \eqref{sistemaconvexo2} to obtain
    \begin{align*}
    \int_{\omega_1 \times (0,T)} \eta^1_1\bar{v}^{1} \ dxdt = \int_{Q} \theta_1\left( \alpha_1\theta_1\chi_{\omega_{1,d}} - \mathcal{N}_1^1(\eta,\theta) \right) + \theta_2\left( \alpha_1\theta_2\chi_{\omega_{1,d}} - \mathcal{N}_2^1(\eta,\theta) \right).
    \end{align*}  By replacing \eqref{N_1N_2masnafren}  in above identity, we obtain the following estimate: 
\begin{equation}\label{imporVVVVV}
\begin{split}
\left|\int_{\omega_1 \times (0,T)} \eta \tilde{v}^{1} \ dxdt\right| &\leq C \int_{Q} \left[ |\theta_1|^2 + |\theta_2|^2 +  2(|\theta_1|^2+|\theta_2|^2)(|p_1^1|+|p_2^1|) \right]dxdt\\
&+C \int_Q \left[|y_{1,x}||\theta_{1,x}| |\theta_1| |p^1_1| + |\theta_{1,x}|^2 |p^1_1| + |\theta_{1,x}| |\theta_1| |p^1_{1,x}|\right] dx dt\\
& +C \int_Q \left[ |y_{2,x}||\theta_{2,x}| |\theta_2| |p^1_2| + |\theta_{2,x}|^2 |p^1_2| + |\theta_{2,x}| |\theta_2| |p^1_{2,x}| \right] dxdt, 
\end{split}
\end{equation}for some positive constant $C$.

    \noindent On the other hand, by standard energy estimates (see, e.g., the appendix in \cite{GYL-Carleman-2025}), we get \begin{equation}\label{eq:energy_theta}
	    \sum_{i=1}^2 \|\theta_{i,x}\|^2_{L^2(0,T;L^2(I))} + \sum_{i=1}^2 	\|\theta_{i,xx}\|^2_{L^2(0,T;L^2(I))} \leq C \|\widetilde{v}^1\|^2_{L^2(\omega_1\times (0,T))},
    \end{equation}and
\begin{equation}\label{eq:energy_p_ik}
\begin{split}
	       &\| y_{j,x}\|^2_{L^\infty(0,T;L^2(I))} + \|y_j\|^2_{L^\infty(0,T;L^2(I))} +
           \sum_{k=1}^2 \| p_{k,x}^j\|^2_{L^\infty(0,T;L^2(I))} + \sum_{k=1}^2 	\|p_k^j\|^2_{L^\infty(0,T;L^2(I))} \\
           & \leq C \left( \|h\|^2_{L^2(\omega \times (0,T))} + \sum_{k=1}^2 \|y_k^0\|^2_{H_0^1(I)} + \sum_{k=1}^2 \|y_{k,d}^1\|_{L^2(\omega_{i,d}\times (0,T))}^2 \right),
        \end{split}
    \end{equation}for $j=1,2$.
    
    Furthermore, from estimates \eqref{eq:energy_theta} and \eqref{eq:energy_p_ik}, we deduce that
\begin{equation}\label{eq:est_v1_3}
\begin{split}
&\int_{Q} \left[ |\theta_1|^2 + |\theta_2|^2 +  2(|\theta_1|^2+|\theta_2|^2)(|p_1^1|+|p_2^1|) \right] dxdt \\
& \leq C\left(\|\theta_1\|_{L^\infty(0,T;L^2(I))}^{2} + \|\theta_2\|_{L^\infty(0,T;L^2(I))}^{2}\right) 
            \left( 1 +  \int_0^T \left(\|p_1^1(t)\|_{H^1_0(I)} + \|p_2^1(t)\|_{H^1_0(I)} \right) \right) \\
            & \leq \widehat{C}_1 \|\widetilde{v}^1\|^2_{L^2(\omega_1\times (0,T))},
\end{split}
\end{equation}where $\widehat{C}_1= \widehat{C}_1\left( h, y^0_1, y^0_2, y^1_{1, d}, y^1_{2, d}\right)$ is some positive constant.

\noindent Analogously, from the Sobolev embedding $H^1_0(I) \hookrightarrow L^\infty(I)$ and estimates \eqref{eq:energy_theta} and  \eqref{eq:energy_p_ik}, it follows that
\begin{equation}\label{eq:est_v1_1}
\begin{split}
&\int_Q \left[|y_{1,x}||\theta_{1,x}| |\theta_1| |p^1_1| + |\theta_{1,x}|^2 |p^1_1| + |\theta_{1,x}| |\theta_1| |p^1_{1,x}|\right] dx dt\\
&\leq \int_0^T \left( \|\theta_{1}\|_{L^\infty(I)}^2 + \|\theta_{1,x}\|_{L^\infty(I)}^2 \right) \left( \| y_{1,x} \|_{L^2(I)} \| p_{1}^1 \|_{L^2(I)} + \| p_{1}^1 \|_{L^1(I)} + \| p_{1,x}^1 \|_{L^1(I)} \right) dxdt\\
& \leq C \left( \| y_{1,x} \|_{L^\infty(0,T;L^2(I))} \| p_{1}^1 \|_{L^\infty(0,T;L^2(I))} + \| p_{1}^1 \|_{L^\infty(0,T;L^2(I))} \right. \\
            & \left. \quad + \| p_{1,x}^1 \|_{L^\infty(0,T;L^2(I))} \right)  \left( \|\theta_{1}\|_{L^2(0,T;H_0^1(I)}^2 + \|\theta_{1,x}\|_{L^2(0,T;H_0^1(I))}^2 \right) \\
& \leq C\left( 2+\|h\|^2_{L^2(\omega\times (0,T))} + \sum_{k=1}^2 \|y_k^0\|^2_{H_0^1(I)} \right. \\
            &\left.\quad + \sum_{k=1}^2 \|y_{k,d}^1\|_{L^2(\omega_{i,d}\times (0,T))}^2  \right)^2 \left( \|\theta_{1}\|_{L^2(0,T;H_0^1(I))}^2 + \|\theta_{1,x}\|_{L^2(0,T;H_0^1(I))}^2 \right) \\
            &\leq \widehat{C}_1\|\widetilde{v}^1\|^2_{L^2(\omega_1\times (0,T))}.
\end{split}
\end{equation}
Proceeding as in the previous estimate, we obtain that
 \begin{equation}
        \label{eq:est_v1_2}
        \int_Q \left( |y_{2,x}||\theta_{2,x}| |\theta_2| |p^1_2| + |\theta_{2,x}|^2 |p^1_2| + |\theta_{2,x}| |\theta_2| |p^1_{2,x}| \right) dxdt \leq C_2 \|\widetilde{v}^1\|^2_{L^2(\omega_1\times (0,T))}.
    \end{equation}Thus, from \eqref{imporVVVVV} and estimates \eqref{eq:est_v1_3}-\eqref{eq:est_v1_2}, we deduce \eqref{eq:integral_eta_w1_bounded3}. This proves the claim.

    Therefore, from \eqref{D2Jigual}, \eqref{eq:integral_eta_w1_bounded3},  and using that $\rho_*^{-2} \leq D$, for some $D>0$, then there exists a positive constant  $C$, independent of $\mu_1$ and $\mu_2$, such that
    \begin{equation*}
        \langle D_1^2 J_1(h,v^1,v^2), (\bar{v}^{1}, \bar{v}^{1}) \rangle \geq  (\frac{\mu_1}{D} -C) \|\bar{v}^1\|^2_{L^2(\omega_1\times (0,T))}.
    \end{equation*}Similarly, it follows that
    \begin{equation*}
         \langle D_2^2 J_1(h,v^1,v^2), (\bar{v}^{2}, \bar{v}^{2}) \rangle \geq  (\frac{\mu_2}{D}-C) \|\bar{v}^2\|^2_{L^2(\omega_2\times (0,T))}.   
    \end{equation*}
    Finally, for $\mu_i$  large enough, from the above estimates we deduce \eqref{ecsegundaderivada}. The proof is complete. 
\end{proof}

From Theorem \ref{teoconvex}, the following result holds.
\begin{corollary}
 Under the conditions of Theorem \ref{teoconvex},  the functionals $J_i$, $i = 1, 2$, given by \eqref{energy} are  convex and the pair $(v^1,v^2)$ is a Nash equilibrium. 
\end{corollary}

\section{ \bf Observability inequality and controllability for linearized system}

In this section, we will determine an observability inequality for the linearized system corresponding to the system \eqref{to1}. This inequality is the most important approach to analyzing controllability. In that sense, we consider the linearized system defined by 
\begin{align} \label{lin1}
	\begin{cases}
		y_{1,t}-a(0)y_{1,xx}+F_1'(0)y_1  = G_1+h\chi_\omega-\frac{1}{\mu_1}\rho_*^{-2}p_1^1\chi_{\omega_1}-\frac{1}{\mu_2}\rho_*^{-2}p_1^2\chi_{\omega_2}, &\text{ in } Q,\\
		y_{2,t}-a(0)y_{2,xx}+F_2'(0)y_2  +cy_1 = G_2, &\text{ in } Q,\\ 
		-p_{1,t}^1-a(0)p_{1,xx}^1+F_1'(0)p_1^1+cp_2^1=G_3+\alpha_1 y_1 \chi_{\omega_{1,d}}, &\text{ in } Q,\\
		-p_{1,t}^2-a(0)p_{1,xx}^2+F_1'(0)p_1^2+cp_2^2=G_4+\alpha_2 y_1\chi_{\omega_{2,d}},&\text{ in } Q,\\
		-p_{2,t}^1-a(0)p_{2,xx}^1+F_2'(0)p_2^1=G_5+\alpha_1 y_2\chi_{\omega_{1,d}}, &\text{ in } Q,\\
		-p_{2,t}^2-a(0)p_{2,xx}^2+F_2'(0)p_2^2=G_6+\alpha_2 y_2 \chi_{\omega_{2,d}}, &\text{ in } Q,\\
  		y_1(0,t) = y_1(L,t) =0,\,  y_2(0,t) = y_2(L,t) = 0, &\text{ on } (0,T),\\ 
		p_1^i(0,t) = p_1^i(L,t) =0,\,  p_2^i(0,t) = p_2^i(L,t) = 0, \, i=1,2,&\text{ on } (0,T),\\ 
		y_1(x,0) = f(x),\,  y_2(x,0) = g(x), &\text{ in } I,\\
		p_1^i(x,T) =0,\,  p_2^i(x,T) =0, \, i=1,2,&\text{ in } I,
	\end{cases}		
\end{align}where $G_i \in L^2(Q)$, $i = 1, \ldots, 6$.

Throughout this section, we assume that $\omega_{1,d}=\omega_{2,d}$,  and denote this identity by
\begin{equation}\label{ec9}	\omega_{d}:=\omega_{1,d}=\omega_{2,d}.
\end{equation}

First, we will introduce some weight functions that will be relevant to the study of controllability. The following  result, due to Fursikov and Imanuvilov in \cite{FurImanu}, is fundamental.
\begin{lemma}
	There is a satisfying $\eta_0 \in C^2(\bar{I})$ function:
	\begin{equation} \label{imafur}
    		\begin{cases}
			\eta_0(0)=\eta_0(L)=0, \\
			\eta_0(x)>0 \text{  in  }  I, \\
			|\eta_0'(x)|>0 \text{  in  } I - \omega_0,
		\end{cases}
	\end{equation}    
	where $\omega_0 \subset \subset \omega$.
\end{lemma}Now, we consider the function $\ell \in C^{\infty}(0,T)$, defined by
\begin{align}\label{llll_fun}
    \begin{cases}
	\ell(t)\geq \frac{T^2}{4}, &\text{ in }[0,\frac{T}{2}],\\
	\ell(t)= t(T-t), &\text{ in }[\frac{T}{2},T].
    \end{cases}
\end{align}Thus , for $\lambda>0$, let us define the following functions:
\begin{equation} \label{alfa}
\alpha(x,t):=\frac{e^{4\lambda|\eta_0|_\infty}-e^{\lambda (2|\eta_0|_\infty+\eta_0(x))}}{t(T-t)} =\frac{\overline{\alpha}(x)}{t(T-t)}.
\end{equation}
\begin{equation}\label{alfa-2}
	\tilde{\alpha}(x,t):=\frac{e^{4\lambda|\eta_0|_\infty}-e^{\lambda (2|\eta_0|_\infty+\eta_0(x))}}{\ell(t)}.
\end{equation}
    \begin{equation}
    \xi(x,t):=\frac{e^{\lambda (2|\eta_0|_\infty+\eta_0(x))}}{t(T-t)} .
\end{equation}
\begin{equation}
    \tilde{\xi}(x,t):=\frac{e^{\lambda (2|\eta_0|_\infty+\eta_0(x))}}{\ell(t)} .
\end{equation}Moreover, we denote by 
\begin{equation}\label{minmaxminn}
	\begin{cases}
    \beta_1:=\underset{x \in \bar{I}}{min}\text{  }\overline{\alpha}(x,t)= e^{4\lambda|\eta_0|_\infty}-e^{3\lambda |\eta_0|_\infty},\qquad  
    &\beta_2:=\underset{x \in \bar{I}}{max}\text{  }\overline{\alpha}(x,t)=e^{4\lambda|\eta_0|_\infty}-e^{2\lambda|\eta_0|_\infty}\\
    \displaystyle\alpha^{\circ}(t):=\min_{x\in I}\alpha(x,t),
     &\displaystyle \alpha^*(t):=\max_{x\in I}\alpha(x,t),\\
    \displaystyle\tilde{\alpha}^{\circ}(t):=\min_{x\in I}\tilde{\alpha}(x,t),
     &\displaystyle \tilde{\alpha}^*(t):=\max_{x\in I}\tilde{\alpha}(x,t).
    \end{cases}
\end{equation}

\begin{remark}
	From \eqref{llll_fun}-\eqref{minmaxminn}, we have that
\begin{equation}\label{pesos-0}
	 e^\frac{s\beta_1}{\ell(t)}<e^{s\tilde{\alpha}}<e^\frac{s\beta_2}{\ell(t)},\quad \forall s>0.
	 \end{equation}
\end{remark}

\begin{remark}
For  $\lambda>0$  large enough,  we obtain that

\begin{equation} \label{pesos-1}
    	\left\{\begin{array}{ll}
    \displaystyle 	\frac{4}{5}\tilde{\alpha}^*(t)\leq \tilde{\alpha}(x,t)\leq \tilde{\alpha}^*(t),\,\,\, &\forall\,\, (x,t)\in I\times(0,T),\\
    \displaystyle 	\tilde{\alpha}^{\circ}(t)\leq \tilde{\alpha}(x,t)\leq \frac{5}{4}\tilde{\alpha}^{\circ}(t),\,\,& \forall\, (x,t)\in I\times(0,T),\\
    \displaystyle\tilde{\alpha}^*(t)<\frac{5}{4}\tilde{\alpha}^{\circ}(t),\,\,&\forall\,\, t\in (0,T).
    \end{array}\right.
\end{equation}

\end{remark}

\noindent In order to state the main result for this section, we  define the weight functions
\begin{equation} \label{pesos}
\begin{cases}
    \rho = e^{s\alpha},\\
    \tilde{\rho}=e^{s\tilde{\alpha}},\\
    \rho_*(t) = e^{s\alpha^*},\\
  \tilde{\rho}_*(t) = e^{s\tilde{\alpha}^*}.
\end{cases}
\end{equation}Therefore, $\rho_*(t)$ is a non-decreasing strictly positive function blowing up at $t=T$, and we consider
\begin{equation} \label{pesos2}
\begin{cases}
    \rho_0 = e^{4s\tilde{\alpha}^{*}/5}, \, \, \rho_1 = e^{4s\tilde{\alpha}^{*}/5}(s\tilde{\gamma})^{-10},\\
    \Hat{\rho}_0 = e^{3s\tilde{\alpha}^{*}/5},\, \, \Hat{\rho}_1 = e^{3s\tilde{\alpha}^{*}/5}(s\tilde{\gamma})^{-1}.
\end{cases}
\end{equation}
where $\gamma=\dfrac{1}{t(T-t)},\ \tilde{\gamma}=\dfrac{1}{\ell(t)}$.
\begin{remark}
From \eqref{llll_fun}-\eqref{pesos2}, there exist a constant $C>0$, such that
\begin{equation}\label{weight-linear1}
\begin{cases}
 \hat{\rho}_{0}\leq C \rho_{1},\\ 
 |\hat{\rho}_{0,t}|\leq C \rho_{0},\\ 
 \hat{\rho}_{0}\leq C \rho_{0}, \\   
 |\hat{\rho}_{1,t}\hat{\rho}_{1}|\leq C \hat{\rho}^{2}_{0},\\
 \hat{\rho}_{1}\leq C\hat{\rho}_{0}.
\end{cases}
\end{equation}Moreover, from \eqref{weight-linear1}, the weight functions satisfy
\begin{equation}\label{eq:compara_rhos3}
	\hat\rho_1 \leq C\hat\rho_0 \leq  C\rho_{1}\leq C\rho_{0}\leq C\tilde \rho, \qquad 
    \text{and} \qquad \tilde \rho \leq C \hat\rho_1^2.   
\end{equation}

\end{remark}

Thus, the main result of this section is as follows.

\begin{theorem} 
\label{teo:linearized_control}
 Suppose that the conditions \eqref{condiciones_a4}-\eqref{condiciones_a4c(x,t)} and  \eqref{ec9} are satisfied,  with $\omega_d\cap \omega\neq \emptyset$, $\tilde{\rho} G_i\in L^2(Q), i=1,\ldots,6$ and the constants $\mu_i, i=1,2$, are large enough. Then,  the system \eqref{lin1} is null-controllable.  More precisely, for any $(f, g) \in  [H_0^1(I)]^2$, there exists a control-state $(y_1, y_2,p_1^1,p_1^2,p_2^1,p_2^2,h)$ satisfying
\begin{equation}\label{est-linear1}
    \begin{cases}
        \rho_1 h \in L^2(\omega\times(0,T)),\\
        \left(\rho_0 y_1, \rho_0 y_2, \rho_0 p_1^1, \rho_0 p_1^2, \rho_0 p_2^1,\rho_0 p_2^2\right) \in [L^2(Q)]^6,
    \end{cases}
\end{equation}    
\begin{equation}\label{est-linear2}
    \begin{split}
&\displaystyle  N_1(y_1, y_2,p_1^1,p_1^2,p_2^1,p_2^2) =\sum^{2}_{i=1}\int_{Q} \hat{\rho}^{2}_{0}| y_{i,x}|^{2}\,dxdt+\sum^{2}_{i,j=1}\int_{Q}\hat{\rho}^{2}_{0}| p^{i}_{j,x}|^{2}\,dxdt\\
&\displaystyle \leq C\left(   \int_{\omega\times(0,T)}\rho^{2}_{1}|h|^{2}\,dxdt+\sum^{6}_{i=1}\int_{Q}\tilde{\rho}^{2}|G_{i}|^{2}\,dxdt+\sum^{2}_{i,j=1}\int_{Q}\rho^{2}_{0}|p^{i}_{j}|^{2}\,dxdt \right.\\
&\,\,\,\,\,\,\,\left.+\int_{I}|f|^{2}\,dx+\int_{I}|g|^{2}\,dx+\sum^{2}_{i=1}\int_{Q}\rho^{2}_{0}|y_{i}|^{2}\,dxdt \right),
\end{split}
\end{equation}and
\begin{equation}\label{est-linear3}
\begin{split}
&\displaystyle  N_2(y_1, y_2,p_1^1,p_1^2,p_2^1,p_2^2) =\sum^{2}_{i=1}\sup_{t\in [0,T]}\int_{I} \hat{\rho}^{2}_{1}| y_{i,x}|^{2}\,dx+\sum^{2}_{i,j=1}\sup_{t\in [0,T]}\int_{I}\hat{\rho}^{2}_{1}| p^{i}_{j,x}|^{2}\,dx\\
& \sum^{2}_{i=1}\int_{Q}\hat{\rho}^{2}_{1}(| y_{i,xx}|^{2}+|y_{i,t}|^{2})\,dxdt
+ \sum^{2}_{i,j=1}\int_{Q}\hat{\rho}^{2}_{1}(| p^{i}_{j,xx}|^{2}+|p^{i}_{j,t}|^{2})\,dxdt\\
&\displaystyle \leq C\left(   \int_{\omega\times(0,T)}\rho^{2}_{1}|h|^{2}\,dxdt+\sum^{6}_{i=1}\int_{Q}\tilde{\rho}^{2}|G_{i}|^{2}\,dxdt+\sum^{2}_{i,j=1}\int_{Q}\rho^{2}_{0}|p^{i}_{j}|^{2}\,dxdt \right.\\
&\,\,\,\,\,\,\,\left.+\int_{I}|f_{x}|^{2}\,dx+\int_{I}|g_{x}|^{2}\,dx+\sum^{2}_{i=1}\int_{Q}\rho^{2}_{0}|y_{i}|^{2}\,dxdt \right).
\end{split}
\end{equation}

\end{theorem}

Now, we will present some important results that will allow us to prove Theorem \ref{teo:linearized_control}.

\subsection{Carleman estimates}
In this section we establish Carleman's inequality and observability inequality that allows us to prove the null controllability of \eqref{lin1}. We start by proving an Carleman's inequality for a single parabolic equation, after proving an observability inequality for the adjoint systems associated to the linearized version \eqref{lin1} of \eqref{to1}.

\subsubsection{Carleman's inequality for single parabolic equation}
The starting point is a well-known global Carleman's inequality for solutions to parabolic equations. First, let us consider a single parabolic equation
\begin{align} \label{eq par}
		\begin{cases}
			z_{t}-a(0)z_{xx}= F, &\text{ in } Q,\\
   			z(0,t) = z(L,t) =0,  &\text{ on } (0,T),\\
			z(x,0) = z^0, &\text{ in } I,
		\end{cases}		
	\end{align}where $F\in L^2(Q)$ and $z^0\in L^2(I)$.

\begin{definition}
	For $m \in \mathbb{R}$, we define
        \begin{align*}
           	    	&I(m,z) := \int_Q \rho^{-2}(s\gamma)^{m-2}|z_x|^2dxdt +\int_Q \rho^{-2}(s\gamma)^{m}|z|^2dxdt\\
	    		&\tilde{I}(m,z) := \int_Q \tilde{\rho}^{-2}(s\tilde{\gamma})^{m-2}|z_x|^2dxdt +\int_Q \tilde{\rho}^{-2}(s\tilde{\gamma})^{m}|z|^2dxdt,   
	\end{align*}where $z$ is the solution of the equation \eqref{eq par}.
\end{definition}The following result, due to Imanuvilov and Yamamoto (see \cite{Imanu-Yama}), provides us a global Carleman inequality for the solutions of equation \eqref{eq par}.

\begin{lemma}[First Carleman's inequality]\label{lema carleman 1}
Let $\omega_0\subset\subset I$ be a nonempty open subset. For any $m\in \mathbb{R}$, there exist constants $s_m>0$ e $C_m>0$ such that, for any $s>s_m$ and every $z^0 \in L^2(I)$, the
solution $z$ of \eqref{eq par} satisfies
    \begin{equation} \label{carleman1}
        I(m,z)\leq C_m\left( \int_{\omega_0 \times (0,T)} \rho^{-2}(s\gamma)^m\vert z \vert^2 dxdt +\int_Q \rho^{-2}(s\gamma)^{m-3}\vert F\vert^2dxdt \right),
    \end{equation}  where $C_m = C_m(I,\omega_0,m)$ and $s_m=(T+T^2)\sigma_m$, with $\sigma_m = \sigma_m(I,\omega_0,m)$.
\end{lemma}

\subsubsection{Adjoint system and Global Carleman’s inequality }
Now, we are in a position to prove the controllability for linearized system \eqref{lin1}. To achieve this purpose,  we consider the adjoint system for \eqref{lin1}, defined by 
\begin{align} \label{ad1}
		\begin{cases}
			-\varphi_{1,t}-a(0)\varphi_{1,xx}+F_1'(0)\varphi_1+c\varphi_2   = H_1+\alpha_1\psi_1^1\chi_{\omega_{1,d}} +\alpha_2\psi_1^2\chi_{\omega_{2,d}},&\text{ in } Q,\\
			-\varphi_{2,t}-a(0)\varphi_{2,xx}+F_2'(0)\varphi_2  = H_2+\alpha_1\psi_2^1\chi_{\omega_{1,d}} +\alpha_2\psi_2^2\chi_{\omega_{2,d}},&\text{ in } Q,\\
			\psi_{1,t}^1-a(0)\psi_{1,xx}^1+F_1'(0)\psi_1^1=H_3-\frac{1}{\mu_1}\rho_*^{-2}\varphi_1\chi_{\omega_1},&\text{ in } Q, \\
			\psi_{1,t}^2-a(0)\psi_{1,xx}^2+F_1'(0)\psi_1^2=H_4-\frac{1}{\mu_2}\rho_*^{-2}\varphi_1\chi_{\omega_2},&\text{ in } Q,\\
			\psi_{2,t}^1-a(0)\psi_{2,xx}^1+F_2'(0)\psi_2^1+c\psi_1^1=H_5, &\text{ in } Q,\\
			\psi_{2,t}^2-a(0)\psi_{2,xx}^2+F_2'(0)\psi_2^2+c\psi_1^2=H_6, &\text{ in } Q,\\
   			\varphi_1(0,t) = \varphi_1(L,t) =0, \varphi_2(0,t) = \varphi_2(L,t) = 0, &\text{ on } (0,T),\\
			\psi_1^i(0,t) = \psi_1^i(L,t) =0, \psi_2^i(0,t) = \psi_2^i(L,t) = 0,\, i=1,2&\text{ on } (0,T),\\
			\varphi_1(x,T) =\varphi_1^T(x), \, \varphi_2(x,T) =\varphi_2^T(x),&\text{ in } I,\\
			\psi_1^i(x,0) = 0, \, \psi_2^i(x,0) = 0, \, i=1,2,&\text{ in } I,
		\end{cases}		
	\end{align}where $\varphi_1^T, \varphi_2^T \in L^2(I)$ and $H_i \in L^2(Q)$, $i = 1,\ldots, 6$.

    \noindent Taking into account\eqref{ec9}, and if we set $\varrho_i$,  $i=1,2$, given by $\varrho_i:=\alpha_1\psi_i^1 + \alpha_2\psi_i^2$, the system \eqref{ad1} is equivalent to 

    \begin{align} \label{ad2}
		\begin{cases}
			-\varphi_{1,t}-a(0)\varphi_{1,xx}+F_1'(0)\varphi_1+c\varphi_2   =H_1+ \varrho_1\chi_{w_{d}},&\text{ in } Q,\\
			-\varphi_{2,t}-a(0)\varphi_{2,xx}+F_2'(0)\varphi_2  = H_2+ \varrho_2\chi_{\omega_{d}},&\text{ in } Q,\\
			\varrho_{1,t}-a(0)\varrho_{1,xx} + F_1'(0)\varrho_1   =\overline{H}_3 -\rho_*^{-2}(\frac{\alpha_1}{\mu_1}\chi_{\omega_1}+\frac{\alpha_2}{\mu_2}\chi_{\omega_2})\varphi_{1}, &\text{ in } Q,\\ 
			\varrho_{2,t}-a(0)\varrho_{2,xx}+F_2'(0)\varrho_2+ c\varrho_1 = \overline{H}_4, &\text{ in } Q,\\ 	 
			\varphi_1(0,t) = \varphi_1(L,t) =0,\,  \varphi_2(0,t) = \varphi_2(L,t) = 0, &\text{ on } (0,T),\\
			\varrho_1(0,t) = \varrho_1(L,t) =0,\,  \varrho_2(0,t) = \varrho_2(L,t) = 0, &\text{ on } (0,T),\\
			\varphi_1(x,T) =\varphi_1^T(x), \, \varphi_2(x,T) =\varphi_2^T(x),&\text{ in } I,\\
			\varrho_1(x,0) = 0, \,\varrho_2(x,0) = 0, &\text{ in } I.\\
		\end{cases}		
	\end{align}

The following result provides us with a first global Carleman estimate for adjoint system \eqref{ad2}.

\begin{proposition}\label{first_carlmeanpropo111}
    Suppose that  \eqref{condiciones_a4}-\eqref{condiciones_a4c(x,t)} and  \eqref{ec9} are satisfied    and $\mathcal{O} =\omega_d\cap \omega \neq \emptyset$. Then, for an adequate selection of parameters $d_i,\, q_i \in \mathbb{R}$,  and $\mu_i$  large enough, $i=1,2$, there exist two positive constants $C$ and $\hat{\sigma}_3$ such that, for any  $(\varphi_1^T,\varphi_2^T)\in [L^2(I)]^2$, the solution $(\varphi_{1},\varphi_{2},\varrho_{1},\varrho_{2})$  of  system \eqref{ad2} satisfies
    
\begin{equation}\label{carleman2}
    \begin{split}
         	& I(d_1,\varphi_1) + I(d_2,\varphi_2) +I(q_1,\varrho_1) + I(q_2,\varrho_2)\\
         	&	 \leq C\left(\int_{\omega\times (0,T)}\rho^{-2}(s\gamma)^{q_3}\vert \varphi_1 \vert^2dxdt +\int_{Q}\rho^{-2}(s\gamma)^{2q_{2}+8-d_{2}}\vert H_1\vert^2dxdt \right.\\
         	&	  + \int_{Q}\rho^{-2}(s\gamma)^{q_2}\vert H_2\vert^2dxdt+ \int_{Q}\rho^{-\frac{8}{5}}_{*} \vert \overline{H}_3\vert^2  dxdt\\
         	& \left.  +\int_{Q}\rho^{-2}(s\gamma)^{q_2-3} \vert \overline{H}_4\vert^2  dxdt \right),
    \end{split}    
    \end{equation}
    for all  $ \,s\geq s_3 = (T+T^2)\hat{\sigma}_3$\,  with $\hat{\sigma}_{3}$ depending on  $I,\, \omega_{0},\, M_{0},\,  F'_i(0),\,  d_{i}$, $\alpha_{i}$,  and $q_{i}$.
\end{proposition}

\begin{proof}
Since $\mathcal{O}\neq \emptyset$, there exists a non-empty open subset $\omega_0 \subset \subset \mathcal{O}$, such that the functions $\eta_0$, $\alpha$ and the  weight function $\rho$, given by  \eqref{imafur},  \eqref{alfa} and  \eqref{pesos}, are well-defined. 

We make the rest of the proof in five steps:

\textbf{Step 1:} Applying inequality \eqref{lema carleman 1} of Lemma \eqref{lema carleman 1} to each $\varphi_i$ and $\varrho_i$, $i=1,2,$ solutions of system \eqref{ad2}, for some different real numbers $d_i$ and $q_i$ to be chosen later, we add such inequalities and obtain the following  estimate:
\begin{equation}\label{pre carleman}
\begin{split}
  &I(d_1,\varphi_1) +I(d_2,\varphi_2)+I(q_1,\varrho_1)+I(q_2,\varrho_2)\\
    &\leq C\left( \sum_{i=1}^2 \int_{\omega_0 \times (0,T)} \rho^{-2}(s\gamma)^{d_i}|\varphi_i|^2 dxdt +\sum_{i=1}^2 \int_{\omega_0 \times (0,T)} \rho^{-2}(s\gamma)^{q_i}|\varrho_i|^2 dxdt \right.\\
   & +\int_Q\rho^{-2}(s\gamma)^{d_1-3}\vert \varrho_1\chi_{\omega_{d}} \vert^2 dxdt+\int_Q\rho^{-2}(s\gamma)^{d_2-3}\vert \varrho_2\chi_{\omega_{d}} \vert^2 dxdt \\
    & +\int_Q\rho^{-2}(s\gamma)^{d_1-3} |c|^2_\infty |\varphi_2|^2dxdt + \int_Q\rho^{-2}(s\gamma)^{q_2-3} |c|^2_\infty |\varrho_1^2dxdt \\
  &+\int_Q\rho^{-2}(s\gamma)^{q_1-3} \rho_*^{-4}\left|\frac{\alpha_1}{\mu_1}\varphi_1\chi_{\omega_1}+ \frac{\alpha_2}{\mu_2}\varphi_1\chi_{\omega_2}\right|^2  dxdt \\
  &+\int_Q\rho^{-2}(s\gamma)^{d_1-3} \vert H_1\vert^2  dxdt +\int_Q\rho^{-2}(s\gamma)^{d_2-3} \vert H_2\vert^2  dxdt \\
  &\left. +\int_Q\rho^{-2}(s\gamma)^{q_1-3} |\overline{H}_3|^2  dxdt +\int_Q\rho^{-2}(s\gamma)^{q_2-3} |\overline{H}_4|^2  dxdt \right),   
\end{split}
\end{equation} for all  $s\geq s_0 := (T+T^2)\hat{\sigma}_{0}$, where $C$ and $\hat{\sigma}_{0}$ are two positive constants depending on  $I$, $\omega_0$, $F_i$, $d_i$ and  $q_i$, $i=1,2$. 

\textbf{Step 2:} Now, we will try to incorporate as many terms from the right side of \eqref{pre carleman} into the terms on the left side as possible. To achieve this, it is necessary to correctly choose the terms $d_i$ and  $q_i$. We can observe that if we establish
\begin{equation*}
    d_1-3<q_1<d_1+3\text{ and }d_2-3<q_2,
\end{equation*}from \eqref{pre carleman}, we deduce that
\begin{equation}\label{pre carleman_222}
\begin{split}
  &I(d_1,\varphi_1) +I(d_2,\varphi_2)+I(q_1,\varrho_1)+I(q_2,\varrho_2)\\
    &\leq C\left( \sum_{i=1}^2 \int_{\omega_0 \times (0,T)} \rho^{-2}(s\gamma)^{d_i}|\varphi_i|^2 dxdt +\sum_{i=1}^2 \int_{\omega_0 \times (0,T)} \rho^{-2}(s\gamma)^{q_i}|\varrho_i|^2 dxdt \right.\\
    & +\int_Q\rho^{-2}(s\gamma)^{d_1-3} |c|^2_\infty |\varphi_2|^2dxdt + \int_Q\rho^{-2}(s\gamma)^{q_2-3} |c|^2_\infty |\varrho_1^2dxdt \\
  &+\int_Q\rho^{-2}(s\gamma)^{d_1-3} \vert H_1\vert^2  dxdt +\int_Q\rho^{-2}(s\gamma)^{d_2-3} \vert H_2\vert^2  dxdt \\
  &\left. +\int_Q\rho^{-2}(s\gamma)^{q_1-3} |\overline{H}_3|^2  dxdt +\int_Q\rho^{-2}(s\gamma)^{q_2-3} |\overline{H}_4|^2  dxdt \right).  
\end{split}
\end{equation}On the other hand, since $\gamma(t)^{-1} \leq \frac{T^2}{4}$ on $(0,T)$, we choose
\begin{equation*}
    d_1-3<d_2 \text{ and } q_2-3<q_1,
\end{equation*} Therefore, for all $ s\geq s_1=(T+T^{2})\hat{\sigma}_{1}$, with positive constants $C$ and $\hat{\sigma}_1$,  depending on $I$, $\omega_0$, $d_i$, $q_i$, $\alpha_i$, $|c|_\infty$ and  $|F'_i(0)|$, from \eqref{pre carleman_222}, we get

\begin{equation}\label{rett}
\begin{split}
  &I(d_1,\varphi_1) +I(d_2,\varphi_2)+I(q_1,\varrho_1)+I(q_2,\varrho_2)\\
    &\leq C\left( \sum_{i=1}^2 \int_{\omega_0 \times (0,T)} \rho^{-2}(s\gamma)^{d_i}|\varphi_i|^2 dxdt +\sum_{i=1}^2 \int_{\omega_0 \times (0,T)} \rho^{-2}(s\gamma)^{q_i}|\varrho_i|^2 dxdt \right.\\
  &+\int_Q\rho^{-2}(s\gamma)^{d_1-3} \vert H_1\vert^2  dxdt +\int_Q\rho^{-2}(s\gamma)^{d_2-3} \vert H_2\vert^2  dxdt \\
  &\left. +\int_Q\rho^{-2}(s\gamma)^{q_1-3} |\overline{H}_3|^2  dxdt +\int_Q\rho^{-2}(s\gamma)^{q_2-3} |\overline{H}_4|^2  dxdt \right).  
\end{split}
\end{equation}

\textbf{Step 3:}  In this step, let us eliminate the terms of $\varrho_i$, $i=1,2$, of the above estimate. Given a set $\mathcal{O}_1$ such that $\omega_0 \subset \subset \mathcal{O}_1 \subset \subset \mathcal{O}$, we consider a function $\xi \in C^\infty (\mathbb{R})$ that verifies:
\begin{equation*}
    \begin{cases}
        0\leq \xi \leq 1, &\text{ if }x \in\mathbb{R}, \\
        \xi = 1, &\text{ if } x \in \omega_0, \\
        \text{supp}(\xi) \subset \mathcal{O}_1,
    \end{cases}
\end{equation*} and
\begin{equation*}
    \frac{\xi_{xx}}{\xi^\frac{1}{2}}, \, \, 
    \frac{\xi_x}{\xi^\frac{1}{2}}\in L^\infty(I).
\end{equation*}Moreover, we set $u_i: = \rho^{-2}(s\gamma)^{q_i}$, $i=1,2$. Thus, if we multiply the first equation in \eqref{ad2} by $u_1\xi\varrho_1$, the second one  by $u_2\xi\varrho_2$ and integrate on $Q$, it follows that
\begin{align*}
\sum_{i=1}^{2}\int_{\omega_0\times (0,T)} \rho^{-2} (s \gamma)^{q_i} |\varrho_i|^2\,dxdt&\leq \sum_{i=1}^{2}\int_{Q} \rho^{-2} (s \gamma)^{q_i} |\varrho_i|^2\chi_{\omega_{d}}\,dxdt \\
&  = \int_{Q} u_1\xi\rho_1(-\varphi_{1,t}-a(0)\varphi_{1,xx}+F_1'(0)\varphi_1+c\varphi_2 )  dxdt\\
&+\int_{Q}u_2\xi\rho_2(	-\varphi_{2,t}-a(0)\varphi_{2,xx}+F_2'(0)\varphi_2) dxdt\\
& -\int_{Q}u_1\xi\varrho_1H_1dxdt-\int_{Q}u_2\xi\varrho_2H_2dxdt.
\end{align*}Consequently, by integrating the right-hand side of the previous inequality by parts, using $\eqref{ad2}_3$ and $\eqref{ad2}_4$, and applying Young's inequality,  we have that
\begin{eqnarray*}
&& \sum_{i=1}^{2}\int_{\omega_0\times (0,T)} \rho^{-2} ( s \gamma)^{q_i} |\varrho_i|^2\\
&&\leq \varepsilon  \left(I(q_1,\varrho_1) +  I(q_2,\varrho_2)\right)+C_{\varepsilon}\int_{\mathcal{O}_{1}\times(0,T)}\rho^{-2}(s\gamma)^{q_{1}}|\varphi_{2}|^{2}\,dxdt\\
&&
 +C_{\varepsilon}\left( \int_{Q}\rho^{-2}(s\gamma)^{q_{1}-3}|\overline{H}_{3}|^{2}\,dxdt+\int_{Q}\rho^{-2}(s\gamma)^{q_{2}-3}|\overline{H}_{4}|^{2}\,dxdt  \right)     \\
&& + C_{\varepsilon} \left(\int_{\mathcal{O}_1\times (0,T)}\rho^{-2}(s\gamma)^{q_1+4}\vert \varphi_1 \vert^2dxdt+\int_{\mathcal{O}_1\times (0,T)}\rho^{-2}(s\gamma)^{q_2+4}\vert \varphi_2 \vert^2dxdt\right)\\
&& + C_{\varepsilon}\left(\int_{\mathcal{O}_1\times (0,T)}\rho^{-2}(s\gamma)^{q_1}\vert H_1\vert^2dxdt+\int_{\mathcal{O}_1\times (0,T)}\rho^{-2}(s\gamma)^{q_2}\vert H_2\vert^2dxdt\right),
\end{eqnarray*}for any $\varepsilon>0$, and some positive constant $C_{\varepsilon}$.

\noindent Furthermore, choosing $q_{1}<q_{2}+4$, from estimative above, we obtain
\begin{equation}\label{desiepaucaux}
    \begin{split}
       & \sum_{i=1}^{2}\int_{\omega_0\times (0,T)} \rho^{-2} \left( s \gamma)^{q_i} |\varrho_i|^2 \leq  \varepsilon  (I(q_1,\varrho_1) +  I(q_2,\varrho_2)\right) \\
	& + C_{\varepsilon} \left(\int_{\mathcal{O}_1\times (0,T)}\rho^{-2}(s\gamma)^{q_1+4}\vert \varphi_1 \vert^2dxdt+\int_{\mathcal{O}_1\times (0,T)}\rho^{-2}(s\gamma)^{q_2+4}\vert \varphi_2 \vert^2dxdt\right)\\
	&
	+ C_{\varepsilon}\left( \int_{Q}\rho^{-2}(s\gamma)^{q_{1}-3}|\overline{H}_{3}|^{2}\,dxdt+\int_{Q}\rho^{-2}(s\gamma)^{q_{2}-3}|\overline{H}_{4}|^{2}\,dxdt  \right)   \\
	& + C_{\varepsilon}\left(\int_{\mathcal{O}_1\times (0,T)}\rho^{-2}(s\gamma)^{q_1}\vert H_1\vert^2dxdt+\int_{\mathcal{O}_1\times (0,T)}\rho^{-2}(s\gamma)^{q_2}\vert H_2\vert^2dxdt\right). 
    \end{split}
\end{equation}Then, from \eqref{rett} and \eqref{desiepaucaux}, we deduce that
\begin{equation}\label{rett_22}
\begin{split}
  &I(d_1,\varphi_1) +I(d_2,\varphi_2)+I(q_1,\varrho_1)+I(q_2,\varrho_2)\\
    &\leq C\left( \int_{\mathcal{O}_1\times (0,T)}\rho^{-2}(s\gamma)^{q_1+4}\vert \varphi_1 \vert^2dxdt+\int_{\mathcal{O}_1\times (0,T)}\rho^{-2}(s\gamma)^{q_2+4}\vert \varphi_2 \vert^2dxdt      \right.\\
  &+\int_Q\rho^{-2}(s\gamma)^{d_1-3} \vert H_1\vert^2  dxdt +\int_Q\rho^{-2}(s\gamma)^{d_2-3} \vert H_2\vert^2  dxdt \\
  &\left. +\int_Q\rho^{-2}(s\gamma)^{q_1-3} |\overline{H}_3|^2  dxdt +\int_Q\rho^{-2}(s\gamma)^{q_2-3} |\overline{H}_4|^2  dxdt \right).  
\end{split}
\end{equation}for all $s\geq s_2 = (T+T^2)\hat{\sigma}_{2}$ where $\hat{\sigma}_{2}$ depends on the terms $I$, $\omega_0$, $d_i$,  $F'_{i}(0)$,  $\alpha_{i}$, $M_{0}$,  and  $q_i$.

\textbf{Step 4:} In this step,  following the techniques from step $3$, we will estimate the term $$\int_{\mathcal{O}_1\times (0,T)}\rho^{-2}(s\gamma)^{q_2+4}\vert \varphi_2 \vert^2dxdt,$$ in inequality \eqref{rett_22}. Indeed, let us appropriately select the parameter $s$ to delete this term. Consider a new set $\mathcal{O}_2$ such that $\mathcal{O}_1 \subset \subset\mathcal{O}_2 \subset \subset \mathcal{O} $,  we consider a function $\hat{\xi} \in C^\infty (\mathbb{R})$ verifying:
\begin{equation*}
	\begin{cases}
		0\leq \hat{\xi} \leq 1, \text{ if }x \in\mathbb{R}, \\
		\hat{\xi} = 1 ,\text{ if }x \in \mathcal{O}_1, \\
         \text{supp}(  \hat{\xi}) \subset \mathcal{O}_{2},
	\end{cases}
\end{equation*} and 
\begin{equation*}
	\frac{\hat{\xi}_{xx}}{\hat{\xi}^\frac{1}{2}}, \, 
	\frac{\hat{\xi}_x}{\hat{\xi}^\frac{1}{2}}\in L^\infty(I).
\end{equation*}We define $u: = \rho^{-2} (s\gamma)^{q_2+4}$.  Then, if   we multiply $\eqref{ad2}_1$ by $u\hat{\xi} \varphi_2$, integrate on $Q$ and using \eqref{condiciones_a4c(x,t)},  it follows that
\begin{equation*}
\begin{split}
c_0\int_{\mathcal{O}_{2}\times (0,T)} \rho^{-2} (s \gamma)^{q_2+4} |\varphi_2|^2dxdt &\leq \int_Q u\hat{\xi}c\vert \varphi_2 \vert^2dxdt\\
&=\int_Q u\hat{\xi} \varphi_2(\varphi_{1,t}\,dxdt+a(0)\varphi_{1,xx}-F_1'(0)\varphi_1 )\,dxdt\\
   & +\int_{Q} u\hat{\xi} \varphi_2( H_1+\varrho_1\chi_{\omega_{d}})\,dxdt.
\end{split}
\end{equation*}Once again, by integrating the right-hand side of the above inequality by parts, using $\eqref{ad2}_2$,  applying Young's inequality  and using  the  estimates \eqref{pesos-1},   we get, for any $\varepsilon>0$, the existence of a positive constant $C_{\varepsilon}$, such that
\begin{equation}\label{rrett_3}
    \begin{split}
    &c_0\int_{\mathcal{O}_1\times (0,T)} \rho^{-2} (s \gamma)^{q_2+4} |\varphi_2|^2 \\
     &\leq \varepsilon \left( I(d_1,\varphi_1)+I(d_2,\varphi_2)+I(q_1,\varrho_1)+I(q_2,\varrho_2)\right)\\
&+C_{\varepsilon}\int_{\mathcal{O}_2\times(0,T)} \rho^{-2} (s \gamma)^{q_3} |\varphi_1|^2 dxdt+\frac{1}{2}\int_{Q}\rho^{-\frac{8}{5}}_{*}|\varrho_1|^{2}\,dxdt\\
&+C_{\varepsilon}\int_{Q}e^{-2s\alpha}(s\gamma)^{2q_{2}+8-d_{2}}|H_{1}|^{2}\,dxdt+C_{\varepsilon}\int_{Q}e^{-2s\sigma}(s\gamma)^{q_2}|H_2|^{2}\,dxdt,
    \end{split}
\end{equation}for all $s\geq s_3 = (T+T^2)\hat{\sigma}_{3}$ where $\hat{\sigma}_{3}$ depends on the terms $I$, $\omega_0$, $d_i$,  $q_i$,  $M_{0}$, $\alpha_{i}$,  and  $F'_{i}(0)$, and  $q_{3}=\max\{4q_{2}+24-2d_{2}-d_{1};2q_{2}+12-d_{2}\}$. 

\noindent Thus, from \eqref{rett_22} and \eqref{rrett_3}, for all $s\geq s_3$, the following estimate is achieved:
\begin{equation}\label{step4-1}
\begin{split}
& I(d_1,\varphi_1) + I(d_2,\varphi_2) +I(q_1,\varrho_1) + I(q_2,\varrho_2)\\
&\leq C\left(\int_{\mathcal{O}_2\times (0,T)}\rho^{-2}(s\gamma)^{q_3}\vert \varphi_1 \vert^2dxdt  +\int_Q \rho^{-\frac{8}{5}}_{*}  |\varrho_1|^2dxdt \right. \\
&+\int_{Q}\rho^{-2}(s\gamma)^{2q_{2}+8-d_{2}}\vert H_1\vert^2dxdt + \int_{Q}\rho^{-2}(s\gamma)^{q_2}\vert H_2\vert^2dxdt\\
&\left. + \int_{Q}\rho^{-2}(s\gamma)^{q_1-3} \vert \overline{H}_3\vert^2  dxdt +\int_{Q}\rho^{-2}(s\gamma)^{q_2-3} \vert \overline{H}_4\vert^2  dxdt \right).
\end{split}
\end{equation}

\textbf{Step 5:} Finally, in \eqref{step4-1}, we will estimate the term
\begin{equation*}
\int_{Q}\rho^{-\frac{8}{5}}_{*}|\varrho_{1}|^{2}\,dxdt.
\end{equation*}To achieve this purpose, multiplying   $\eqref{ad2}_{3}$  by  $\rho^{-\frac{8}{5}}_{*}\varrho_{1}$, integrating on $I$, using \eqref{condiciones_a4}, \eqref{condiciones_a4f} and Young's inequality, we have that
\begin{equation*}
    \begin{split}
       \frac{1}{2}\frac{d}{dt}\left(\int_{I}\rho^{-\frac{8}{5}}_{*}|\varrho_{1}|^{2}\,dx \right)+M_{0}\int_{I}\rho^{-\frac{8}{5}}_{*}|\varrho_{1,x}|^{2}\,dx &\leq C \left(\int_{I}-\frac{8}{5}\rho^{-\frac{8}{5}-1}_{*}\left( \rho_{*}\right)_{t}|\varrho_{1}|^{2}\,dx+\int_{I}\rho^{-\frac{8}{5}}_{*}|\overline{H}_{3}|^{2}\,dx \right)\\
       &+C\left(\frac{1}{\mu^{2}_{1}}+\frac{1}{\mu^{2}_{2}}\right)\int_{I}\rho^{-4}_{\ast}\rho^{-\frac{8}{5}}_{*}|\varphi_{1}|^{2}\,dx.
    \end{split}
\end{equation*}Since
\begin{align*}
   -\frac{8}{5}\rho^{-\frac{8}{5}-1}_{*}\left( \rho_{*}\right)_{t}\leq 0, \, \, \, \text{and}\, \, \, \rho^{-4}_{\ast}\rho^{-\frac{8}{5}}_{*}\leq e^{-2s\alpha}(s\gamma)^{d_1}, 
\end{align*}integrating the previous inequality on $(0, T)$, yields
\begin{align}\label{prepre_finafin}
\int_{Q}\rho^{-\frac{8}{5}}_{*}|\varrho_1|^{2}\,dxdt \leq C\left(\frac{1}{\mu^{2}_{1}}+\frac{1}{\mu^{2}_{2}} \right)I(d_{1},\varphi_{1})+  C\int_{Q}\rho^{-\frac{8}{5}}_{*}|\overline{H}_{3}|^{2}\,dxdt.
\end{align}Furthermore, for $\mu_{1}$  and  $\mu_{2}$  sufficiently large, from \eqref{step4-1} and \eqref{prepre_finafin}, we obtain \eqref{carleman2}, for all $s\geq s_3$. This concludes the proof of Theorem. 
    
\end{proof}

 Now, using the previous Proposition, we will prove a Carleman inequality for the adjoint system \eqref{ad2}.

\begin{theorem}\label{Carleman-Theorem-2} 
Suppose that  \eqref{condiciones_a4}-\eqref{condiciones_a4c(x,t)} and  \eqref{ec9} are satisfied and   $\mathcal{O} =\omega_d\cap \omega \neq \emptyset$. Then, for $\mu_i, i=1,2$  large enough, there exists  a positive constant $C$, such that, for every $(\varphi_1^T,\varphi_2^T) \in (L^2(I))^2$, the  solution  $(\varphi_{1},\varphi_{2},\psi^{1}_{1},\psi^{2}_{1},\psi^{1}_{2},\psi^{2}_{2})$  of  system \eqref{ad1} satisfies
\begin{equation}\label{car_for_ad_ptipipri}
    \begin{split}
       & \int_{Q}\tilde{\rho}^{-2}_{*}(s\tilde{\gamma})^{d_1}|\varphi_1|^{2}\,dxdt+\int_{Q}\tilde{\rho}^{-2}_{*}(s\tilde{\gamma})^{d_2}|\varphi_2|^{2}\,dxdt\\
       &+\sum_{i=1}^{2}\int_{I}|\varphi_{i}(0)|^{2}\,dx+\sum_{i,j=1}^{2}\int_{Q}\tilde{\rho}^{-2}_{*}|\psi^{i}_{j}|^{2}\,dxdt\\
       &\leq C\left(\int_{\omega\times (0,T)}\tilde{\rho}^{-\frac{8}{5}}_{*}(s\tilde{\gamma})^{q_3}\vert \varphi_1 \vert^2dxdt +\int_{Q}\tilde{\rho}^{-\frac{8}{5}}_{*}\vert H_1\vert^2dxdt + \int_{Q}\tilde{\rho}^{-\frac{8}{5}}_{*}\vert H_2\vert^2dxdt\right.\\
       &	  + \int_{Q}\tilde{\rho}^{-\frac{8}{5}}_{*}\vert H_3\vert^2dxdt +\int_{Q}\tilde{\rho}^{-\frac{8}{5}}_{*} \vert H_{4}\vert^2  dxdt\\
       & \left.+\int_{Q}\tilde{\rho}^{-\frac{8}{5}}_{*} \vert H_{5}\vert^2  dxdt+\int_{Q}\tilde{\rho}^{-\frac{8}{5}}_{*} \vert H_{6}\vert^2  dxdt\right),
    \end{split}
\end{equation}for all $s\geq s_{3}=(T+T^{2})\hat{\sigma}_{3}$ with $\hat{\sigma}_{3}$ depends on  $I,\, \omega_{0},\, M_{0},\,  F'_i(0),\,  d_{i}$, $\alpha_{i}$,  and $q_{i}$.
\end{theorem}

\begin{proof}
In order to  prove this theorem, we will follow some steps.

\textbf{Step1:} In this  step,  we obtain  the Carleman inequality for the system \eqref{ad2}  with weights $\displaystyle \tilde{\rho}^{-2}(s\tilde{\gamma})^{k} $ and  $\tilde{\rho}^{-\frac{8}{5}}_{*},\, k\in \mathbb{Z}$.

\noindent We  are going  to decompose  the  integrals  $\tilde{I}(d_{i},\varphi_{i})$ and  $\tilde{I}(q_{i},\varrho_{i})$, $i=1, 2$, in this  form:
$$
\int_{Q}=\int_{I\times(0,T/2)}+\int_{I\times(T/2,T)},
$$where   the integrals of $\tilde{I}(d_{i},\varphi_{i})$  in $I\times(0,T/2)$  will be  denoted by  $\tilde{I}_{1}(d_{i},\varphi_{i})$, and  the  integrals in $I\times(T/2,T)$  will  be  denoted by $\tilde{I}_{2}(d_{i},\varphi_{i})$. Similarly, with  the term   $\tilde{I}(q_{i},\varrho_{i})$, we have the terms $\tilde{I}_{1}(q_{i},\varrho_{i})$ and  $\tilde{I}_{2}(q_{i},\varrho_{i})$, $i=1, 2$.

\noindent Since $\displaystyle \tilde{\rho}^{-2}(s\tilde{\gamma})^{k} = \displaystyle {\rho}^{-2}(s{\gamma})^{k}, $ in $I \times(T/2, T)$, we have 
\begin{align}\label{liiss}
    &\tilde{I}_{2}(d_{1},\varphi_1)+\tilde{I}_{2}(d_{2},\varphi_{2})+\tilde{I}_{2}(q_{1},\varrho_1)+\tilde{I}_{2}(q_{2},\varrho_{2}) \\
    &\leq I(d_{1},\varphi_1)+I(d_{2},\varphi_{2})+I(q_{1},\varrho_1)+I(q_{2},\varrho_{2}). \nonumber
\end{align}Thus, from \eqref{liiss}, \eqref{carleman2} and using the properties of  functions $\alpha,\, \gamma,\, \tilde{\alpha}$, and  $\tilde{\gamma}$, we  deduce  that
\begin{equation}\label{step6-c-1}
\begin{split}
& \tilde{I}_{2}(d_{1},\varphi_1)+\tilde{I}_{2}(d_{2},\varphi_{2})+\tilde{I}_{2}(q_{1},\varrho_1)+\tilde{I}_{2}(q_{2},\varrho_{2})\\
&\leq C\left(\int_{\omega\times (0,T)}\tilde{\rho}^{-2}(s\tilde{\gamma})^{q_3}\vert \varphi_1 \vert^2dxdt +\int_{Q}\tilde{\rho}^{-2}(s\tilde{\gamma})^{2q_{2}+8-d_{2}}\vert H_1\vert^2dxdt \right.\\
&+ \int_{Q}\tilde{\rho}^{-2}(s\tilde{\gamma})^{q_2}\vert H_2\vert^2dxdt+ \int_{Q}\tilde{\rho}^{-\frac{8}{5}}_{*} \vert \overline{H}_3\vert^2  dxdt\\
& \left.  +\int_{Q}\tilde{\rho}^{-2}(s\tilde{\gamma})^{q_2-3} \vert \overline{H}_4\vert^2  dxdt \right).
\end{split}
\end{equation}Now, we estimate  the  terms $$\tilde{I}_{1}(d_{1},\varphi_1)+\tilde{I}_{1}(d_{2},\varphi_{2})+\tilde{I}_{1}(q_{1},\varrho_1)+\tilde{I}_{1}(q_{2},\varrho_{2}).$$ To achieve this objective, multiplying   $\eqref{ad2}_{1}$  by $\varphi_1$,   integrating  over $I$ and using the conditions \eqref{condiciones_a4} and \eqref{condiciones_a4f},  we result
\begin{align}\label{step6-c-2}
&-\frac{1}{2}\frac{d}{dt}\left(\int_{I}|\varphi_{1}|^{2}\,dx\right)+M_{0}\int_{I}|\varphi_{1,x}|^{2}\,dx\\
&\leq C\left(\int_{I}|\varphi_{1}|^{2}\,dx+ \int_{I}|\varphi_{2}|^{2}\,dx+\int_{I}|H_1|^{2}\,dx+\int_{I}|\varrho_{1}|^{2}\,dx\right). \nonumber
\end{align}Multiplying \eqref{step6-c-2} by  $e^{2Ct}$ and integrating on  $(\tilde{r},r)$,  with  $0 \leq \tilde{r} \leq \frac{T}{2}\leq r \leq \frac{3T}{4}$, we have
\begin{equation*}
    \int_{I}|\varphi_{1}(\tilde{r})|^{2}\,dx \leq C\int_{I}|\varphi_{1}(r)|^{2}\,dx+ C\int_{0}^{r}\left(\int_{I}|\varphi_{2}|^{2}\,dx+\int_{I}|H_1|^{2}\,dx+\int_{I}|\varrho_{1}|^{2}\,dx   \right)\,ds.
\end{equation*}Moreover, integrating the above estimate, firstly, $\tilde{r}$  on  $(0,T/2)$ and, secondly,  $r$  on  $(T/2,3T/4)$, it follows that 
\begin{align}\label{step6-c-3}
    & \int_{I\times(0,T/2)}|\varphi_{1}|^{2}\,dxdt\\
    &	 \leq  C\left(\int_{I\times(T/2,3T/4)}|\varphi_{1}|^{2}\,dxdt+\int_{I\times(0,3T/4)}(|H_1|^{2}+|\varrho_1|^{2}+|\varphi_2|^{2})\,dxdt \right).\nonumber
\end{align}On the other hand, integrating  \eqref{step6-c-2} on  $(0,r)$,  with  $\frac{T}{2}\leq r \leq   \frac{3T}{4}$,  and  using  \eqref{step6-c-3}, we obtain that
\begin{align}\label{step6-c-4}
&\int_{I\times(0,T/2)}|\varphi_{1,x}|^{2}\,dxdt	\leq C \int_{I\times(0,3T/4)}(|\varphi_1|^{2}+|H_1|^{2}+|\varrho_1|^{2}+|\varphi_2|^{2})\,dxdt.
\end{align}Hence, from \eqref{step6-c-3} and \eqref{step6-c-4}, we obtain that there exists a constant $C>0$, such that
\begin{align}\label{step6-c-5}
 & \int_{I\times(0,T/2)}(|\varphi_{1}|^{2}+|\varphi_{1,x}|^{2})\,dxdt	\\
  &\leq  C\left(\int_{I\times(T/2,3T/4)}|\varphi_{1}|^{2}\,dxdt+\int_{I\times(0,3T/4)}(|\varphi_{2}|^{2}+|H_1|^{2}+|\varrho_{1}|^{2})\,dxdt \right).\nonumber  
\end{align}Performing similar computations for the term 
$$
\int_{I\times(0,T/2)}(|\varphi_{2}|^{2}+|\varphi_{2,x}|^{2}+|\varrho_{i}|^{2}+|\varrho_{i,x}|^{2} )\,dxdt,
$$and joint with \eqref{step6-c-5}, we get
\begin{equation}\label{aux_23aunaun}
\begin{split}
&  \sum_{i=1}^{2}\int_{I\times(0,T/2)}(|\varphi_{i}|^{2}+|\varphi_{i,x}|^{2}+|\varrho_{i}|^{2}+|\varrho_{i,x}|^{2})\,dxdt\\
&\leq C\int_{I\times(T/2,3T/4)}(|\varphi_1|^{2}+|\varrho_{1}|^{2}+|\varrho_2|^{2})\,dxdt\\
&+ C\int_{I\times(0,3T/4)}(|H_1|^{2}+|H_{2}|^{2}+|\overline{H}_{3}|^{2}+|\overline{H}_{4}|^{2} )\,dxdt.
\end{split}
\end{equation}Since
\begin{equation}\label{new_limi_rhos}
\begin{cases} 
 		\displaystyle \tilde{\rho}^{-2}(s\tilde{\gamma})^{k} \asymp C, &\text{ in } \, I \times(0, T/2),\\
        \\
            \displaystyle \tilde{\rho}^{-2}(s\tilde{\gamma})^{k} = \displaystyle {\rho}^{-2}(s{\gamma})^{k}, &\text{ in } I \times(T/2, T),
 	      \end{cases}
    \end{equation}from \eqref{aux_23aunaun} and \eqref{new_limi_rhos}, we conclude that
\begin{equation}\label{s_aux_forr}
\begin{split}
&\tilde{I}_{1}(d_{1},\varphi_1)+\tilde{I}_{1}(d_{2},\varphi_{2})+\tilde{I}_{1}(q_{1},\varrho_1)+\tilde{I}_{1}(q_{2},\varrho_{2})\\
&\leq C\left(\int_{\omega\times (0,T)}\tilde{\rho}^{-2}(s\tilde{\gamma})^{q_3}\vert \varphi_1 \vert^2dxdt +\int_{Q}\tilde{\rho}^{-2}(s\tilde{\gamma})^{2q_{2}+8-d_{2}}\vert H_1\vert^2dxdt \right.\\
&+ \int_{Q}\tilde{\rho}^{-2}(s\tilde{\gamma})^{q_2}\vert H_2\vert^2dxdt+ \int_{Q}\tilde{\rho}^{-\frac{8}{5}}_{*}\vert \overline{H}_3\vert^2  dxdt\\
& \left.  +\int_{Q}\tilde{\rho}^{-2}(s\tilde{\gamma})^{q_2-3} \vert \overline{H}_4\vert^2  dxdt \right).
\end{split}
 \end{equation}Consequently, from \eqref{step6-c-1} and \eqref{s_aux_forr}, we deduce the following estimate
 \begin{equation}\label{step6-c-6}
\begin{split}
& \tilde{I}(d_{1},\varphi_1)+\tilde{I}(d_{2},\varphi_{2})+\tilde{I}(q_{1},\varrho_1)+\tilde{I}(q_{2},\varrho_{2})\\
&\leq C\left(\int_{\omega\times (0,T)}\tilde{\rho}^{-2}(s\tilde{\gamma})^{q_3}\vert \varphi_1 \vert^2dxdt +\int_{Q}\tilde{\rho}^{-2}(s\tilde{\gamma})^{2q_{2}+8-d_{2}}\vert H_1\vert^2dxdt \right.\\
&	  + \int_{Q}\tilde{\rho}^{-2}(s\tilde{\gamma})^{q_2}\vert H_2\vert^2dxdt+ \int_{Q}\tilde{\rho}^{-\frac{8}{5}}_{*} \vert \overline{H}_3\vert^2  dxdt\\
& \left.  +\int_{Q}\tilde{\rho}^{-2}(s\tilde{\gamma})^{q_2-3} \vert \overline{H}_4\vert^2  dxdt \right).
\end{split}
 \end{equation}

 \textbf{Step 2:} Now, we obtain  the  observability inequality  for the system \eqref{ad2}. 

 \noindent Indeed, first, let us integrate  over $(0,t)$  in \eqref{step6-c-2},  with   $\frac{T}{2}\leq t \leq \frac{3T}{4}$, and  after integrating  with respect to $t$  over $(T/2,3T/4)$, we deduce  that 
 \begin{align}\label{varp_cero_1}
 \int_{I}|\varphi_{1}(0)|^{2}\,dx  \leq C \int_{I\times(T/2,3T/4)}|\varphi_1|^{2}\,dxdt+C\int_{I\times(0,3T/4)}(|H_{1}|^{2} +  |\varphi_2|^{2}+|\varrho_1|^{2})\,dxdt.    
 \end{align}A similar calculations shows that
 \begin{align}\label{varp_cero_2}
 \int_{I}|\varphi_{2}(0)|^{2}\,dx \leq C \int_{I\times(T/2,3T/4)}|\varphi_2|^{2}\,dxdt+C\int_{I\times(0,3T/4)}(|H_{2}|^{2} +  |\varphi_2|^{2})\,dxdt.    
 \end{align}Therefore, from \eqref{new_limi_rhos}, \eqref{varp_cero_1} and \eqref{varp_cero_2}, we obtain that
 \begin{equation}\label{step6-c-7}
  \begin{split}
& \int_{I}|\varphi_{1}(0)|^{2}\,dx+\int_{I}|\varphi_{2}(0)|^{2}\,dx\\
&\leq C\left(\int_{\omega\times (0,T)}\tilde{\rho}^{-2}(s\tilde{\gamma})^{q_3}\vert \varphi_1 \vert^2dxdt +\int_{Q}\tilde{\rho}^{-2}(s\tilde{\gamma})^{2q_{2}+8-d_{2}}\vert H_1\vert^2dxdt \right.\\
&	  + \int_{Q}\tilde{\rho}^{-2}(s\tilde{\gamma})^{q_2}\vert H_2\vert^2dxdt+ \int_{Q}\tilde{\rho}^{-\frac{8}{5}}_{*} \vert \overline{H}_3\vert^2  dxdt\\
& \left.  +\int_{Q}\tilde{\rho}^{-2}(s\tilde{\gamma})^{q_2-3} \vert \overline{H}_4\vert^2  dxdt \right).
  \end{split}   
 \end{equation}

 \textbf{Step 3:} Here, we estimate $\displaystyle \sum_{i,j=1}^{2}\int_{Q}\tilde{\rho}^{-2}_{*}|\psi^{i}_{j}|^{2}\,dxdt$.

 \noindent For this purpose, we multiply $\eqref{ad1}_{3}$ by $\tilde{\rho}^{-2}_{*}\psi^{1}_{1}$, integrating  over  $I$ and from \eqref{condiciones_a4} and \eqref{condiciones_a4f}, we get
 \begin{align*}
&\frac{1}{2}\frac{d}{dt}\left( \int_{I}\tilde{\rho}^{-2}_{*}|\psi^{1}_{1}|^{2}\,dx \right)+M_{0}\int_{I}\tilde{\rho}^{-2}_{*}|\psi^{1}_{1,x}|^{2}\,dx\\
& \leq \int_{I}-2\tilde{\rho}^{-3}_{*}(\tilde{\rho}_{*})_{t}|\psi^{1}_{1}|^{2}\,dx+C\int_{I} \tilde{\rho}^{-2}_{*}(|\psi^{1}_{1}|^{2}+|H_{3}|^{2})\,dx\\
&+\frac{C}{\mu^{2}_{1}}\int_{I}\rho^{-2}_{*}\tilde{\rho}^{-2}_{*}|\psi^{1}_1|^{2}\,dx + C\int_{\omega}\rho^{-2}_{*}\tilde{\rho}^{-2}_{*}|\varphi_1|^{2}\,dx 
 \end{align*}Since    $-2\tilde{\rho}^{-3}_{*}(\tilde{\rho}_{*})_{t}\leq 0$, from estimative above,   by applying Gronwall's inequality, it follows that
 \begin{align*}
&\int_{Q}\tilde{\rho}^{-2}_{*}|\psi^{1}_{1}|^{2}\,dxdt \\
&\leq C\left(\frac{1}{\mu^{2}_{1}}\int_{Q}\rho^{-2}_{*}\tilde{\rho}^{-2}_{*}|\psi^{1}_{1}|^{2}\,dxdt+\int_{Q}\tilde{\rho}^{-2}_{*}|H_{3}|^{2}\,dxdt +\int_{\omega\times (0,T)}\rho^{-2}_{*}\tilde{\rho}^{-2}_{*}|\varphi_1|^{2}\,dx\right).
 \end{align*}Thus, from the last inequality,  for $\mu_1>0$ large enough, we have
 \begin{align}\label{psi_1_1limi}
\int_{Q}\tilde{\rho}^{-2}_{*}|\psi^{1}_{1}|^{2}\,dxdt \leq C\left(\int_{Q}\tilde{\rho}^{-2}_{*}|H_{3}|^{2}\,dxdt +\int_{\omega\times (0,T)}\rho^{-2}_{*}\tilde{\rho}^{-2}_{*}|\varphi_1|^{2}\,dx\right).
 \end{align}Observe that, from similar computations, the following estimate holds
 \begin{equation}\label{allpsis_234}
\begin{split}
&\int_{Q}\tilde{\rho}^{-2}_{*}|\psi^{2}_{1}|^{2}\,dxdt + \int_{Q}\tilde{\rho}^{-2}_{*}|\psi^{1}_{2}|^{2}\,dxdt + \int_{Q}\tilde{\rho}^{-2}_{*}|\psi^{2}_{2}|^{2}\,dxdt  \\
& \leq C\left(\int_{Q}\tilde{\rho}^{-2}_{*}|H_{4}|^{2}\,dxdt +\int_{Q}\tilde{\rho}^{-2}_{*}|H_{5}|^{2}\,dxdt \right.\\
&\left.+\int_{Q}\tilde{\rho}^{-2}_{*}|H_{6}|^{2}\,dxdt +\int_{\omega\times (0,T)}\rho^{-2}_{*}\tilde{\rho}^{-2}_{*}|\varphi_1|^{2}\,dx\right).
\end{split}
 \end{equation}Consequently, from \eqref{psi_1_1limi}, \eqref{allpsis_234} and the properties for the weight functions, given by \eqref{pesos-1}, we deduce that
 \begin{equation}\label{totallimi_forpsis12212}
\begin{split}
&\sum_{i,j=1}^{2}\int_{Q}\tilde{\rho}^{-2}_{*}|\psi^{i}_{j}|^{2}\,dxdt \\
&\leq C\left(\int_{\omega\times (0,T)}\tilde{\rho}^{-2}(s\tilde{\gamma})^{q_3}\vert \varphi_1 \vert^2dxdt +\int_{Q}\tilde{\rho}^{-2}(s\tilde{\gamma})^{2q_{2}+8-d_{2}}\vert H_1\vert^2dxdt \right.\\
&	  + \int_{Q}\tilde{\rho}^{-2}(s\tilde{\gamma})^{q_2}\vert H_2\vert^2dxdt +\int_{Q}\tilde{\rho}^{-\frac{8}{5}}_{*} \vert H_{3}\vert^2  dxdt\\
&\left. +\int_{Q}\tilde{\rho}^{-\frac{8}{5}}_{*}\vert H_{4}\vert^2  dxdt+\int_{Q}\tilde{\rho}^{-\frac{8}{5}}_{*} \vert H_{5}\vert^2  dxdt
    +\int_{Q}\tilde{\rho}^{-\frac{8}{5}}_{*} \vert H_{6}\vert^2  dxdt\right).
\end{split}
 \end{equation}

 Finally, from \eqref{step6-c-6}, \eqref{step6-c-7} and \eqref{totallimi_forpsis12212}, we obtain the estimate \eqref{car_for_ad_ptipipri}.

\end{proof}

Now, we will demonstrate the main result of this section, that is, the system \eqref{lin1} is controllable to zero in time $T>0$. For this purpose, let us first  introduce the following definition.

\begin{definition}
Let $G_i \in L^2(Q)$, $i=1, \ldots, 6$, $h \in L^2(\omega\times (0,T))$ and $(f, g) \in  [H_0^1(I)]^2$.  A solution by transposition to the system \eqref{lin1} is defined as a vector-valued function $U= \left (y_1,  y_2,  p_1^1,  p_1^2,  p_2^1, p_2^2\right)^{\top}$ belonging to the space $[L^2(Q)]^6$, such that, for any $H_i \in L^2(Q)$, $i=1, \ldots, 6$,  the following integral identity holds:
\begin{equation}\label{transposition_identity}
\begin{split}
&\int_{Q}y_{1}H_{1}\,dxdt+\int_{Q}y_{2}H_{2}\,dxdt+\int_{Q}p^1_{1}H_{3}\,dxdt\\
&+\int_{Q}p^2_{1}H_{4}\,dxdt+\int_{Q}p^1_{2}H_{5}\,dxdt+\int_{Q}p^2_{2}H_{6}\,dxdt\\
&=\int_{I}f\varphi_{1}(0)\,dx+\int_{I}g\varphi_{2}(0)\,dx\\
&+\int_{Q}\left(G_{1}+h\chi_{\omega}\right)\varphi_{1}\,dxdt+\int_{Q}G_{2}\varphi_{2}\,dxdt+\int_{Q}G_{3}\psi^{1}_{1}\,dxdt\\
&+\int_{Q}G_{4}\psi^{2}_{1}\,dxdt+\int_{Q}G_{5}\psi^{1}_{2}\,dxdt+\int_{Q}G_{6}\psi^{2}_{2}\,dxdt,
\end{split}
\end{equation}where $W=(\varphi_{1},\varphi_{2},\psi^{1}_{1},\psi^{2}_{1},\psi^{1}_{2},\psi^{2}_{2})$ is solution of the adjoint system
\begin{align} \label{ad3333}
		\begin{cases}
			-\varphi_{1,t}-a(0)\varphi_{1,xx}+F_1'(0)\varphi_1+c\varphi_2   = H_1+\alpha_1\psi_1^1\chi_{\omega_{1,d}} +\alpha_2\psi_1^2\chi_{\omega_{2,d}},&\text{ in } Q,\\
			-\varphi_{2,t}-a(0)\varphi_{2,xx}+F_2'(0)\varphi_2  = H_2+\alpha_1\psi_2^1\chi_{\omega_{1,d}} +\alpha_2\psi_2^2\chi_{\omega_{2,d}},&\text{ in } Q,\\
			\psi_{1,t}^1-a(0)\psi_{1,xx}^1+F_1'(0)\psi_1^1=H_3-\frac{1}{\mu_1}\rho_*^{-2}\varphi_1\chi_{\omega_1},&\text{ in } Q, \\
			\psi_{1,t}^2-a(0)\psi_{1,xx}^2+F_1'(0)\psi_1^2=H_4-\frac{1}{\mu_2}\rho_*^{-2}\varphi_1\chi_{\omega_2},&\text{ in } Q,\\
			\psi_{2,t}^1-a(0)\psi_{2,xx}^1+F_2'(0)\psi_2^1+c\psi_1^1=H_5, &\text{ in } Q,\\
			\psi_{2,t}^2-a(0)\psi_{2,xx}^2+F_2'(0)\psi_2^2+c\psi_1^2=H_6, &\text{ in } Q,\\
   			\varphi_1(0,t) = \varphi_1(L,t) =0, \varphi_2(0,t) = \varphi_2(L,t) = 0, &\text{ on } (0,T),\\
			\psi_1^i(0,t) = \psi_1^i(L,t) =0, \psi_2^i(0,t) = \psi_2^i(L,t) = 0,\, i=1,2&\text{ on } (0,T),\\
			\varphi_1(x,T) =0, \, \varphi_2(x,T) =0,&\text{ in } I,\\
			\psi_1^i(x,0) = 0, \, \psi_2^i(x,0) = 0, \, i=1,2,&\text{ in } I. 
		\end{cases}		
	\end{align}
\end{definition}

The next Theorem establishes the existence and uniqueness of solutions for system \eqref{lin1} in the sense of transposition.

\begin{theorem}\label{wellpo_trans_sis}
    Let $G_i \in L^2(Q)$, $i=1, \ldots, 6$, $h \in L^2(\omega\times (0,T))$ and $(f, g) \in  [H_0^1(I)]^2$. Then, there exists a unique solution  $U= \left (y_1,  y_2,  p_1^1,  p_1^2,  p_2^1, p_2^2\right)^{\top} \in \left[C\left([0, T]; L^2(I) \right)\right]^6$ of system \eqref{lin1}, which verifies \eqref{transposition_identity}. 
\end{theorem}
\begin{proof}
The theorem can be proved following the arguments used in \cite{BAPA}, Theorem $2.4$, and therefore the details are omitted.
\end{proof}
Let us prove the  main result of  this section. 
\begin{proof}[Proof of Theorem~\ref{teo:linearized_control}] Let us define  the following space
$$
{\mathcal{P}}=\{ (\varphi_{1}, \varphi_{2}, \psi^{1}_{1}, \psi^{2}_{1},  \psi^{1}_{2}, \psi^{2}_{2})\in C^{3}(\overline{Q})^{6}: \, \varphi_{i}=0,\, \psi^{i}_{j}=0\,\,\mbox{on}\,\, \Sigma,\, \psi^{i}_{j}(0)=0\,\, \mbox{in}\,\,  I \},
$$and we consider the function $\mathcal{B}: {\mathcal{P}}\times {\mathcal{ P}}\mapsto \mathbb{R}$, defined by 
$$\begin{array}{l}
	\mathcal{B}\left( (\varphi_{1},\varphi_{2},\psi^{1}_{1},\psi^{2}_{1},\psi^{1}_{2},\psi^{2}_{2});(\tilde{\varphi}_{1},\tilde{\varphi}_{2},\tilde{\psi}^{1}_{1},\tilde{\psi}^{2}_{1},\tilde{\psi}^{1}_{2}, \tilde{\psi}^{2}_{2}) \right)\\
	\noalign{\smallskip}
	\displaystyle 	= \int_{Q} \rho^{-2}_{0}\left(L^{*}_{1}\varphi_{1}+c\varphi_{2}-\alpha_{1}\psi^{1}_{1}\chi_{\omega_{d}}-\alpha_{2}\psi^{2}_{1}\chi_{\omega_{d}} \right) \left( L^{*}_{1}\tilde{\varphi}_{1}+c\tilde{\varphi}_{2}-\alpha_{1}\tilde{\psi}^{1}_{1}\chi_{\omega_{d}}-\alpha_{2}\tilde{\psi}^{2}_{1}\chi_{\omega_{d}} \right) \,dxdt \\
	\noalign{\smallskip}
	\displaystyle +\int_{Q} \rho^{-2}_{0}\left(L^{*}_{2}\varphi_{2}-\alpha_{1}\psi^{1}_{2}\chi_{\omega_{d}}-\alpha_{2}\psi^{2}_{2}\chi_{\omega_{d}} \right) \left( L^{*}_{2}\tilde{\varphi}_{2}-\alpha_{1}\tilde{\psi}^{1}_{2}\chi_{\omega_{d}}-\alpha_{2}\tilde{\psi}^{2}_{2}\chi_{\omega_{d}} \right) \,dxdt \\
	\noalign{\smallskip}
	\displaystyle +\sum^{2}_{i=1}\int_{Q}\rho^{-2}_{0}\left( L_{1}\psi^{i}_{1}+\frac{1}{\mu_{i}}\rho^{-2}_{*}\varphi_{1}\chi_{\omega_{i}}  \right)\left( L_{1}\tilde{\psi}^{i}_{1}+\frac{1}{\mu_{i}}\rho^{-2}_{*}\tilde{\varphi}_{1}\chi_{\omega_{i}}  \right)\,dxdt\\
	\noalign{\smallskip}
	\displaystyle  +\sum^{2}_{i=1}\int_{Q}\rho^{-2}_{0}\left(L_{2}\psi^{i}_{2}+c\psi^{i}_{1}\right)\left(L_{2}\tilde{\psi}^{i}_{2}+c\tilde{\psi}^{i}_{1}\right)\,dxdt\\
	\noalign{\smallskip}
	\displaystyle + \int_{\omega\times (0,T)}\rho^{-2}_{1}\varphi_{1}\tilde{\varphi}_{1}\,dxdt,\\
	\noalign{\smallskip}
\end{array}$$	for all $(\varphi_{1}, \varphi_{2}, \psi^{1}_{1}, \psi^{2}_{1},  \psi^{1}_{2}, \psi^{2}_{2}),\, (\tilde{\varphi}_{1}, \tilde{\varphi}_{2}, \tilde{\varphi}^{1}_{1}, \tilde{\varphi}^{2}_{1},  \tilde{\varphi}^{1}_{2}, \tilde{\varphi}^{2}_{2})\in \, \mathcal{P}$, where
$$
\left\{\begin{array}{l}
	\displaystyle L^{*}_{i}\varphi_{i}=-\varphi_{i,t}-a(0)\varphi_{ixx}+F'_{i}(0)\varphi_{i},\\
	\noalign{\smallskip}
	\displaystyle L_{i}\psi^{j}_{i}=\psi^{j}_{i,t}-a(0)\psi^{j}_{i,xx}, \, \, i=1, 2, \, \, j=1, 2.\end{array}\right.
$$

\noindent From the arguments devoloped in \cite{FurImanu}, we deduce that $\mathcal{B}(\cdot,\cdot)$  is an inner product in   $\mathcal{P}$. By a standard result, we have   that  there  exists  a  vectorial space  $\overline{\mathcal{P}}$  such that  $\mathcal{P} \subset \overline{\mathcal{P}}$  and  $\left( \overline{\mathcal{P}},  \left\langle \cdot,\cdot \right\rangle   \right)$  with  $\left\langle \cdot, \cdot \right\rangle_{\overline{\mathcal{P}}}=\mathcal{B}(\cdot,\cdot) $  is  a  Hilbert space.

\noindent On the other hand, let us    define  the  linear  functional  $\mathcal{L}: \overline{\mathcal{P}}\mapsto \mathbb{R}$,  given by 
\begin{align*}
 \mathcal{L} (\varphi_{1}, \varphi_{2}, \psi^{1}_{1}, \psi^{2}_{1},  \psi^{1}_{2}, \psi^{2}_{2})  &= \int_{I}f\varphi_{1}(0)\,dx+\int_{I}g \varphi_{2}(0)\,dx\\   
 &	+\int_{Q}\left( G_{1}\varphi_{1}+G_{2}\varphi_{2}+G_{3}\psi^{1}_{1}+G_{4}\psi^{2}_{1}+G_{5}\psi^{1}_{2}+G_{6}\psi^{2}_{2}  \right)\,dxdt.
\end{align*}From Theorem \ref{Carleman-Theorem-2}, we deduce that
\begin{align*}
 & \left| \mathcal{L}   (\varphi_{1}, \varphi_{2}, \psi^{1}_{1}, \psi^{2}_{1},  \psi^{1}_{2}, \psi^{2}_{2}) \right|  \\
  & \leq C \left( \sum_{i=1}^{2}\int_{Q}\tilde{\rho}^{-2}|\varphi_{i}|^{2}\,dxdt+\sum_{i,j=1}^{2}\int_{Q} \tilde{\rho}^{-2}|\psi^{i}_{j}|^{2}\,dxdt+\sum_{i=1}^{2}\int_{I}|\varphi_{i}(0)|^{2}\,dx  \right)^{1/2}\\
  & \leq C \left( 
	\int_{\omega\times (0,T)}\tilde{\rho}^{-\frac{8}{5}}_{*}(s\tilde{\gamma})^{q_3}\vert \varphi_1 \vert^2dxdt +\displaystyle \int_{Q} \rho^{-2}_{0}| L^{*}_{1}\varphi_{1}+c\varphi_{2}-\alpha_{1}\psi^{1}_{1}\chi_{\omega_{d}}-\alpha_{2}\psi^{2}_{1}\chi_{\omega_{d}}|^{2}\,dxdt\right.\\
    & +\int_{Q} \rho^{-2}_{0}| L^{*}_{2}\varphi_{2}-\alpha_{1}\psi^{1}_{2}\chi_{\omega_{d}}-\alpha_{2}\psi^{2}_{2}\chi_{\omega_{d}}|^{2} \,dxdt +\sum^{2}_{i=1}\int_{Q}\rho^{-2}_{0}| L_{1}\psi^{i}_{1}+\frac{1}{\mu_{i}}\rho^{-2}_{*}\varphi_{1}\chi_{\omega_{i}}  |^{2}\,dxdt\\
    &+\left.\sum^{2}_{i=1}\int_{Q}\rho^{-2}_{0}|L_{2}\psi^{i}_{2}+c\psi^{i}_{1}|^{2}\,dxdt + \int_{\omega\times (0,T)}\rho^{-2}_{1}|\varphi_{1}|^{2}\,dxdt
		\right)^{1/2}\\
        &\leq C \mathcal{B}\left((\varphi_{1},\varphi_{2},\psi^{1}_{1},\psi^{2}_{1},\psi^{1}_{2},\psi^{2}_{2}); (\varphi_{1},\varphi_{2},\psi^{1}_{1},\psi^{2}_{1},\psi^{1}_{2},\psi^{2}_{2}) \right)^{1/2}\\
        &\leq C \Vert(\varphi_{1},\varphi_{2},\psi^{1}_{1},\psi^{2}_{1},\psi^{1}_{2},\psi^{2}_{2})\Vert_{\mathcal{P}},
\end{align*}where,  $C=C\left(f,g,G_{1},G_{2},G_{3},G_{4},G_{5},G_{6}\right)$, is some positive constant. Thus, $\mathcal{L}$ is continuous.

Therefore, using the  Lax Milgram  theorem,   there exists a unique  $(\hat{\varphi}_{1},\hat{\varphi}_{2},\hat{\psi}^{1}_{1},\hat{\psi}^{2}_{1},\hat{\psi}^{1}_{2},\hat{\psi}^{2}_{2})\in \overline{\mathcal{P}}$, such that the following variational identity holds:
\begin{equation}\label{variati_lax_pre}
\mathcal{B}\left( (\hat{\varphi}_{1},\hat{\varphi}_{2},\hat{\psi}^{1}_{1},\hat{\psi}^{2}_{1},\hat{\psi}^{1}_{2},\hat{\psi}^{2}_{2}), (\varphi_{1},\varphi_{2},\psi^{1}_{1},\psi^{2}_{1},\psi^{1}_{2},\psi^{2}_{2}) \right) = \mathcal{L}(\varphi_{1},\varphi_{2},\psi^{1}_{1},\psi^{2}_{1},\psi^{1}_{2},\psi^{2}_{2}),
\end{equation}for all $(\varphi_{1},\varphi_{2},\psi^{1}_{1},\psi^{2}_{1},\psi^{1}_{2},\psi^{2}_{2})\in \overline{\mathcal{P}}$.

On the other hand, if we denote by
\begin{equation}\label{second_for_transpo}
\left\{\begin{array}{l}
y_{1}=\rho^{-2}_{0}\left(L^{*}_{1}\hat{\varphi}_{1}+c\hat{\varphi}_{2}-\alpha_{1}\hat{\psi}^{1}_{1}\chi_{\omega_{d}}-\alpha_{2}\hat{\psi}^{2}_{1}\chi_{\omega_{d}} \right),\\
	\noalign{\smallskip}
	y_{2}=\rho^{-2}_{0}\left(L^{*}_{2}\hat{\varphi}_{2}-\alpha_{1}\hat{\psi}^{1}_{2}\chi_{\omega_{d}}-\alpha_{2}\hat{\psi}^{2}_{2}\chi_{\omega_{d}} \right),\\
	p^{1}_{1}= \rho^{-2}_{0}\left( L_{1}\hat{\psi}^{1}_{1}+\frac{1}{\mu_{1}}\rho^{-2}_{*}\hat{\varphi}_{1}\chi_{\omega_{1}}  \right), \\
	\noalign{\smallskip}
	p^{2}_{1}= \rho^{-2}_{0}\left( L_{1}\hat{\psi}^{2}_{1}+\frac{1}{\mu_{2}}\rho^{-2}_{*}\hat{\varphi}_{1}\chi_{\omega_{2}}  \right), \\
	\noalign{\smallskip}
	p^{1}_{2}= \rho^{-2}_{0}\left( L_{2}\hat{\psi}^{1}_{2}+c \hat{\psi}^{1}_{1}\right), \\
	\noalign{\smallskip}
	p^{2}_{2}= \rho^{-2}_{0}\left( L_{2}\hat{\psi}^{2}_{2}+ c \hat{\psi}^{2}_{1} \right), \\
	\noalign{\smallskip}
	h=-\rho^{-2}_{1}\hat{\varphi}_{1},
\end{array}\right.
\end{equation}from \eqref{variati_lax_pre} and \eqref{second_for_transpo}, we deduce that
\begin{equation}\label{vara_2identi_2}
\begin{array}{l}
	\displaystyle \int_{Q}y_{1}H_{1}\,dxdt+\int_{Q}y_{2}H_{2}\,dxdt\\
	\displaystyle \sum_{i=1}^{2}\int_{Q}p^{i}_{1}H_{i+2}\,dxdt+\sum_{i=1}^{2}\int_{Q}p^{i}_{2}H_{i+4}\,dxdt\\
	\displaystyle = \int_{I}f \varphi_{1}(0)\,dx+\int_{I}g \varphi_{2}(0)\,dx\\
	\displaystyle +\int_{Q}\left( (G_{1}+h\chi_{\omega})\varphi_{1}+G_{2}\varphi_{2}+G_{3}\psi^{1}_{1}+G_{4}\psi^{2}_{1}+G_{5}\psi^{1}_{2}+G_{6}\psi^{2}_{2} \right)\,dxdt, 
\end{array}
\end{equation}for all $(\varphi_{1},\varphi_{2},\psi^{1}_{1},\psi^{2}_{1},\psi^{1}_{2},\psi^{2}_{2})\in \overline{\mathcal{P}}$, where  $(\varphi_{1},\varphi_{2},\psi^{1}_{1},\psi^{2}_{1},\psi^{1}_{2},\psi^{2}_{2})$  is    solution  of the adjoint system \eqref{ad3333}.

\noindent Thus,  from \eqref{transposition_identity},  \eqref{vara_2identi_2} and  the uniqueness of solutions by transposition, given in Theorem \ref{wellpo_trans_sis}, $(y_{1},y_{2},p^{1}_{1},p^{2}_{1}, p^{1}_{2},p^{2}_{2})$ is the unique weak solution of \eqref{lin1}. In particular, for
\begin{equation*}
(\varphi_{1},\varphi_{2},\psi^{1}_{1},\psi^{2}_{1},\psi^{1}_{2},\psi^{2}_{2}) = (\hat{\varphi}_{1}, \hat{\varphi}_{2}, \hat{\psi}^{1}_{1}, \hat{\psi}^{2}_{1}, \hat{\psi}^{1}_{2}, \hat{\psi}^{2}_{2}),
\end{equation*}from \eqref{ad3333}-\eqref{vara_2identi_2}, we deduce that
\begin{align*}
\displaystyle \int_{Q}\rho^{2}_{0}|y_{1}|^{2}\,dxdt + \int_{Q}\rho^{2}_{0}|y_{2}|^{2}\,dxdt  +\sum_{i,j=1}^{2}\int_{Q}\rho^{2}_{0}|p^{i}_{j}|^{2}\,dxdt+\int_{Q}\rho^{2}_{1}|h|^{2}\,dxdt<+\infty,
\end{align*}proving the regularities \eqref{est-linear1}.

Now, let us prove  the estimate \eqref{est-linear2}. Let $t \in [0, T]$.  Multiplying 
$\eqref{lin1}_{1}$ by $\hat{\rho}^{2}_{0}y_{1}$ and integrate over  $I$, we have that
\begin{align*}
	&\frac{1}{2}\frac{d}{dt} \left( \int_{I}\hat{\rho}^{2}_{0}|y_{1}|^{2}\,dx \right)+a(0)\int_{I}\hat{\rho}^{2}_{0}|y_{1,x}|^{2}\,dx + F'_{1}(0)\int_{I}\hat{\rho}^{2}_{0}|y_{1}|^{2}\,dx \\
   & =\int_{I}\hat{\rho}_{0}\hat{\rho}_{0,t}|y_{1}|^{2}\,dx +\int_{I}\left(G_{1}+h\chi_{\omega}-\frac{1}{\mu_{1}}\rho^{-2}_{\ast}p^{1}_{1}\chi_{\omega_{1}}-\frac{1}{\mu_{2}}\rho^{-2}_{\ast}p^{2}_{1}\chi_{\omega_{2}}    \right)\hat{\rho}^{2}_{0}y_{1}\,dx.
\end{align*}Then, integrating the previous  identity on $(0,t)$, from \eqref{weight-linear1},  we get 
\begin{equation}\label{linear-eq3}
\begin{split}
&\int_{I}\hat{\rho}^{2}_{0}|y_{1}(x,t)|^{2}\,dx+a(0)\int_{Q}\hat{\rho}^{2}_{0}|y_{1,x}|^{2}\,dx dt \\
&\leq C\left( \int_{Q}\rho^{2}_{0}|y_{1}|^{2}\,dxdt+\int_{Q}\rho^{2}_{0}(|p^{1}_{1}|^{2}+|p^{2}_{1}|^{2})\,dxdt \right.\\
			 &\left. +\int_{\omega\times(0,T)}\rho^{2}_{1}|h|^{2}\,dxdt+\int_{Q}\tilde{\rho}^{2}|G_{1}|^{2}\,dxdt+\int_{I}|f|^{2}\,dx \right),
\end{split}
\end{equation}for all $t \in [0, T]$.

\noindent Similarly, multiplying $\eqref{lin1}_{2}$, $\eqref{lin1}_{3}$, $\eqref{lin1}_{4}$, $\eqref{lin1}_{5}$ and $\eqref{lin1}_{6}$      by $\hat{\rho}^{2}_{0}y_{2}$, $\hat{\rho}^{2}_{0}p^{1}_{1}$, $\hat{\rho}^{2}_{0}p^{2}_{1}$, $\hat{\rho}^{2}_{0}p^{1}_{2}$ and $\hat{\rho}^{2}_{0}p^{2}_{2}$, respectively,  integrate over  $I\times (0, t)$ and using \eqref{weight-linear1},    we obtain that
\begin{equation}\label{linear-eq4}
\begin{split}
&\int_{I}\hat{\rho}^{2}_{0}|y_{2}(x,t)|^{2}\,dx+a(0)\int_{Q}\hat{\rho}^{2}_{0}|y_{2,x}|^{2}\,dxdt \\
&\leq  C\left(   \int_{Q}\rho^{2}_{0}|y_{1}|^{2}\,dxdt+\int_{Q}\rho^{2}_{0}|y_{2}|^{2}\,dxdt +\int_{Q}\tilde{\rho}^{2}|G_{2}|^{2}\,dxdt+\int_{I}|g|^{2}\,dx \right)
\end{split}
\end{equation}
\begin{equation}\label{linear-eq5}
\begin{split}
&\int_{I}\hat{\rho}^{2}_{0}|p^{i}_{1}(x,t)|^{2}\,dx+a(0)\int_{Q}\hat{\rho}^{2}_{0}|p^{i}_{1,x}|^{2}\,dxdt\\
&\leq C\left(  \int_{Q}\rho^{2}_{0}|y_{1}|^{2}\,dxdt+ \int_{Q}\rho^{2}_{0}|p^{i}_{1}|^{2}\,dxdt  +\int_{Q}\rho^{2}_{0}|p^{i}_{2}|^{2}\,dxdt+\int_{Q}\tilde{\rho}^{2}|G_{i+2}|^{2}\,dxdt\right), 
\end{split}
\end{equation}and 
\begin{equation}\label{linear-eq6}
\begin{split}
  & \int_{I}\hat{\rho}^{2}_{0}|p^{i}_{2}(x,t)|^{2}\,dx+a(0)\int_{Q}\hat{\rho}^{2}_{0}|p^{i}_{2,x}|^{2}\,dxdt\\
  &\leq C\left(  \int_{Q}\rho^{2}_{0}|y_{2}|^{2}\,dxdt +\int_{Q}\rho^{2}_{0}|p^{i}_{2}|^{2}\,dxdt+\int_{Q}\tilde{\rho}^{2}|G_{i+4}|^{2}\,dxdt\right),
\end{split}
\end{equation}for all $t \in [0, T]$,  $i=1,2.$ Consequently, from estimates \eqref{linear-eq3}-\eqref{linear-eq6}, the inequality \eqref{est-linear2} holds.

Finally, let us prove  the  estimate \eqref{est-linear3}.  Multiplying    $\eqref{lin1}_{1}$  by  $\hat{\rho}^{2}_{1}y_{1,t}$ and  integrating over  $I$, we  have 
\begin{align*}
&\frac{a(0)}{2}\frac{d}{dt}\left( \int_{I}\hat{\rho}^{2}_{1}|y_{1,x}|^{2}\,dx	\right)+\int_{I}\hat{\rho}^{2}_{1}|y_{1,t}|^{2}\,dx\\
&=\int_{I}a(0)\hat{\rho}_{1}\hat{\rho}_{1,t}|y_{1,x}|^{2}\,dx-\int_{I}F'_{1}(0)\hat{\rho}^{2}_{1}y_{1}y_{1,t}\,dx  \\
&+\int_{I}\hat{\rho}^{2}_{1}(G_{1}+h\chi_{\omega}-\frac{1}{\mu_{1}}\rho^{-2}_{\ast}p^{1}_{1}\chi_{\omega_{1}}-\frac{1}{\mu_{2}}\rho^{-2}_{\ast}p^{2}_{1}\chi_{\omega_{2}})y_{1,t}\,dx.
\end{align*}Integrating the above  identity on $(0,t)$, using \eqref{weight-linear1} and estimate \eqref{est-linear2}, it follows that
\begin{equation}\label{linear-eq8}
\begin{split}
&\int_{I}\hat{\rho}^{2}_{1}|y_{1,x}(x, t)|^{2}\,dx+\int_{Q}\hat{\rho}^{2}_{1}|y_{1,t}|^{2}\,dxdt \\
& \leq C \left( \int_{Q}\rho^{2}_{0}|y_{1}|^{2}\,dxdt + \int_{Q}\rho^{2}_{0}(|p^{1}_{1}|^{2}+|p^{2}_{1}|^{2})\,dxdt \right.\\
&\left.	  +\int_{\omega\times(0,T)}\rho^{2}_{1}|h|^{2}\,dxdt+\int_{Q}\tilde{\rho}^{2}|G_{1}|^{2}\,dxdt +\int_{Q}\hat{\rho}^{2}_{0}|y_{1,x}|^{2}\,dxdt+\int_{I}|f_{x}|^{2}\,dx\right) \\
&\leq C\left( \sum_{i=1}^{2}\int_{Q}\rho^{2}_{0}|y_{i}|^{2}\,dxdt+\sum_{i,j=1}^{2}\rho^{2}_{0}|p^{i}_{j}|^{2}\,dxdt +\sum_{i=1}^{2}\int_{Q}\tilde{\rho}^{2}|G_{i}|^{2}\,dxdt \right.\\
&\left.+\int_{I}|f_{x}|^{2}\,dx+\int_{I}|g_{x}|^{2}\,dx+\int_{\omega\times(0,T)}\rho^{2}_{1}|h|^{2}\,dxdt\right),
\end{split}
\end{equation}for all $t \in [0, T]$.

\noindent In a similar way, multiplying $\eqref{lin1}_{2}$, $\eqref{lin1}_{3}$, $\eqref{lin1}_{4}$, $\eqref{lin1}_{5}$ and $\eqref{lin1}_{6}$      by $\hat{\rho}^{2}_{1}y_{2,t}$, $-\hat{\rho}^{2}_{1}p^{1}_{1,t}$, $-\hat{\rho}^{2}_{1}p^{2}_{1,t}$, $-\hat{\rho}^{2}_{1}p^{1}_{2,t}$ and $-\hat{\rho}^{2}_{1}p^{2}_{2,t}$, respectively,  integrate over  $I\times (0, t)$, using \eqref{weight-linear1} and inequality \eqref{est-linear2}, we deduce that
\begin{equation}\label{linear-eq9}
\begin{split}
&\int_{I}\hat{\rho}^{2}_{1}|y_{2,x}(x,t)|^{2}\,dx+\int_{Q}\hat{\rho}^{2}_{1}|y_{2,t}|^{2}\,dxdt\\
&\leq C\left( \sum_{i=1}^{2}\int_{Q}\rho^{2}_{0}|y_{i}|^{2}\,dxdt+\sum_{i,j=1}^{2}\int_{Q}\rho^{2}_{0}|p^{i}_{j}|^{2}\,dxdt+\sum_{i=1}^{2}\int_{Q}\tilde{\rho}^{2}|G_{i}|^{2}\,dxdt   \right.\\
&\left.+\int_{I}|f_{x}|^{2}\,dx+\int_{I}|g_{x}|^{2}\,dx+\int_{\omega\times (0,T)}\rho^{2}_{1}|h|^{2}\,dxdt\right),
	\end{split}
\end{equation}
\begin{equation}\label{linear-eq10}
\begin{split}
& \int_{I}\hat{\rho}^{2}_{1}|p^{i}_{1,x}(x,t)|^{2}\,dx+\int_{Q}\hat{\rho}^{2}_{1}|p^{i}_{1,t}|^{2}\,dxdt \\
& \leq C\left( \sum_{i=1}^{2}\int_{Q}\rho^{2}_{0}|y_{i}|^{2}\,dxdt+\sum_{i,j=1}^{2}\int_{Q}\rho^{2}_{0}|p^{i}_{j}|^{2}\,dxdt+\sum_{i=1}^{2}\int_{Q}\tilde{\rho}^ {2}|G_{i+2}|^{2}\,dxdt   \right.\\
&\left.+\int_{I}|f_{x}|^{2}\,dx+\int_{I}|g_{x}|^{2}\,dx+\int_{\omega\times (0,T)}\rho^{2}_{1}|h|^{2}\,dxdt\right),
\end{split}
\end{equation}and
\begin{equation}\label{linear-eq11}
\begin{split}
&\int_{I}\hat{\rho}^{2}_{1}|p^{i}_{2,x}(x,t)|^{2}\,dx+\int_{Q}\hat{\rho}^{2}_{1}|p^{i}_{2,t}|^{2}\,dxdt\\
&\leq C\left( \sum_{i=1}^{2}\int_{Q}\rho^{2}_{0}|y_{i}|^{2}\,dxdt+\sum_{i,j=1}^{2}\int_{Q}\rho^{2}_{0}|p^{i}_{j}|^{2}\,dxdt+\sum_{i=1}^{2}\int_{Q}\tilde{\rho}^{2}|G_{i+4}|^{2}\,dxdt   \right.\\
& \left.+\int_{I}|f_{x}|^{2}\,dx+\int_{I}|g_{x}|^{2}\,dx+\int_{\omega\times (0,T)}\rho^{2}_{1}|h|^{2}\,dxdt\right),
\end{split}
\end{equation}for all $t \in [0, T]$, $i=1, 2$.

\noindent Furthermore, multiplying   $\eqref{lin1}_{1}$  by  $-\hat{\rho}^{2}_{1}y_{1,xx}$ and  integrating over  $I$, we obtain  that
\begin{align*}
&a(0)\int_{I}\hat{\rho}^{2}_{1}|y_{1,xx}|^{2}\,dx\\
	& = -\int_{I}\hat{\rho}^{2}_{1}(-y_{1,t}-F'_{1}(0)y_{1} +G_{1}+h\chi_{\omega}-\frac{1}{\mu_{1}}\rho^{-2}_{\ast}p^{1}_{1} \chi_{\omega_{1}}-\frac{1}{\mu_{2}}\rho^{-2}_{\ast}p^{2}_{1}\chi_{\omega_{2}} )y_{1,xx}\,dx.
\end{align*}Consequently, from above equality and the estimates \eqref{est-linear2}, \eqref{linear-eq8}, -\eqref{linear-eq11}, we have that
\begin{equation}
\begin{split}
& \int_{Q}\hat{\rho}^{2}_{1}|y_{1,xx}|^{2}\,dxdt \\
&\leq C\left(   \sum_{i=1}^{2}\int_{Q}\rho^{2}_{0}|y_{i}|^{2}\,dxdt+\sum_{i,j=1}^{2}\int_{Q}\rho^{2}_{0}|p^{i}_{j}|^{2}\,dxdt+\sum_{i=1}^{2}\int_{Q}\tilde{\rho}^{2}|G_{i}|^{2}\,dxdt   \right. \\
& \left.+\int_{I}|f_{x}|^{2}\,dx+\int_{I}|g_{x}|^{2}\,dx+\int_{\omega\times (0,T)}\rho^{2}_{1}|h|^{2}\,dxdt\right).
\end{split}
\end{equation}Similarly,
\begin{equation}\label{linear-eq13}
\begin{split}
&\int_{Q}\hat{\rho}^{2}_{1}|y_{2,xx}|^{2}\,dxdt+\sum_{i,j=1}^{2}\int_{Q}\hat{\rho}^{2}_{1}|p^{i}_{j,xx}|^{2}\,dxdt\\
&\leq C\left(   \sum_{i=1}^{2}\int_{Q}\rho^{2}_{0}|y_{i}|^{2}\,dxdt+\sum_{i,j=1}^{2}\int_{Q}\rho^{2}_{0}|p^{i}_{j}|^{2}\,dxdt+\sum_{i=1}^{2}\int_{Q}\tilde{\rho}^{2}|G_{i}|^{2}\,dxdt   \right.\\
&\left.+\int_{I}|f_{x}|^{2}\,dx+\int_{I}|g_{x}|^{2}\,dx+\int_{\omega\times (0,T)}\rho^{2}_{1}|h|^{2}\,dxdt\right).
\end{split}
\end{equation}Therefore, from estimates \eqref{linear-eq8}-\eqref{linear-eq13}, we have that \eqref{est-linear3}  holds and the proof ends.

\end{proof}

\section{\bf Hierarchical local null controllability problem}\label{sec:qlinear}

In the last section we proved the global null controllability of the linearized system. Using the inverse function theorem of Liusternik, in this section we prove the local controllability of the initial nonlinear system \eqref{Eq1}. Let us recall the statement of the theorem:

\begin{theorem}[Liusternik's Theorem]
    Let $Y$ and $Z$ be Banach spaces and let $A : B_r(0) \subset Y \rightarrow Z$ be a $C^1$ mapping. Let us assume that the derivative $A'(0) : Y \rightarrow Z$ is onto and let us denote set $\xi_0 = A(0)$. Then, there exist $\epsilon>0$, a mapping $W : B_\epsilon(\xi_0) \subset Z \rightarrow Y$, and a constant $K > 0$ satisfying , for any $z \in  B_\epsilon(\xi_0)$, 
    \begin{itemize}
        \item $W(z) \in B_r(0)$ and $A(W (z)) = z$,
        \item $\parallel W(z)\parallel_Y \leq K\parallel z - \xi_0 \parallel_Z$.
    \end{itemize}
\end{theorem}
The proof of this theorem can be found in \cite{Alekseev}.

\subsection{Spaces of regularity for state and main control}
Using the notations
$$
    \displaystyle L^{*}_{i}\varphi_{i}=-\varphi_{i,t}-a(0)\varphi_{ixx}+F'_{i}(0)\varphi_{i}, \qquad
	\displaystyle L_{i}\psi^{j}_{i}=\psi^{j}_{i,t}-a(0)\psi^{j}_{i,xx} +F'_{i}(0)\psi_{i}^j,
$$
let us introduce the space
\begin{equation*}
 	\begin{split}
 	Y:=\{&(y_1,y_2,p^1_1,p^1_2,p^2_1,p^2_2,h) \ | \ y_1,y_2, y_{1,x}, y_{2,x}, \rho_0y_1, \rho_0y_2, \rho_0p^i_j \in L^2(Q),i,j=1,2,\\
 	&\rho_1h \in L^2(\omega \times (0,T)), \\ 
 	&\tilde{\rho} [L_1 y_1  -h\chi_\omega +\frac{1}{\mu_1}\rho_*^{-2}p_1^1\chi_{\omega_1} +\frac{1}{\mu_2}\rho_*^{-2}p_1^2\chi_{\omega_2}] \in L^2(Q), \ \tilde{\rho}[L_2 y_2   +cy_1] \in L^2(Q),\\
 	&\tilde{\rho}[L_1^* p_1^i + cp_2^i -\alpha_i y_1 \chi_{\omega_{i,d}}] \in L^2(Q), \
 	\tilde{\rho}[L_2^* p_2^i  -\alpha_i y_2 \chi_{\omega_{i,d}}]  \in L^2(Q),\\
 	& y_1(0), \ y_2(0) \in  H^1_0(I), \ y_1 |_{\Sigma} = 0, \ y_2 |_{\Sigma} = 0, \ p_j^i |_{\Sigma} = 0, i,j=1,2  
    \}
 	\end{split} 
\end{equation*}

\noindent with the norm
\begin{equation*}
\begin{split}
&\Vert(y_1,y_2,p^1_1,p^1_2,p^2_1,p^2_2,h)\Vert^2_Y \\
&=\norm{\rho_0 y_1} ^2_{L^2(Q)}+\norm{\rho_0 y_2} ^2_{L^2(Q)}+\sum_{i,j=1}^{2}\norm{\rho_0 p_i^j} ^2_{L^2(Q)}+\norm{\rho_1 h} ^2_{L^2(\omega\times (0,T))}\\
    &+\norm{\tilde{\rho}(L_1 y_{1} -h\chi_\omega +\frac{1}{\mu_1}\rho_*^{-2}p_1^1\chi_{\omega_1} +\frac{1}{\mu_2}\rho_*^{-2}p_1^2\chi_{\omega_2})}^{2}_{L^2(Q)}\\
    &+\norm{\tilde{\rho}(L_2 y_{2}  +cy_1)}^{2}_{L^2(Q)} +\norm{\tilde{\rho}[L_1^*p_1^1 +cp_2^1 -\alpha_1 y_1 \chi_{\omega_{1,d}}]}^{2}_{L^2(Q)}\\
    &+\norm{\tilde{\rho}[L_1^* p_1^2+cp_2^2 -\alpha_2 y_1 \chi_{\omega_{2,d}}]}^{2}_{L^2(Q)}+\norm{\tilde{\rho}[L_2^* p_2^1 -\alpha_1 y_2 \chi_{\omega_{1,d}}]}^{2}_{L^2(Q)}\\
    &+\norm{\tilde{\rho}[L_2^* p_2^2 -\alpha_2 y_2 \chi_{\omega_{2,d}}]}^{2}_{L^2(Q)}+\norm{y_1(0)}^2_{H^1_0(I)}+\norm{y_2(0)}^2_{H^1_0(I)}.
\end{split}
\end{equation*}
It is clear that $Y$ is a Banach space with the norm $\| \cdot \|_Y$. From Theorem \ref{teo:linearized_control}, we have
$$
\sum_{i=1}^2 \| \hat \rho_0 y_i \|^2_{L^2(0,T;H_0^1(I))} + \sum_{i,j=1}^2 \| \hat \rho_0 p_j^i \|^2_{L^2(Q)} \leq C \| (y_1,y_2,p_i^1,p_2^1,p_1^2,p_2^2,h) \|_Y^2.
$$
and
\begin{equation*}
    \begin{split}
        &\sum_{i=1}^2 \sup_{t \in [0,T]} \int_I \hat \rho_1^2 |y_{i,x}|^2 dx + \sum_{i,j=1}^2 \sup_{t \in [0,T]} \int_I \hat \rho_1^2 |p_{j,x}^i|^2 dx \\
        &+ \sum_{i=1}^2 \left( \int_Q \hat \rho_1^2 |y_{i,xx}|^2 dxdt + \int_Q \hat \rho_1^2 |y_{i,t}|^2 dxdt \right) \\
        & + \sum_{i,j=1}^2 \left( \int_Q \hat \rho_1^2 |p_{j,xx}^i|^2 dxdt + \int_Q \hat \rho_1^2 |p_{j,t}^i|^2 dxdt \right) \leq \| (y_1,y_2,p_i^1,p_2^1,p_1^2,p_2^2,h) \|_Y^2.        
    \end{split}
\end{equation*}

On the other hand, let us consider the Hilbert space
$$\mathcal{ M}:= L^2(\tilde{\rho}^2,Q) = \{ w \in L^2(Q) \ : \ \tilde{\rho} w \in L^2(Q) \},$$ 
and define the space
\begin{eqnarray*}
    Z:= \mathcal{ M} \times \mathcal{ M} \times \mathcal{ M} \times \mathcal{ M} \times \mathcal{ M} \times \mathcal{ M} \times H^1_0(I) \times H^1_0(I),
\end{eqnarray*}
with the norm
\begin{equation*}
    \begin{split}
        \| (F,G,F_1,F_2,F_3,F_4,y^{0}_1,y^{0}_{2}) \|_Z^2 = &\int_Q \tilde{\rho}^2 |F|^2 dxdt + \int_Q \tilde{\rho}^2 |G|^2 dxdt + \sum_{i=1}^4 \int_Q \tilde{\rho}^2 |F_i|^2 dxdt \\
        &+ \|y^{0}_{1}\|_{H_0^1(I)}^2+ \|y^{0}_{2}\|_{H_0^1(I)}^2.
    \end{split}
\end{equation*}
Consider the mapping $A:Y \to Z$, defined by $A: = (A_0^1,A_0^2,A_1^1,A_1^2,A_2^1,A_2^2,A_3^1,A_3^2)$,  where
\begin{equation} \label{aplicação A}
	\begin{cases}
		A_0^1 =  y_{1,t}-(a(y_1)y_{1,x})_x + F_1(y_1)   - h\chi_w+\frac{1}{\mu_1} \rho_*^{-2} p_1^1\chi_{w_1}+\frac{1}{\mu_2} \rho_*^{-2} p_1^2\chi_{w_2}, \\  
		A_0^2 = y_{2,t}-(a(y_2)y_{2,x})_x + F_2(y_2) + c y_1,  \\
        A_1^1 = -p_{1,t}^1 - (a(y_1)p_{1,x}^1)_x + a'(y_1)y_{1,x}p_{1,x}^1+F_1'(y_1)p_1^1+cp_2^1-\alpha_1 y_1 \chi_{w_{1,d}}, \\
		A_1^2 = -p_{1,t}^2 - (a(y_1)p_{1,x}^2)_x + a'(y_1) y_{1,x} p_{1,x}^2 + F_1'(y_1)p_1^2+cp_2^2-\alpha_2 y_1 \chi_{w_{2,d}}, \\
		A_2^1 = -p_{2,t}^1 - (a(y_2)p_{2,x}^1)_x + a'(y_2) y_{2,x} p_{2,x}^1 + F_2'(y_2)p_2^1-\alpha_1 y_2 \chi_{w_{1,d}}, \\
		A_2^2 = -p_{2,t}^2 - (a(y_2)p_{2,x}^2)_x +a'(y_2)y_{2,x}p_{2,x}^2 + F_2'(y_2)p_2^2 - \alpha_2 y_2 \chi_{w_{2,d}}, \\
		A_3^1 = y_1(0),\\
        A_3^2 = y_2(0).\\
	\end{cases}
\end{equation}

The proof of the main Theorem uses the following three lemmas.
\begin{lemma}\label{lema1}
$A : Y \rightarrow Z$ is well defined
and continuous.
\end{lemma}
\begin{proof} Let   $(y_1,y_2,p^1_1,p^1_2,p^2_1,p^2_2,h) \in Y$.

\textbf{  $A$ is well-defined: }
	Indeed. First, 	  we have
	$$
    \begin{array}{l}
        \parallel A_0^1 (y_1,y_2,p^1_1,p^1_2,p^2_1,p^2_2,h) \parallel^2_{\mathcal{ M}}\\
        \displaystyle=\int_{Q}\tilde{\rho}^{2}\left|y_{1,t}-(a(y_{1})y_{1,x})_{x}+F_{1}(y_{1})-h\chi_{\omega}+\frac{1}{\mu_{1}} \rho_*^{-2} p^{1}_{1}\chi_{\omega_{1}}+\frac{1}{\mu_{2}} \rho_*^{-2} p^{2}_{1}\chi_{\omega_{2}}  \right|^{2}\,dxdt.
    \end{array}
    $$
    
    \noindent Consequently, we have
    \begin{equation}\label{lemma-eq1}
        \begin{array}{l}
    	   \left\| A_0^1(y_1,y_2,p^1_1,p^1_2,p^2_1,p^2_2,h) \right\|^2_{\mathcal{ M}} \\[1.0em]
    	   \displaystyle \leq 4\left\|\tilde{\rho}\left( L_{1}y_{1}-h\chi_{\omega}+\frac{1}{\mu_{1}}\rho^{-2}_{\ast}\chi_{\omega_{1}}+\frac{1}{\mu_{2}}\rho^{-2}_{\ast}p^{2}_{1}\chi_{2}   \right) \right\|^{2}_{L^{2}(Q)}\\[1.0em]
    	   \displaystyle + 4 \|\tilde{\rho}\left(  a(y_{1})-a(0) \right)y_{1,xx}\|^{2}_{L^{2}(Q)}+4 \Vert\tilde{\rho}a'(y_{1})(y_{1,x})^{2}\Vert^{2}_{L^{2}(Q)}\\[1.0em]
    	   \displaystyle  +4 \|\tilde{\rho} \left( F_{1}(y_{1})-F'(0)y_{1} \right)\|^{2}_{L^{2}(Q)}\\ [1.0em]
    	   =I_{1}+I_{2}+I_{3}+I_{4}.
        \end{array}
    \end{equation}Now, let us estimate  each term $I_{i}$,  for  $i=1,2,3,4$. First, by the definition of the norm in $Y$,
    \begin{equation}\label{lemma-eq2}
    	I_{1}\leq 4  ||(y_1,y_2,p^1_1,p^1_2,p^2_1,p^2_2,h)||^{2}_{Y}.
    \end{equation}On the other hand,
    $$
    \begin{array}{l}
    	\displaystyle I_{2}=6\int_{Q}\tilde{\rho}^{2}|(a(y_{1})-a(0))y_{1,x}|^{2}\,dxdt
        \leq C \int_{Q}\tilde{\rho}^{2}|y_{1}|^{2}|y_{1,xx}|^{2}\,dxdt.
    \end{array}
    $$
    From Sobolev embedding and using \eqref{eq:compara_rhos3}, we have that
    
    $$
    \begin{array}{l}
        \displaystyle 	I_{2}\leq  C\int_{0}^{T}\tilde{\rho}^{2}||y_{1}||^{2}_{L^{\infty}(I)}||y_{1,xx}||^{2}_{L^{2}(I)}\,dt\\[1.0em]
        
        \displaystyle \quad	\leq C \int_{0}^{T}\tilde{\rho}^{2}||y_{1,x}||^{2}_{L^{2}(I)}||y_{1,xx}||^{2}_{L^{2}(I)}\,dt\\[1.0em]
    
        \displaystyle \quad\leq C \left(  \sup_{t\in [0,T]}\int_{I}\hat{\rho}^{2}_{1}|y_{1,x}|^{2}\,dx \right)  \left( \int_{Q}\hat{\rho}^{2}_{1}|y_{1,xx}|^{2}\,dxdt \right).
    \end{array}
    $$
    
    \noindent Then,
    \begin{equation}\label{lemma-eq5}
        I_{2}\leq C ||(y_1,y_2,p^1_1,p^1_2,p^2_1,p^2_2,h)||^{4}_{Y}
    \end{equation}   	Once again, by Sobolev embedding  and estimate  \eqref{eq:compara_rhos3}, we deduce that
    $$
    \begin{array}{l}
    \displaystyle I_{3} \leq 4 \int_{Q}\tilde{\rho}^{2}|a'(y_{1})|^{2}|y_{1,x}|^{4}\,dxdt\\
      \noalign{\smallskip}\phantom{AA}
    \displaystyle \leq C \int_{0}^{T}\tilde{\rho}^{2}||y_{1,x}||^{2}_{L^{2}(I)}||y_{1,x}||^{2}_{L^{\infty}(I)}\,dt\\
      \noalign{\smallskip}\phantom{AA}
    \displaystyle \leq C \int_{0}^{T}\tilde{\rho}^{2}||y_{1,x}||^{2}_{L^{2}(I)}||y_{1,xx}||^{2}_{L^{2}(I)}\,dt\\
      \noalign{\smallskip}\phantom{AA}
    \displaystyle \leq C \left(\sup_{t\in [0,T]}\int_{I}\hat{\rho}^{2}_{1}|y_{1,x}|^{2}\,dx  \right)\left( \int_{Q}\hat{\rho}^{2}_{1}|y_{1,xx}|^{2}\,dxdt\right).
        \end{array}
    $$
    Therefore,
    \begin{equation}\label{lemma-eq6}
        I_{3}\leq C ||(y_1,y_2,p^1_1,p^1_2,p^2_1,p^2_2,h)||^{4}_{Y}.
    \end{equation}
    
    \noindent Proceeding analogously, 
    since  $F_{1}(y_{1})=F'_{1}(\lambda_{0}y_{1})y_{1},\,\,  \mbox{for  some}\,\, \lambda_{0}\in (0,1)$, and \eqref{eq:compara_rhos3}, we get
    $$
    \begin{array}{l}
        \displaystyle 	I_{4}=4\int_{Q}\tilde{\rho}^{2}|F_{1}(y_{1})-F'_{1}(0)y_{1}|^{2}\,dxdt\\
    	  \noalign{\smallskip}\phantom{AA}
        \displaystyle 	=4\int_{Q}\tilde{\rho}^{2}|(F'_{1}(\lambda_{0}y_{1})-F'_{1}(0))y_{1}|^{2}\,dxdt\\
    	  \noalign{\smallskip}\phantom{AA}
        \displaystyle \leq 	C\int_{Q}\tilde{\rho}^{2}|y_{1}|^{4}\,dxdt\\
    	  \noalign{\smallskip}\phantom{AA}
        \displaystyle 	\leq C \left(\sup_{t\in [0,T]}\int_{I}\hat{\rho}^{2}_{1}|y_{1,x}|^{2}\,dx  \right) \left( \int_{Q}\rho^{2}_{0}|y_{1}|^{2}\,dxdt  \right).    
    \end{array}
    $$
    
    \noindent Thus, 
    \begin{equation}\label{lemma-eq7}
        I_{4}\leq 	C   ||(y_1,y_2,p^1_1,p^1_2,p^2_1,p^2_2,h)||^{4}_{Y}.
    \end{equation} From estimates \eqref{lemma-eq2}, \eqref{lemma-eq5},  \eqref{lemma-eq6}, and \eqref{lemma-eq7}, we get
    \begin{equation}\label{lemma-eq8}
        \begin{array}{l}
    	   \parallel A_0^1 (y_1,y_2,p^1_1,p^1_2,p^2_1,p^2_2,h) \parallel^2_{\mathcal{ M}}\\ [1.0em]
            \leq 	C||(y_1,y_2,p^1_1,p^1_2,p^2_1,p^2_2,h)||^{2}_{Y}\left(1+||(y_1,y_2,p^1_1,p^1_2,p^2_1,p^2_2,h)||^{2}_{Y} \right).
    	\end{array}
    \end{equation} 
    \noindent Secondly, it follows that
    $$
    \begin{array}{l}
   		\displaystyle \parallel A_0^2 (y_1,y_2,p^1_1,p^1_2,p^2_1,p^2_2,h) \parallel^2_{\mathcal{M}}\\[1.0em]
\displaystyle=\int_{Q}\tilde{\rho}^{2}|y_{2,t}-(a(y_{2})y_{2,x})_{x}+F_{2}(y_{2})+cy_{1}|^{2}\,dxdt\\[1.0em]
   		
   		\displaystyle \leq 4\int_{Q}\tilde{\rho}^{2}|L_{2}+cy_{1}|^{2}\,dxdt+4\int_{Q}\tilde{\rho}^{2}|(a(y_{2})-a(0))y_{2,xx}|^{2}\,dxdt\\[1.0em]
   		
   		\displaystyle +4\int_{Q}\tilde{\rho}^{2}|a'(y_{2})|^{2}|y_{2,xx}|^{4}\,dxdt
   		\displaystyle  +4\int_{Q}\tilde{\rho}^{2}|F_{2}(y_{2})-F'_{2}(0)y_{2}|^{2}\,dxdt\\[1.0em]
   		
   		= J_{1}+J_{2}+J_{3}+J_{4}.
    \end{array}
    $$

    \noindent Clearly,
    \begin{equation}\label{lemma-eq9}
    	J_{1}= 4\int_{Q}\tilde{\rho}^{2}|L_{2}+cy_{1}|^{2}\,dxdt \leq 4||(y_1,y_2,p^1_1,p^1_2,p^2_1,p^2_2,h)||^{2}_{Y}.
    \end{equation}On the other hand, using  \eqref{eq:compara_rhos3}, we obtain that
    \begin{equation}\label{lemma-eq10}
    \begin{aligned}
     	J_{2}&= 4\int_{Q}\tilde{\rho}^{2}|a(0)-a(y_{2})|^{2}|y_{2,xx}|^{2}\,dxdt\\
    	  \noalign{\smallskip}\phantom{AA}
        &\displaystyle 	\leq  C  \int_{Q}\tilde{\rho}^{2}|y_{2}|^{2}|y_{2,xx}|^{2}\,dxdt\\
    	  \noalign{\smallskip}\phantom{AA}
    	&\displaystyle \leq C   \left(\sup_{t\in [0,T]}\int_{I}\hat{\rho}^{2}_{1}|y_{2,x}|^{2} \,dx\right)\left( \int_{Q}\hat{\rho}^{2}_{1}|y_{2,xx}|^{2}\,dxdt  \right)\\
\noalign{\smallskip}\phantom{AA}
    	  & \leq C||(y_1,y_2,p^1_1,p^1_2,p^2_1,p^2_2,h)||^{4}_{Y}.
    \end{aligned}
    \end{equation}Similarly,  we have that the following estimate:
    \begin{equation}\label{lemma-eq11}
    \begin{aligned}
\displaystyle J_{3}&\leq 4 \int_{Q}\tilde{\rho}^{2}|a'(y_{2})|^{2}|y_{2,x}|^{2}\,dxdt\\
    	  \noalign{\smallskip}\phantom{AA}
    	\displaystyle &\leq C\int_{Q}\tilde{\rho}^{2}|y_{2,x}|^{4}\,dxdt\\
        \noalign{\smallskip}\phantom{AA}
    	\displaystyle &\leq  C\left( \sup_{t\in [0,T]}\int_{I}\hat{\rho}^{2}_{1}|y_{2,x}|^{2}\,dx \right) \left( \int_{Q}\hat{\rho}^{2}_{1}|y_{2,xx}|^{2}\,dxdt \right) \\
        \noalign{\smallskip}\phantom{AA}
        &\leq C \Vert(y_1,y_2,p^1_1,p^1_2,p^2_1,p^2_2,h)\Vert^{4}_{Y}.
    \end{aligned}
    \end{equation}Moreover,  
    $$
    \begin{array}{l}
        \displaystyle J_{4}= 4\int_{Q}\tilde{\rho}^{2}|F_{2}(y_{2})-F'_{2}(0)y_{2}|^{2}|y_{2}|^{2}\,dxdt
        =4\int_{Q}\tilde{\rho}^{2}|F'_{2}(\lambda_{0}y_{2})-F'_{2}(0)|^{2}|y_{2}|^{2}\,dxdt,
    \end{array}
    $$
    for some  $\lambda_{0}\in (0,1)$. From  Sobolev embedding and \eqref{eq:compara_rhos3}, we deduce that
    \begin{equation}\label{lemma-eq12}
    \begin{aligned}
\displaystyle J_{4}&\leq C\int_{Q}\tilde{\rho}^{2}|y_{2}|^{4}\,dxdt\\
        \noalign{\smallskip}\phantom{AA}
        \displaystyle &\leq C \left(\sup_{t\in [0,T]}\int_{I}\hat{\rho}^{2}_{1}|y_{2,x}|^{2}\,dx \right) \left(\int_{Q}\rho^{2}_{0}|y_{2}|^{2}\,dxdt \right)\\
        \noalign{\smallskip}\phantom{AA}
        &\leq C  \Vert(y_1,y_2,p^1_1,p^1_2,p^2_1,p^2_2,h)\Vert^{4}_{Y}.
    \end{aligned}
    \end{equation}Thus, from \eqref{lemma-eq9},  \eqref{lemma-eq10}, \eqref{lemma-eq11} and \eqref{lemma-eq12}, we  have
    \begin{equation}\label{lemma-eq13}
        \begin{array}{l}
    	   \parallel A_0^2 (y_1,y_2,p^1_1,p^1_2,p^2_1,p^2_2,h) \parallel^2_{\mathcal{M}}\\[1.0em]
    		 \leq C  ||(y_1,y_2,p^1_1,p^1_2,p^2_1,p^2_2,h)||^{2}_{Y}(1+ ||(y_1,y_2,p^1_1,p^1_2,p^2_1,p^2_2,h)||^{2}_{Y}).
    	\end{array}
    \end{equation}
    
    \noindent As a third point, we have
    $$
    \begin{array}{l}
        \displaystyle 	\parallel A_1^1 (y_1,y_2,p^1_1,p^1_2,p^2_1,p^2_2,h) \parallel^2_{\mathcal{M}}\\[1.0em]
     \leq  3 \displaystyle\int_{Q}\tilde{\rho}^{2}|L^{\ast}_{1}p^{1}_{1}+cp^{1}_{2}-\alpha_{1}y_{1}\chi_{\omega_{1d}}|^{2}\,dxdt +3\displaystyle\int_{Q}\tilde{\rho}^{2}|a(y_{1})-a(0)|^{2}|p^{1}_{1,xx}|^{2}\,dxdt\\[1.0em]
    	\displaystyle +3\int_{Q}\tilde{\rho}^{2}|F'_{1}(y_{1})-F'_{1}(0)|^{2}|p^{1}_{1}|^{2}\,dxdt =K_{1}+K_{2}+K_{3}.
    \end{array}
    $$
    
     \noindent Now, we estimate the  terms  $K_{1}$,  $K_{2}$, and  $K_{3}$. Obviously, we have
    \begin{equation}\label{lemma-eq14}
        K_{1}\leq  C  ||(y_1,y_2,p^1_1,p^1_2,p^2_1,p^2_2,h)||^{2}_{Y}.
    \end{equation} On the other hand, using  \eqref{eq:compara_rhos3}, we  deduce  that
    \begin{equation}\label{lemma-eq15}
\begin{aligned}
\displaystyle 	K_{2}&=3\int_{Q}\tilde{\rho}^{2}|a(y_{1})-a(0)|^{2}|p^{1}_{1,xx}|^{2}\,dxdt\\
        \noalign{\smallskip}\phantom{AA}
      	&\displaystyle \leq C\int_{Q}\tilde{\rho}^{2}|y_{1}|^{2}|p^{1}_{1,xx}|^{2}\,dxdt\\
    	  \noalign{\smallskip}\phantom{AA}
    	&\displaystyle \leq C \left( \sup_{t\in [0,T]}\int_{I}\hat{\rho}^{2}_{1}|y_{1,x}|^{2}\,dx  \right)\left( \int_{Q}\hat{\rho}^{2}_{1}|p^{1}_{1,xx}|^{2}\,dxdt  \right)\\
    	  \noalign{\smallskip}\phantom{AA}
         & \leq C \Vert(y_1,y_2,p^1_1,p^1_2,p^2_1,p^2_2,h)\Vert^{4}_{Y}.
\end{aligned}
    \end{equation}Proceeding analogously, 
    we get
\begin{equation}\label{lemma-eq16}
\begin{aligned}
 \displaystyle 	K_{3}&=3\int_{Q}\tilde{\rho}^{2}|F'_{1}(y_{1})-F'_{1}(0)|^{2}|p^{1}_{1}|^{2}\,dxdt\\
    	  \noalign{\smallskip}\phantom{AA}
    	&\displaystyle \leq C\int_{Q}\tilde{\rho}^{2}|y_{1}|^{2}|p^{1}_{1}|^{2}\,dxdt\\
    	  \noalign{\smallskip}\phantom{AA}
    	&\displaystyle \leq C\left(  \sup_{t\in [0,T]}\int_{I}\hat{\rho}^{2}_{1}|y_{1,x}|^{2}\,dx  \right) \left(  \int_{Q}\rho^{2}_{0}|p^{1}_{1}|^{2}\,dxdt \right)\\
    	  \noalign{\smallskip}\phantom{AA}
         & \leq C \Vert(y_1,y_2,p^1_1,p^1_2,p^2_1,p^2_2,h)\Vert^{4}_{Y}.
\end{aligned}
\end{equation} From the estimates \eqref{lemma-eq14}, \eqref{lemma-eq15}  and  \eqref{lemma-eq16}, it follows that 
    \begin{equation}\label{lemma-eq17}
    	\begin{array}{l}
    		\parallel A_1^1 (y_1,y_2,p^1_1,p^1_2,p^2_1,p^2_2,h) \parallel^2_{\mathcal{M}}\\ [1.0em]
    		\leq C  \Vert(y_1,y_2,p^1_1,p^1_2,p^2_1,p^2_2,h)\Vert^{2}_{Y}(1+ \Vert(y_1,y_2,p^1_1,p^1_2,p^2_1,p^2_2,h)\Vert^{2}_{Y}).
    	\end{array}
    \end{equation} Analogously to the estimate \eqref{lemma-eq17}, we also have, for $i=1,2$, $j=1,2$,
    \begin{equation}\label{lemma-eq18}
    	\begin{array}{l}
    		\parallel A_i^j (y_1,y_2,p^1_1,p^1_2,p^2_1,p^2_2,h) \parallel^2_{\mathcal{M}}\\[1.0em]
    		
    		\leq C  ||(y_1,y_2,p^1_1,p^1_2,p^2_1,p^2_2,h)||^{2}_{Y}(1+ ||(y_1,y_2,p^1_1,p^1_2,p^2_1,p^2_2,h)||^{2}_{Y}).
    	\end{array}
    \end{equation}Furthermore,  from Theorem \ref{teo:linearized_control}, we obtain, for $i=3$, $j=1,2$,
    \begin{equation}\label{lemma-eq19}
        \begin{array}{l}
    	   \parallel A_i^j (y_1,y_2,p^1_1,p^1_2,p^2_1,p^2_2,h) \parallel^2_{H^{1}_{0}(I)}\\[1.0em]
    		
    		\leq C  ||(y_1,y_2,p^1_1,p^1_2,p^2_1,p^2_2,h)||^{2}_{Y}.
        \end{array}
    \end{equation}
    
    Finally, \eqref{lemma-eq8}, \eqref{lemma-eq13}, \eqref{lemma-eq17}, \eqref{lemma-eq18} and \eqref{lemma-eq19}, we  deduce that
    $$
        A(y_1,y_2,p^1_1,p^1_2,p^2_1,p^2_2,h)\in Z,\quad \text{ for all }\,\, (y_1,y_2,p^1_1,p^1_2,p^2_1,p^2_2,h)\in Y.
    $$
    
    \noindent Therefore, the mapping  $A$  is well defined. Similarly, we are able to show that  $A$ is continuous. This ends the proof of Lemma \ref{lema1}.
\end{proof}

\begin{lemma}\label{lema2}
    The mapping $A : Y \rightarrow Z$ is continuously differentiable.
\end{lemma}
\begin{proof}
    First, we prove that $A$ is Gâteaux derivative at any $(y_1,y_2,p_i^j,h) \in Y$ and compute the
    $\textit{G-derivative}$ $A^{\prime}(y_1,y_2,p_i^j,h)$.
    Consider the linear mapping $D A: Y \to Z$ given in components by
    $$
        D A(y_1,y_2,p_i^j,h) = (D A^1_0,D A^2_0, D A^1_1, D A^2_1, D A^1_2, D A^2_2, D A^1_3, D A^2_3),
    $$
    where, for $i=1,2$,
    \begin{equation*}
        \begin{cases}
		      D A^1_0(\bar y_1,\bar y_2,\bar p_i^j,\bar h) = &  \,
              \bar y_{1,t} - \frac{\partial}{\partial x}(a'(y_{1})\bar y_{1} y_{1,x} +a(y_1) \bar y_{1,x}) \\
            &+  D_1 F_1(y_1,y_2)\bar y_1 + D_2 F_1(y_1,y_2)\bar y_2 - \bar h \chi_{\omega} \\
            &+ \frac{1}{\mu_1} \rho_*^{-2} \bar p^1_1 \chi_{\omega_{1}} + \frac{1}{\mu_2} \rho_*^{-2} \bar p^2_1 \chi_{\omega_{2}}, \\
		      D A^2_0(\bar y_1,\bar y_2,\bar p_i^j,\bar h) =
            &  \, \bar y_{2,t} - (a'(y_{2})\bar y_{2}y_{2,x} + a(y_{2}) \bar y_{2,x})_x \\
            &+ D_1 F_2(y_1,y_2)\bar y_1 + D_2 F_2(y_1,y_2)\bar y_2, \\
            D A^1_i(\bar y_1,\bar y_2,\bar p_i^j,\bar h) =& - \bar p_{1,t}^i - \frac{\partial}{\partial x}\left(a'(y_{1})\bar y_{1} p^i_{1,x} + a(y_{1}) \bar p^i_{1,x}\right) \\
            &+ \left( a''(y_{1})\bar y_{1} y_{1,x} p^i_{1,x} + a'(y_{1})\bar y_{1}  p^i_{1,x} + a'(y_1)y_{1,x} \bar p^i_{1,x} \right) \\
            &+ D_{11}^2 F_1(y_1,y_2)\bar y_1 p^i_1 + D_{12}^2 F_1(y_1,y_2)\bar y_2 p^i_1 \\
            &+ D_{11}^2 F_2(y_1,y_2)\bar y_1 p^i_2 + D_{12}^2 F_2(y_1,y_2)\bar y_2 p^i_2 \\
            &+ D_1 F_1(y_1,y_2) \bar p^i_1 + D_1 F_2(y_1,y_2) \bar p^i_2 - \alpha_i \bar y_1 \chi_{\omega_{i,d}}, \\
		      D A^2_i(\bar y_1,\bar y_2,\bar p_i^j,\bar h) =& - \bar p_{2t}^i - \frac{\partial}{\partial x}\left(a'(y_{2})\bar y_{2} p^i_{2,x} + a(y_{2}) \bar p^i_{2,x}\right) \\
            &+ \left( a''(y_{2})\bar y_{2} y_{2,x} p^i_{2,x} +  a'(y_{2})\bar y_{2,x}  p^i_{2,x} + a'(y_2)y_{2,x} \bar p^i_{2,x} \right)\\
            &+ D_{21}^2 F_1(y_1,y_2)\bar y_1 p^i_1 + D_{22}^2 F_1(y_1,y_2)\bar y_2 p^i_1 \\
            &+ D_{21}^2 F_2(y_1,y_2)\bar y_1 p^i_2 + D_{22}^2 F_2(y_1,y_2)\bar y_2 p^i_2 \\
            &+ D_2 F_1(y_1,y_2) \bar p^i_1 + D_2 F_2(y_1,y_2) \bar p^i_2 
            - \alpha_i \bar y_2 \chi_{\omega_{i,d}}, \\
		      D A^1_3(\bar y_1,\bar y_2,\bar p_i^j,\bar h) =& \, \bar y_1(0),\\
		      D A^2_3(\bar y_1,\bar y_2,\bar p_i^j,\bar h) =& \, \bar y_2(0).
        \end{cases}
    \end{equation*}We have to show that, for $i=0,1,2,3$ and $j=1,2$, 
    $$
        \frac{1}{\lambda}\left[ A^j_i (( y_1, y_2, p_i^j, h)+\lambda(\bar y_1,\bar y_2,\bar p_i^j,\bar h)) - A^j_i ( y_1, y_2, p_i^j, h) \right] \to D A^j_i(\bar y_1,\bar y_2,\bar p_i^j,\bar h),
    $$
    strongly in the corresponding factor of $Z$ as $\lambda \to 0$.

    Initially, we have
    \begin{equation*}		
    	\begin{split}
    		&\left\| \frac{1}{\lambda}\left[ A^1_0 (( y_1, y_2, p_i^j, h)+\lambda(\bar y_1,\bar y_2,\bar p_i^j,\bar h)) - A^1_0( y_1, y_2, p_i^j, h) \right] - D A^1_0(\bar y_1,\bar y_2,\bar p_i^j,\bar h) \right\|^{2}_{\mathcal{M}} \\
    		&=\left\|
    		\frac{1}{\lambda}\Big[ y_{1,t}+\lambda \bar y_{1,t} - \left( a(y_1 + \lambda \bar y_1) (y_{1,x}+\lambda \bar y_{1,x}) \right)_x + F_1(y_1+\lambda \bar y_1, y_2+\lambda \bar y_2)- (h+\lambda \bar h) \chi_{\omega} \right.   \\
    		& \left. \left.    + \frac{1}{\mu_1} \rho_*^{-2} (p^1_1+\lambda \bar p^1_1) \chi_{\omega_1} + \frac{1}{\mu_2} \rho_*^{-2} (p^2_1+\lambda \bar p^2_1) \chi_{\omega_2} -y_{1,t} + \left( a(y_1) y_{1,x} \right)_x - F_1(y_1,y_2) + h \chi_{\omega}   \right. \right.\\
    		& \left. \left.    - \frac{1}{\mu_1} \rho_*^{-2} p^1_1 \chi_{\omega_1} - \frac{1}{\mu_2} \rho_*^{-2} p^2_1 \chi_{\omega_2} \right] -\bar y_{1,t} + (a'(y_{1})\bar y_{1} y_{1,x} +a(y_1) \bar y_{1,x})_x - D_1 F_1(y_1,y_2)\bar y_1 \right.\\
    		&  \left.  - D_2 F_1(y_1,y_2)\bar y_2 + \bar h \chi_{\omega} \right.  \left. -\frac{1}{\mu_1} \rho_*^{-2} \bar p^1_1 \chi_{\omega_{1}} -\frac{1}{\mu_2} \rho_*^{-2} \bar p^2_1 \chi_{\omega_{2}} \right\|^{2}_{\mathcal{M}} \\
    		& = \int_Q \tilde\rho^2 \left| \frac{1}{\lambda}\left( - \left( a(y_1 + \lambda \bar y_1) (y_{1,x}+\lambda \bar y_{1,x}) \right)_x + \left( a(y_1)y_{1,x} \right)_x \right) + (a'(y_{1})\bar y_{1} y_{1,x} +a(y_1) \bar y_{1,x})_x \right|^2 \\
            & + \int_Q \tilde\rho^2 \left| \frac{1}{\lambda}( F_1(y_1+\lambda\bar y_1,y_2 + \lambda\bar y_2) - F_1(y_1,y_2)) - D_1 F_1(y_1,y_2)\bar y_1 - D_2 F_1(y_1,y_2)\bar y_2 \right|^2 \\
            &= J_1+J_2.
    	\end{split}
    \end{equation*}From the Mean Value Theorem, there exists $\theta \in [0,1]$, such that
    \begin{equation*}
        \begin{split}
            J_1 \leq &\int_Q \tilde\rho^2 \left( \left| a''(y_1 + \theta \lambda \bar y_1) - a''(y_1) \right|^2 |\bar y_1 y_{1,x} |^2 - \left| a'(y_1 + \lambda \bar y_1)-a'(y_1) \right|^2 \left| 2 y_{1,x} \bar y_{1,x} \right|^2 \right.\\
            &\left.  + \left| a'(y_1+\lambda\bar y_1)-a'(y_1) \right|^2 |\bar y_1 y_{1,xx}|^2 + |a(y_1+\lambda\bar y_1)-a(y_1)|^2 |\bar y_{1,xx}|^2 \right) 
        \end{split}
    \end{equation*}Since $a \in C^3(I)$,  from estimate abive and Lebesgue's dominated convergence theorem, we get that $J_1$ converges to zero as $\lambda \to 0$.
    
    \noindent Once again, from the Mean Value Theorem, there exists 
    $\theta \in [0,1]$, such that
    \begin{eqnarray*}
    	J_2&=&\int_Q \tilde\rho^2 \left| \frac{1}{\lambda} \nabla F_1(y_1+\theta\lambda\bar y_1,y_2+\theta\lambda\bar y_2)\cdot (\lambda \bar y_1,\lambda \bar y_2) - D_1 F_1(y_1,y_2)\bar y_1 - D_2 F_1(y_1,y_2)\bar y_2 \right|^2\\
        &=& \int_Q \tilde\rho^2 \left| \nabla F_1(y_1+\theta\lambda\bar y_1,y_2+\theta\lambda\bar y_2)\cdot (\bar y_1,\bar y_2) - \nabla F_1(y_1,y_2) \cdot (\bar y_1,\bar y_2) \right|^2\\
        &=& \int_Q \tilde\rho^2 \left| \nabla F_1(y_1+\theta\lambda\bar y_1,y_2+\theta\lambda\bar y_2) - \nabla F_1(y_1,y_2) \right|^2 (\bar y_1^2 + \bar y_2^2 ).
    \end{eqnarray*}
    Therefore, from \eqref{condiciones_a4f}, we obtain that $J_2$ converges to zero as $\lambda \to 0$. 
    
\noindent On the other hand, for $D A^2_0(\bar y_1,\bar y_2,\bar p_i^j,\bar h)$, we get
    \begin{equation*}		
    	\begin{split}
    		&\left\| \frac{1}{\lambda}\left[ A^2_0 (( y_1, y_2, p_i^j, h)+\lambda(\bar y_1,\bar y_2,\bar p_i^j,\bar h)) - A^2_0( y_1, y_2, p_i^j, h) \right] - D A^2_0(\bar y_1,\bar y_2,\bar p_i^j,\bar h) \right\|^{2}_{\mathcal{M}} \\
    		& = \int_Q \tilde\rho^2 \left| \frac{1}{\lambda}\left( - \left( a(y_2 + \lambda \bar y_2) (y_{2,x}+\lambda \bar y_{2,x}) \right)_x + \left( a(y_2)y_{2,x} \right)_x \right) + (a'(y_{2})\bar y_{2} y_{2,x} +a(y_2) \bar y_{2,x})_x \right|^2 \\ 
            &\quad + \int_Q \tilde\rho^2 \left| \frac{1}{\lambda}(F_2(y_1+\lambda\bar y_1,y_2+\lambda\bar y_2)-F_2(y_1,y_2)) - D_1 F_2(y_1,y_2)\bar y_1 - D_2 F_2(y_1,y_2)\bar y_2 \right|^2 \\
            & = J_3+J_4.
    	\end{split}
    \end{equation*}As above, from the Mean Value Theorem and Lebesgue's dominated convergence theorem, it follows that
    \begin{equation*}
        \begin{split}
            J_3 \leq &\int_Q \tilde\rho^2 \left( \left| a''(y_2 + \theta \lambda \bar y_2) - a''(y_2) \right|^2 |\bar y_2 y_{2,x} |^2 - \left| a'(y_2 + \lambda \bar y_2)-a'(y_2) \right|^2 \left| 2 y_{2,x} \bar y_{2,x} \right|^2 \right.\\
            &\left. \qquad + \left| a'(y_2 + \lambda\bar y_2) - a'(y_2) \right|^2 |\bar y_2 y_{2,xx}|^2 + |a(y_2+\lambda\bar y_2)-a(y_2)|^2 |\bar y_{2,xx}|^2 \right) \to 0,
        \end{split}
    \end{equation*} as $\lambda \to 0$ and some  constant $\theta \in [0,1]$ 
    
   \noindent Analogously, 
    $$
        J_4 = \int_Q \tilde\rho^2 \left| \nabla F_1(y_1+\theta\lambda\bar y_1,y_2+\theta\lambda\bar y_2) - \nabla F_1(y_1,y_2) \right|^2 (\bar y_1^2 + \bar y_2^2 ) \to 0,
    $$ as $\lambda \to 0$.
    
    \noindent Moreover, for $A^1_i$, $i=1,2$, we get
    \begin{equation}
    	\begin{split}
    		&\left\| \frac{1}{\lambda}\left[ A^1_i (( y_1, y_2, p_i^j, h)+\lambda(\bar y_1,\bar y_2,\bar p_i^j,\bar h)) - A^1_i( y_1, y_2, p_i^j, h) \right] - D A^1_i(\bar y_1,\bar y_2,\bar p_i^j,\bar h) \right\|^{2}_{\mathcal{M}} \\
            &= \left\| \frac{1}{\lambda} \left[ 
            - \left(a(y_1+\lambda\bar y_1)(p_{1,x}^1 + \lambda \bar p_{1,x}^1) + a'(y_1+\lambda\bar y_1)(y_{1,x}+\lambda\bar y_{1,x})(p_{1,x}^1+\lambda\bar p_{1,x}^1) \right.\right.\right. \\
            &\left.\left.\left. \qquad + a(y_1)p_{1,x}^1 - a'(y_1)y_{1,x}p_{1,x}^1 \right)_x \right] + \left(a'(y_{1})\bar y_{1} p^i_{1,x} + a(y_{1}) \bar p^i_{1,x}\right)_x \right.\\
            &\qquad \left. - \left( a''(y_{1})\bar y_{1} y_{1,x} p^i_{1,x} + a'(y_{1})\bar y_{1}  p^i_{1,x} + a'(y_1)y_{1,x} \bar p^i_{1,x} \right)
            \right\|_{L^2(\tilde\rho^2,Q)}^2 \\
            &\quad + \left\| \frac{1}{\lambda} \left[ D_1 F_1(y_1+\lambda \bar y_1,y_2+\lambda \bar y_2) (p^i_1 +\lambda \bar p^i_1) + D_1 F_2(y_1+\lambda \bar y_1,y_2+\lambda \bar y_2) (p^i_2 +\lambda \bar p^i_2) \right.\right. \\
            &\left.\left.\qquad - D_1 F_1(y_1,y_2) p^i_1 - D_1 F_2(y_1,y_2) p^i_2 \right] - D_{11}^2 F_1(y_1,y_2)\bar y_1 p^i_1 - D_{12}^2 F_1(y_1,y_2)\bar y_2 p^i_1 \right. \\
            &\left.\qquad - D_{11}^2 F_2(y_1,y_2)\bar y_1 p^i_2 - D_{12}^2 F_2(y_1,y_2)\bar y_2 p^i_2 
            - D_1 F_1(y_1,y_2) \bar p^i_1 - D_1 F_2(y_1,y_2) \bar p^i_2
            \right\|^{2}_{\mathcal{M}} \\
            &= J_5+J_6.
    	\end{split}
    \end{equation}We observe that  $J_5$  is formed by two terms: $J_{51}$ and $J_{51}$, satisfying
    \begin{equation*}
        \begin{split}
            J_{51} \leq 2\left\| \left( \frac{1}{\lambda} \left[ a(y_1+\lambda\bar y_1)(p^1_{1,x}+\lambda\bar p^i_{1,x} ) - a(y_1)p^i_{1,x} \right] - a'(y_1)\bar y_1 p^i_{1,x} + a(y_1)\bar p^i_{1,x} \right)_x\right\|^{2}_{\mathcal{M}}
        \end{split}
    \end{equation*}
    and
    \begin{equation*}
        \begin{split}
            J_{52} \leq &2\left\| \left( \frac{1}{\lambda} \left[ a'(y_1+\lambda\bar y_1)(y_{1,x}-\lambda\bar y_{1,x})(p^1_{1,x}+\lambda\bar p^i_{1,x} ) - a'(y_1)y_{1,x}p^i_{1,x} \right] \right.\right. \\
            &\left.\left. \qquad - (a''(y_1)\bar y_1 y_{1,x} p^1_{1,x} + a'(y_1)\bar y_{1,x} p^1_{1,x} + a'(y_1)y_{1,x}\bar p^i_{1,x}) \right)_x\right\|^{2}_{\mathcal{M}}.
        \end{split}
    \end{equation*}From the above estimates, calculations similar to the previous ones allow us to obtain that $J_{51}$ and $J_{52}$ converge to zero as $\lambda \to 0$.

    \noindent Since 
    \begin{equation*}
    J_6 \leq     J_{61} + J_{62} + J_{63},
    \end{equation*}where
    \begin{equation*}
    \begin{cases}
    J_{61} &= \displaystyle \int_Q \tilde\rho^2 \Big| \frac{1}{\lambda}\left[D_1 F_1(y_1+\lambda \bar y_1,y_2+\lambda \bar y_2) p^i_1 - D_1 F_1(y_1,y_2) p^i_1 \right] - \\
   & \quad D_{11}^2 F_1(y_1,y_2)\bar y_1 p^i_1 - D_{12}^2 F_1(y_1,y_2)\bar y_2 p^i_1\Big|^2, \\
   J_{62} &= \displaystyle \int_Q \tilde\rho^2 \Big| \frac{1}{\lambda}\left[ D_1 F_2(y_1+\lambda \bar y_1,y_2+\lambda \bar y_2) p^i_2  - D_1 F_2(y_1,y_2) p^i_2 \right] \\
   & \quad - D_{11}^2 F_2(y_1,y_2)\bar y_1 p^i_2 - D_{12}^2 F_2(y_1,y_2)\bar y_2 p^i_2 \Big|^2,\\
   J_{63} &= \displaystyle \int_Q \tilde\rho^2 \Big| D_1 F_1(y_1+\lambda \bar y_1,y_2+\lambda \bar y_2) \bar p^i_1 + D_1 F_2(y_1+\lambda \bar y_1,y_2+\lambda \bar y_2) \bar p^i_2 \\
   & \quad D_1 F_1(y_1,y_2) \bar p^i_1 - D_1 F_2(y_1,y_2) \bar p^i_2 \Big|^2,
    \end{cases}
    \end{equation*}
    
    \noindent from the Mean Value Theorem, for $\theta \in [0,1]$,
    \begin{equation*}
        \begin{split}
            J_{61} &\leq \int_Q \tilde\rho^2 \left| \frac{1}{\lambda}
            \nabla (D_1 F_1) (y_1+ \theta\lambda \bar y_1,y_2+\theta\lambda \bar y_2)\cdot(\lambda \bar y_1,\lambda \bar y_2) p^i_1 - \nabla (D_1 F_1)(y_1,y_2)\cdot(\bar y_1,\bar y_2)p^i_1
            \right|^2,
        \end{split}
    \end{equation*}
    and the fact that $D_1 F \in C^1$, we get that $J_{31}$ converges to zero as $\lambda \to 0$. 
    
   \noindent  From similar computations, we find  that $J_{62}$ and $F_{63}$ converges to zero as $\lambda \to 0$
    
    For $A^2_i$, $i=1,2$, a similar estimation gives
    $$
        \left\| \frac{1}{\lambda}\left[ A^1_i (( y_1, y_2, p_i^j, h)+\lambda(\bar y_1,\bar y_2,\bar p_i^j,\bar h)) - A^1_i( y_1, y_2, p_i^j, h) \right] - D A^1_i(\bar y_1,\bar y_2,\bar p_i^j,\bar h) \right\|^{2}_{\mathcal{M}} \to 0,
    $$
    as $\lambda \to 0$. This finishes the proof that $A$ is Gateaux differentiable, with a \textit{G-derivative} given by $D A( y_1, y_2, p_i^j, h)$.
    
     Now, take $( y_1, y_2, p_i^j, h)\in Y$ and let $\{(y_{1,n},y_{2,n},p^{j}_{i,n},h_{n})\}_{n=0}^{\infty}$ be a sequence that converges to  $( y_1, y_2, p_i^j, h)$ in $Y$. For each $(\bar y_1,\bar y_2,\bar p_i^j,\bar h)\in B_{r}(0)$, we have to prove that 
    \begin{equation}\label{eq:continuity_DA_m_l}
    	\begin{split}
    		&\left\|(D A^l_m(y_{1,n},y_{2,n},p^{j}_{i,n},h_{n})-D A^l_m( y_1, y_2, p_i^j, h))(\bar y_1,\bar y_2,\bar p_i^j,\bar h) \right\|^{2}_{\mathcal{M}}\rightarrow 0,
    	\end{split}
    \end{equation}
    as $\lambda \to 0$, for $m=0,1,2,3$ and $l=1,2$.
    
    We present the details for $m=0$ and $l=1$. The other terms are estimated analogously, and we omit them. We have
    \begin{equation*}
        \begin{split}
    	   &(D A^1_0(y_{1,n},y_{2,n},p^{j}_{i,n},h_{n})-D A^1_0( y_1, y_2, p_i^j, h))(\bar y_1,\bar y_2,\bar p_i^j,\bar h)\\
           &= \left(a'(y_{1,n})\bar y_{1} y_{1,nx} - a'(y_{1})\bar y_{1} y_{1,x} + a(y_{1,n}) \bar y_{1,x} - a(y_1) \bar y_{1,x} \right)_x \\
           &\quad + [ (D_1 F_1(y_{1,n},y_{2,n})-D_1 F_1(y_1,y_2))\bar y_1 + (D_2 F_1(y_{1,n},y_{2,n})-D_2 F_1(y_1,y_2))\bar y_2] \\
           &= ((a'(y_{1,n})-a'(y_{1}))\bar y_{1} y_{1,nx})_x + (a'(y_{1,n})\bar y_{1} (y_{1,nx}-y_{1,x}))_x + ((a(y_{1,n})-a(y_{1})) \bar y_{1,x})_x \\&\quad + [ (D_1 F_1(y_{1,n},y_{2,n})-D_1 F_1(y_1,y_2))\bar y_1 + (D_2 F_1(y_{1,n},y_{2,n})-D_2 F_1(y_1,y_2))\bar y_2] \\
           &= X_0^{11}+X_0^{12}+X_0^{13}+X_0^{14}.
        \end{split}
    \end{equation*}First, 
    \begin{eqnarray*}
        \int_{Q} \tilde\rho^{2}|X_{0}^{11}|^{2} dxdt
    	&\leq&C\int_{Q} \tilde\rho^{2} M^2 \left( \left|(y_{1,n}-y_1)_x \bar y_1 y_{1,nx} \right|^2 + \left|(y_{1,n}-y_1) \bar y_{1,x} y_{1,nx}  \right|^2 \right. \\
        &&\left. \qquad\quad +\left|(y_{1,n}-y_1) \bar y_1 y_{1,nxx} \right|^2 \right) dxdt.
    \end{eqnarray*}
    Using embedding $H^1_0(I) \hookrightarrow L^\infty(I)$ and 
    \eqref{eq:compara_rhos3}, we get
    \begin{equation*}
        \begin{split}
            &\int_{Q} \tilde\rho^{2} \left|(y_{1,n}-y_1)_x \bar y_1 y_{1,nx} \right|^2 dx dy \\
            &\leq C\int_0^T \tilde\rho^{2} \|y_{1,nx}-y_{1,x}\|_{L^\infty(I)}^2 \|\bar y_1\|_{L^\infty(I)}^2 \|y_{1,nx}\|_{L^2(I)}^2 dt \\
            &\leq C \int_0^T \tilde\rho^{2} \hat\rho_1^{-6} \hat\rho_1^2 \|y_{1,nxx}-y_{1,xx}\|_{L^2(I)}^2 \hat\rho_1^2 \|\bar y_{1,x}\|_{L^2(I)}^2 \hat\rho_1^2 \|y_{1,nx}\|_{L^2(I)}^2 dt \\
            &\leq C \sup_{t\in[0,T]} \hat\rho_1^2 \|\bar y_{1,x}\|_{L^2(I)}^2 \sup_{t\in[0,T]} \hat\rho_1^2 \|y_{1,nx}\|_{L^2(I)}^2  \int_0^T \hat\rho_1^2 \|y_{1,nxx}-y_{1,xx}\|_{L^2(I)}^2 dt \\
            &\leq C \|(\bar y_1,\bar y_2,\bar p_i^j,\bar h)\|_{Y} \|(y_1,y_2,p_i^j,h)\|_{{Y}} \| (y_{1,n},y_{2,n},p_{i,n}^j,h_n) - (y_1,y_2,p_i^j,h) \|_{{Y}}.
        \end{split}
    \end{equation*}Proceeding similarly with the other terms of $X_0^{11}$, we get
    \begin{equation*}
        \begin{split}
                &\int_{Q} \tilde\rho^{2}|X_{0}^{11}|^{2} dxdt \\
                &\leq C \|(\bar y_1,\bar y_2,\bar p_i^j,\bar h)\|_{{Y}} \|(y_1,y_2,p_i^j,h)\|_{{Y}} \| (y_{1,n},y_{2,n},p_{i,n}^j,h_n) - (y_1,y_2,p_i^j,h) \|_{{Y}}.
        \end{split}
    \end{equation*}Analogously, for $X_0^{12}$ and $X_0^{13}$, 
    \begin{equation*}
        \begin{split}
            &\int_{Q} \tilde\rho^{2} (|X_{0}^{12}|^{2} + |X_{0}^{13}|^{2} )dxdt\\
            &\leq C\int_{Q} \tilde\rho^{2} M^2 \left( |(\bar y_1 (y_{1,nx}-y_{1,x}))_x|^2 + |((y_{1,n}-y_1) \bar y_{1,x})_x|^{2} \right) dxdt \\
            &\leq C\|(\bar y_1,\bar y_2,\bar p_i^j,\bar h)\|_{{Y}} \| (y_{1,n},y_{2,n},p_{i,n}^j,h_n) - (y_1,y_2,p_i^j,h) \|_{{Y}}.
        \end{split}
    \end{equation*}
    Therefore,
    \begin{equation*}
        \int_{Q} \tilde\rho^{2} (|X_{0}^{11}|^{2} + |X_{0}^{12}|^{2} + |X_{0}^{13}|^{2} )dxdt \to 0, \qquad\text{as}\quad n\to \infty.
    \end{equation*}
    On the other hand,
    \begin{equation*}
        \begin{split}
        	\int_{Q} \tilde\rho^{2} |X_{0}^{14}|^{2} dxdt&\leq  \int_{Q} \tilde\rho^{2} |(D_1 F_1(y_{1,n},y_{2,n})-D_1 F_1(y_1,y_2))|^2 |\bar y_1|^{2}dxdt\\
            & + \int_{Q} \tilde\rho^{2} |(D_2 F_1(y_{1,n},y_{2,n})-D_2 F_1(y_1,y_2))|^2 |\bar y_2|^{2}dxdt\\
        	&\leq C\int_{Q} \tilde\rho^{2} M^2 ( |y_{1,n}-y_1|^2 + |y_{2,n}-y_2|^2 ) (|\bar y_1|^{2} + |\bar y_2|^{2} ) dxdt.
        \end{split}
    \end{equation*}
    
    \noindent From  \eqref{eq:compara_rhos3}, we have  that  $\tilde\rho^2\leq C \rho_1^2 \rho^2_1\leq C \hat{\rho}^2 \rho_0^2 $. Then
    \begin{equation*}
    \begin{split}
    	&\int_{Q} \tilde\rho^{2}|X_{0}^{14}|^{2} dxdt\\
        &\leq C \int_{Q} \hat{\rho}_1^2 \rho_0^2  ( |y_{1,n}-y_1|^2 + |y_{2,n}-y_2|^2 ) (|\bar y_1|^{2} + |\bar y_2|^{2} )dxdt\\
    	&\leq C \int_0^T \left(\hat{\rho}_1^2 (\| \bar{y}_1\|^2_{L^\infty(I)} + \| \bar{y}_2\|^2_{L^\infty(I)})\right)\int_{0}^{1} \rho_0^2  (|y_{1,n}-y_1|^2 + |y_{2,n}-y_2|^2 ) dt\\
        &\leq C \int_0^T \sup_{t\in [0,T]}\left(\hat{\rho}_1^2 (\| \bar{y}_{1,x}\|^2_{L^2(I)} + \| \bar{y}_{2,x}\|^2_{L^2(I)})\right) \left( \|\rho_0 (y_{1,n}-y_1)\|_{L^2(I)} + \|\rho_0 (y_{2,n}-y_2)\|_{L^2(I)} \right) dt\\
    	&\leq C\|(\bar y_1,\bar y_2,\bar p_i^j,\bar h)\|_{{Y}} \cdot \left( \|\rho_0 (y_{1,n}-y_1)\|_{L^2(Q)} + \|\rho_0 (y_{2,n}-y_2)\|_{L^2(Q)} \right)\\
    	&\leq C\|(y_{1,n}-y_1),(y_{2,n}-y_2),(p^{j}_{i,n}-p^{j}_i),(h_{n}-h)\|_{{Y}}\rightarrow 0 \quad\text{as}\quad n \to \infty.
    \end{split}
    \end{equation*}
    This shows that \eqref{eq:continuity_DA_m_l} is satisfied. 
\end{proof}

 \begin{lemma}\label{lema3}
 Let $A$ be the mapping in \eqref{aplicação A}. Then, $A^{\prime}(0,0,0,0,0,0,0)$ is onto.
 \end{lemma}
\begin{proof}
    Let $(F_0, \overline{F}_0,F_1, \overline{F}_1,F_2, \overline{F}_2,y^{0}_1, y^{0}_2)\in {Z}$. From Theorem \ref{teo:linearized_control} we know that there exist functions $y_1,y_2,p^{1}_1,p^{2}_1,p^{1}_2,p^{2}_2$ satisfying \eqref{lin1} and the estimates \eqref{est-linear1}, \eqref{est-linear2} and \eqref{est-linear3}. 
    Consequently, $(y,p^{1},p^{2},h)\in {Y}$ and $$A^{\prime}(0,0,0,0,0,0,0)(y_1,y_2,p^{1}_1,p^{2}_1,p^{1}_2,p^{2}_2,h)=(F_0, \overline{F}_0,F_1, \overline{F}_1,F_2, \overline{F}_2,y^{0}_1,y^0_2).$$ 
    \hfill

    This ends the proof.
\end{proof}

\subsection{Proof of main theorem}

In accordance with Lemmas \eqref{lema1}, \eqref{lema2} e \eqref{lema3}, we can apply Liusternik’s Theorem and deduce that there exists $\epsilon > 0$, a mapping $W : B_\epsilon(0) \subset Z \rightarrow Y$ such that 

\begin{itemize}
    \item $W (G_1, G_2, G_3, G_4, G_5, G_6,f,g) \in B_r(0)$, 
    \item $A(W (G_1, G_2, G_3, G_4, G_5, G_6,f,g)) = (G_1, G_2,     G_3, G_4, G_5, G_6,f,g)$ \\
        $\forall (G_1, G_2, G_3, G_4, G_5, G_6,f,g) \in B_\epsilon(0)$.
\end{itemize}


\noindent Taking $(0, 0, -\alpha_1y^1_{1,d}\chi_{\omega_{1,d}},-\alpha_2y^2_{1,d}\chi_{\omega_{2,d}},-\alpha_1y^1_{2,d}\chi_{\omega_{1,d}},-\alpha_2y^2_{2,d}\chi_{\omega_{2,d}}, f,g) \in B_\epsilon(0)$, and
$$
    (y_1,y_2,p^1_1,p^1_2,p^2_1,p^2_2,h) = W (0, 0, -\alpha_1y^1_{1,d}\chi_{\omega_{1,d}},-\alpha_2y^2_{1,d}\chi_{\omega_{2,d}},-\alpha_1y^1_{2,d}\chi_{\omega_{1,d}},-\alpha_2y^2_{2,d}\chi_{\omega_{2,d}}, f,g) \in Y,
$$
we have that 
$$
    A((y_1,y_2,p^1_1,p^1_2,p^2_1,p^2_2,h)) = W (0, 0, -\alpha_1y^1_{1,d}\chi_{\omega_{1,d}},-\alpha_2y^2_{1,d}\chi_{\omega_{2,d}},-\alpha_1y^1_{2,d}\chi_{\omega_{1,d}},-\alpha_2y^2_{2,d}\chi_{\omega_{2,d}}, f,g).
$$

Thus, we proved that \eqref{to1} is null locally controllable at time $T > 0$.

\hfill\qed

\noindent\textbf{Acknowledgments}

\noindent Cristian Loli was  supported by  CAPES - grant number 88887.836417/2023-00.  George J. Bautista was  supported by FAPERJ under the program PDS-2024, grant number SEI-0260003/019497/2024. Juan Límaco was partially supported by CNPq grant number 310860/2023-7.

\bibliographystyle{abbrv}
\bibliography{bib}	

\end{document}